\documentclass[a4paper]{article}

\usepackage[english]{babel}
\usepackage[utf8x]{inputenc}
\usepackage[T1]{fontenc}
\usepackage{graphicx}
\usepackage{caption, subcaption}
\usepackage[title]{appendix}
\usepackage{tikz}
\usepackage{pgfplots}
\usepackage{cutwin}
\graphicspath{ {MATLAB/} }
\usetikzlibrary{intersections,decorations.pathreplacing,decorations.markings,calc,angles,quotes,arrows.meta,pgfplots.fillbetween,patterns}
\usepackage[a4paper,top=2.9cm,bottom=2cm,left=2.5cm,right=2.5cm,marginparwidth=1.75cm]{geometry}

\usepackage{amsmath,yhmath}
\usepackage{amsthm}
\usepackage{amssymb}
\usepackage[hidelinks]{hyperref}
\usepackage{cleveref}
\usepackage{epstopdf}
\usepackage{comment}
\usepackage{todonotes}
\usepackage{color}
\usepackage{float}
\usepackage[sort&compress,numbers]{natbib}
\newtheorem{thm}{Theorem}% 	 	[section]
\newtheorem{lem}[thm]{Lemma}%[section]
\newtheorem{prop}[thm]{Proposition}%[section]
\newtheorem{corollary}[thm]{Corollary}
\newtheorem{assump}[thm]{Assumption}%[section]
\theoremstyle{definition}
\newtheorem{rmk}[thm]{Remark}%[section]
\numberwithin{equation}{section}
\numberwithin{thm}{section}

\newcommand{\A}{\mathcal{A}}
\newcommand{\B}{\mathcal{B}}
\newcommand{\C}{\mathcal{C}}
\newcommand{\D}{\mathcal{D}}
\newcommand{\F}{\mathcal{F}}
\newcommand{\Hc}{\mathcal{H}}
\newcommand{\K}{\mathcal{K}}
\renewcommand{\L}{\mathcal{L}}

\newcommand{\N}{\mathbb{N}}

\renewcommand{\P}{\mathcal{P}}
\newcommand{\R}{\mathbb{R}}

\renewcommand{\S}{\mathcal{S}}
\newcommand{\T}{\mathcal{T}}
\newcommand{\V}{\mathcal{V}}
\newcommand{\Z}{\mathbb{Z}}
\newcommand{\p}{\partial}
\renewcommand{\epsilon}{\varepsilon}
\newcommand{\dx}{\: \mathrm{d}}

\newcommand{\ie}{\textit{i.e.}}
\newcommand{\nm}{\noalign{\smallskip}}
\newcommand{\ds}{\displaystyle}
\newcommand{\iu}{\mathrm{i}\mkern1mu}
\renewcommand{\Re}{\operatorname{Re}}

\makeatletter
\newcommand{\neutralize}[1]{\expandafter\let\csname c@#1\endcsname\count@}
\makeatother

\title{Robust edge modes in dislocated systems of subwavelength resonators}

\author{Habib Ammari\thanks{\footnotesize Department of Mathematics, ETH Z\"urich, R\"amistrasse 101, CH-8092 Z\"urich, Switzerland (habib.ammari@math.ethz.ch).}\and Bryn Davies\thanks{\footnotesize Department of Mathematics, Imperial College London, 180 Queen's Gate, London SW7 2AZ, United Kingdom (bryn.davies@imperial.ac.uk).}  \and Erik Orvehed Hiltunen\thanks{\footnotesize Department of Mathematics, Yale University, 51 Prospect Street, New Haven CT 06511, USA  (erik.hiltunen@yale.edu).}.}

\date{}

\begin{document}
	\maketitle

	\begin{abstract}
		Robustly manipulating waves on subwavelength scales can be achieved by, firstly, designing a structure with a subwavelength band gap and, secondly, introducing a defect so that eigenfrequencies fall within the band gap. Such frequencies are well known to correspond to localized modes. We study a one-dimensional array of subwavelength resonators, prove that there is a subwavelength band gap, and show that by introducing a dislocation we can place localized modes at any point within the band gap. We complement this analysis by studying the stability properties of the corresponding finite array of resonators, demonstrating the value of being able to customize the position of eigenvalues within the band gap.
	\end{abstract}
\vspace{0.5cm}
	\noindent{\textbf{Mathematics Subject Classification (MSC2000):} 35J05, 35C20, 35P20.

\vspace{0.2cm}

\vspace{0.5cm}
%	\def\keywords2{\vspace{.5em}{\textbf{  Mathematics Subject Classification
%				(MSC2000).}~\,\relax}}
%	\def\endkeywords2{\par}
%	\keywords2{35R30, 35C20.}
%
%	\def\keywords{\vspace{.5em}{\textbf{ Keywords.}~\,\relax}}
%	\def\endkeywords{\par}
%	\keywords{subwavelength resonance, high contrast, subwavelength phononic crystal, topological insulator, edge mode.}

	\section{Introduction}
Recent breakthroughs in the field of wave manipulation have led to the creation of structures that can guide, localize and trap waves at \emph{subwavelength} scales (\emph{i.e.} at spatial scales that are significantly smaller than the operating wavelength) \cite{defectSIAM, linedefect, shvets2007guiding, experiment2013, superfocusing, brynCochlea, Lemoult_sodacan, smith2004metamaterials, kaina2015negative, phononic1, phononic2, ma2016acoustic}. The building blocks of these structures are  subwavelength resonators: objects exhibiting resonant phenomena in response to wavelengths much greater than their size. Examples include plasmonic particles, Minnaert bubbles and high-index dielectric particles.  The highly contrasting material parameters (relative to the background medium) of these objects are the crucial mechanism responsible for their subwavelength response (see \cite{first}, also \cite{ammari2021functional} for a general review). The goal for researchers, now, is to develop \emph{robust} versions of these designs, that retain their wave-manipulation properties even in the presence of structural imperfections \cite{ammari2019topological, Yves2, khanikaev2013photonic, khanikaev2017two, Yves1}.

%Manipulating waves on subwavelength scales is of fundamental importance in many areas of physics and engineering. In applications, it is often of interest to localize, trap or guide a wave in a small region of space. Recently, developments have been made to do this on a \textit{subwavelength} scale, meaning that the spatial dimensions are significantly smaller than the operating wavelength \cite{defectSIAM,linedefect}.

An approach to creating materials with low-frequency localized modes is to start with an array of subwavelength resonators that exhibits a \emph{subwavelength band gap}, that is, a range of frequencies within the subwavelength regime that cannot propagate through the material. We then introduce a \textit{defect} to the structure. If done correctly, this perturbation creates subwavelength resonant frequencies that are inside the band gap and correspond to resonant modes whose amplitude decays exponentially away from the defect \cite{defectSIAM, linedefect, Lemoult_sodacan, moghaddam2019slow, cha2018experimental}. We will refer to these resonant frequencies as \emph{mid-gap frequencies} and the associated modes as \emph{localized modes}.

It is widely understood that both the rate at which the localized mode decays and the stability of the mid-gap frequency depend on the location of the frequency within the band gap \cite{combes1973asymptotic, kuchment2}. Typically, the localization is stronger if the frequency is closer to the middle of the band gap. Moreover, eigenvalues in the middle of the band gap are more robust to imperfections of the material, particularly since a small perturbation is likely to keep the eigenvalue inside the band gap. With this in mind, our aim is to introduce defects in such a way that we are able to place a mid-gap frequency at any given point in the subwavelength band gap, enabling controllable and robust wave guiding at subwavelength scales.

%. If the bulk of the material is kept intact by the defect, such modes will be exponentially decaying away from the band gap \cite{insert citation here!}.

In this work, we will begin with a one-dimensional array of pairs of subwavelength resonators which, we prove, exhibits a band gap within the subwavelength regime. We will then introduce a defect by adding a dislocation within one of the resonator pairs (see \Cref{fig:infinite}). We will see that, as a result of this dislocation, mid-gap frequencies enter the band gap from either side and converge to a single frequency, within the band gap, as the dislocation becomes arbitrarily large (see \Cref{fig:gapcrossing}).

The localized modes studied in this work are, in particular, \emph{edge modes}. Localized modes are known as edge modes when the defect responsible for their existence is the interface between two materials with different \emph{bulk indices}. Edge modes will propagate along the interface without entering the bulk of the material. The bulk index of a material is a topological quantity associated with a periodic structure and  it is well known that the interface of two materials with different indices supports robust edge modes
%
%the principle that the interface of materials with different bulk indexes exhibits edge modes is known as the \emph{bulk-boundary correspondence}
\cite{zak, zak_experiment, ammari2019topological, top_review, yang2015topological, bulkbdy, shapiro2019strongly, haldane2008possible, halperin1982quantized, hatsugai1993chern}. A typical example of an edge mode is that occurring at the interface between a material with non-zero bulk index and free space (corresponding to the fact that free space has a bulk index of zero). It is in this sense that the two localized modes studied here are edge modes, since it was proved in \cite{ammari2019topological} that the corresponding array of resonator pairs has non-zero bulk index.

%is to superimpose an ``edge potential'' which interpolates between two distinct asymptotic periodic structures \cite{fefferman2016edge}.

There are a plethora of different ways to introduce an interface capable of supporting edge modes. An example from the setting of the Schr\"odinger operator is to introduce dislocations to periodic potentials. This has been widely studied in both one \cite{drouot2, drouot1, korotyaev2000lattice, korotyaev2005schrodinger,dohnal2009localized} and two dimensions \cite{hempel2011spectral, hempel2011variational, hempel2012dislocation, hempel2015bound}. %The dislocations are non-local perturbations whereby, typically, half of the plane is either translated or, in two-dimensions, rotated (\emph{e.g.} \cite{hempel2011spectral}). %These represent a special case of edge state problems \cite{ammari2019topological, fefferman2016edge}.
There are some important differences between the dislocation of an array of resonators (as studied here) and the dislocation of a periodic potential. Most notably, when a periodic potential is dislocated the original configuration will be recovered periodically. Then, a quantity of interest is the \emph{edge index}, which can be defined as the net number of eigenvalues which cross a band gap over a period of dislocation (see for example \cite{drouot2,Avila2013}). If the edge index is non-zero, it means that a mid-gap frequency can be placed at any given position within the band gap (which, we said, is our goal). Moreover, according to the \emph{bulk-edge correspondence} \cite{drouot2, drouot2019bulk, graf2013bulk2d, graf2018bulk, graf2018bulk2d, drouot2019microlocal}, the edge index coincides with the bulk index of the structure without dislocation.

In our setting we will not periodically recover the original structure as we increase the dislocation and will, instead, produce two coupled half-space arrays. As the dislocation is increased, the coupling between the two halves will diminish and both mid-gap frequencies will converge to a single frequency. This single frequency corresponds to the edge mode of a half-space array, the existence of which is predicted by the bulk-edge correspondence. In contrast to the dislocation of a periodic potential (as in \cite{dohnal2009localized,drouot2, drouot1, korotyaev2000lattice, korotyaev2005schrodinger}), there will always be either 0 or 2 edge modes in the present case. There are two main results of our analysis of the dislocated infinite structure. Firstly, we will show that when a dislocation is introduced, a mid-gap frequency enters the band gap from each edge (\Cref{prop:smalld}). Following this, we prove that there are two mid-gap frequencies which converge to a single frequency within the band gap as the dislocation becomes large (\Cref{thm:main}). These two frequencies correspond to the hybridized modes of two semi-infinite arrays.

Physical realizations of the infinite structures studied here are arrays of finitely many resonators, corresponding to truncated versions of the infinite structures. To complement the aforementioned analysis, we also study a finite array of resonator pairs to which a dislocation is introduced (\Cref{sec:finite}). We show that, similar to the infinite structure, the finite array decouples into two half-systems as the dislocation increases and the two half-system hybridize for intermediate dislocations. We also conduct a stability analysis to demonstrate that the edge-mode frequencies are more stable with respect to physical imperfections than frequencies in bulk of the bandgap. We also demonstrate that the optimal stability is achieved when the frequency is in the middle of the band gap.

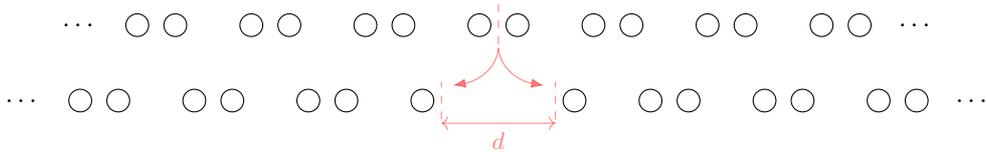
\begin{figure}
	\centering
	\begin{tikzpicture}[scale=0.5]
	\begin{scope}[shift={(1.5,2)}]
	\draw (-6,0) circle (0.3);
	\draw (-5,0) circle (0.3);
	\draw (-3,0) circle (0.3);
	\draw (-2,0) circle (0.3);
	\draw (0,0) circle (0.3);
	\draw (1,0) circle (0.3);
	\draw (3,0) circle (0.3);
	\draw (4,0) circle (0.3);
	\draw (6,0) circle (0.3);
	\draw (7,0) circle (0.3);
	\draw (9,0) circle (0.3);
	\draw (10,0) circle (0.3);
	\draw (12,0) circle (0.3);
	\draw (13,0) circle (0.3);
	\node at (14.5,0) {$\dots$};
	\node at (-7.5,0) {$\dots$};
	\draw[dashed,opacity=0.5,red] (3.5,-0.5) -- (3.5,0.6);
	\draw[-Latex,opacity=0.5,red] (3.5,-0.6) to[out=-95,in=10] (2.3,-1.6);
	\draw[-Latex,opacity=0.5,red] (3.5,-0.6) to[out=-85,in=170] (4.7,-1.6);
	\end{scope}
	\draw (-6,0) circle (0.3);
	\draw (-5,0) circle (0.3);
	\draw (-3,0) circle (0.3);
	\draw (-2,0) circle (0.3);
	\draw (0,0) circle (0.3);
	\draw (1,0) circle (0.3);
	\draw (3,0) circle (0.3);
	\draw[dashed,opacity=0.5,red] (3.5,-0.5) -- (3.5,0.5);
	\draw[dashed,opacity=0.5,red] (6.5,-0.5) -- (6.5,0.5);
	\draw (7,0) circle (0.3);
	\draw (9,0) circle (0.3);
	\draw (10,0) circle (0.3);
	\draw (12,0) circle (0.3);
	\draw (13,0) circle (0.3);
	\draw (15,0) circle (0.3);
	\draw (16,0) circle (0.3);
	\node at (17.5,0) {$\dots$};
	\node at (-7.5,0) {$\dots$};
	\draw[<->,opacity=0.5,red] (3.5,-0.6) -- (6.5,-0.6) node[pos=0.5,below]{\small $d$};
	%	\draw[dotted] (-6.5,-1.5) rectangle (3.5,1.5);
	%	\draw[dotted] (6.5,-1.5) rectangle (16.5,1.5);
	% 	\draw[<->] (3,0.5) -- (7,0.5) node[pos=0.5, yshift=7pt]{$l+d$};
	% 	\draw[<->] (0,0.5) -- (1,0.5) node[pos=0.5, yshift=7pt]{$l$};
	% 	\draw[<->] (0,-0.5) -- (3,-0.5) node[pos=0.5, yshift=-7pt]{$L$};
	% 	\draw[<->] (9,0.5) -- (10,0.5) node[pos=0.5, yshift=7pt]{$l$};
	% 	\draw[<->] (7,-0.5) -- (10,-0.5) node[pos=0.5, yshift=-7pt]{$L$};
	%	\draw[<->, opacity=0.5] (0,-0.8) -- (5,-0.8) node[pos=0.5, yshift=-7pt,]{$L$};
	\end{tikzpicture}
	\caption{We start with an array of pairs of subwavelength resonators, known to have a subwavelength band gap. A dislocation (with size $d>0$) is introduced to create mid-gap frequencies.} \label{fig:infinite}
\end{figure}

\begin{figure}
	\centering
	\begin{tikzpicture}
	\draw[line width=0.5mm,red] plot [smooth] coordinates {(0,0.5) (1,0.8) (5.5,1.05) (7.7,1.1)};
	\draw[line width=0.5mm,red] plot [smooth] coordinates {(0,2.1) (1,1.3) (4,1.12) (7.7,1.1)};
	\draw[gray!60!white,fill=gray!60!white] (0,0) -- (7.7,0) -- (7.7,0.5) --  (0,0.5);
	% 	\draw plot [smooth] coordinates {(7.7,0.5) (4,0.55) (2,0.5) (0,0.5)};
	\draw (0,0.5) -- (7.7,0.5);
	\draw[gray!60!white,fill=gray!60!white] (0,2.6) -- (7.7,2.6) -- (7.7,2.1) -- (0,2.1);
	% 	\draw plot [smooth] coordinates {(7.7,2.1) (4,2.1) (2,2.05) (0,2)} -- (0,2.5);
	\draw (7.7,2.1) -- (0,2.1);
	\draw[->,line width=0.4mm] (0,0) -- (8,0) node[pos=0.95, yshift=-8pt]{$d$};
	\draw[->,line width=0.4mm] (0,0) -- (0,2.85) node[pos=0.85, xshift=-8pt]{$\omega$};
	\draw[<-,opacity=0.5] (3.5,0.95) to[out=80,in=180] (4,1.7);
	\node[right,opacity=0.5] at (4,1.7) {\small mid-gap frequencies};
	\draw[<-,opacity=0.5] (3.4,1.15) to[out=80,in=180] (4,1.7);
	\draw [decorate,opacity=0.5,decoration={brace,amplitude=10pt}]
	(7.8,2.1) -- (7.8,0.5) node [midway,right,xshift=3mm]{\small band gap};
	\node at (3.8,2.35){\color{white}\small essential spectrum};
	\node at (3.8,0.25){\color{white}\small essential spectrum};
	\draw [decorate,opacity=0.5,decoration={brace,amplitude=10pt}]
	(-0.6,0) -- (-0.6,2.6) node [midway,xshift=-4mm,align=right,left]{\small subwavelength\\ regime};
	\end{tikzpicture}
	\caption{As the dislocation size $d$ increases from zero, a mid-gap frequency appears from each edge of the subwavelength band gap. These two frequencies converge to a single value within the subwavelength band gap as $d\to\infty$.} \label{fig:gapcrossing}
\end{figure}
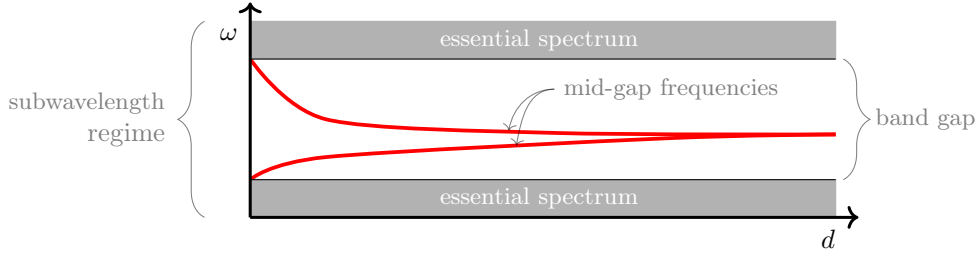

\section{Preliminaries}
In this section, we briefly review the layer potential operators and Floquet-Bloch theory that will be used in the subsequent analysis. More details on this material can, for example, be found in \cite{MaCMiPaP}.

\subsection{Layer potential techniques} \label{sec:layerpot}
Let $\Omega\subset \R^3$ be a bounded domain such that $\p \Omega$ is of class $\C^{1,s}$ for some $0<s<1$. Let $G^0$ and $G^k$  be the Laplace and outgoing Helmholtz Green's functions, respectively, defined by
$$
G^k(x,y) := -\frac{e^{\iu k|x-y|}}{4\pi|x-y|}, \quad x,y \in \R^3, x\neq y, \mathrm{Re}(k)\geq 0.
$$
We define the single layer potential $\S_{\Omega}^k: L^2(\partial \Omega) \rightarrow H_{\textrm{loc}}^1(\R^3)$ by
\begin{equation*}
\S_\Omega^k[\phi](x) := \int_{\partial \Omega} G^k(x,y)\phi(y) \dx \sigma(y), \quad x \in \R^3.
\end{equation*}
Here, the space $H_{\textrm{loc}}^1(\R^3)$ consists of functions that are square integrable on every compact subset of $\R^3$ and have a weak first derivative that is also square integrable. It is well known that the trace $\S_\Omega^0: L^2(\p \Omega) \rightarrow H^1(\p \Omega)$ is an invertible operator (see, for example, \cite{nedelec,MaCMiPaP}). Here $H^1(\p \Omega)$ denotes the set of functions that are square integrable on $\p \Omega$ and have a weak first derivative that is also square integrable.

We also define the Neumann-Poincar\'e operator $\K_\Omega^{k,*}: L^2(\partial \Omega) \rightarrow L^2(\partial \Omega)$ by
\begin{equation*}
\K_\Omega^{k,*}[\phi](x) := \int_{\partial \Omega} \frac{\partial }{\partial \nu_x}G^k(x,y) \phi(y) \dx \sigma(y), \quad x \in \partial \Omega,
\end{equation*}
where $\partial/\partial \nu_x$ denotes the outward normal derivative at $x\in\p D$.

The following so-called \emph{jump relations} describe the behaviour of the trace of $\S_\Omega^k$ on the boundary $\partial \Omega$ (see, for example, \cite{MaCMiPaP}):
\begin{equation*}
\S_\Omega^k[\phi]\big|_+ = \S_\Omega^k[\phi]\big|_-,
\end{equation*}
and
\begin{equation*}
\frac{\partial }{\partial \nu}\S_\Omega^k[\phi]\Big|_{\pm}  =  \left(\pm\frac{1}{2} I + \K_\Omega^{k,*}\right) [\phi],
\end{equation*}
where $|_+$ and $|_-$ are used to denote the limits from outside and inside $\Omega$, respectively, and $I$ is the identity. When $k$ is small, we have the following  low-frequency expansions \cite[Appendix A]{first}:
\begin{equation} \label{eq:Sk0exp}
	\S_\Omega^{k} =  \S_\Omega^{0} + O(k), \quad \mathcal{K}_\Omega^{k,*} = \mathcal{K}_\Omega^{0,*} + O(k^2).
\end{equation}
Here, the error terms are with respect to the operator norms in $\B\left(L^2(\p D),H^1(\p D) \right)$ and $\B\left(L^2(\p D)\right)$ respectively, where $\B(A,B)$ (respectively $\B(A)$) denotes the space of bounded linear operators $A\rightarrow B$ (respectively $A\rightarrow A$).

\subsection{Floquet-Bloch theory and quasiperiodic layer potentials}\label{sec:floquet}
A function $f(x)\in L^2(\R)$ is said to be $\alpha$-quasiperiodic, with quasiperiodicity $\alpha\in\mathbb{R}$, if $e^{-\iu \alpha x}f(x)$ is periodic. If the period is $L\in\mathbb{R}^+$, the quasiperiodicity $\alpha$ is an element of the torus $Y^*:= \R / \tfrac{2\pi}{L} \Z \simeq (-\pi/L, \pi/L]$, known as the \textit{Brillouin zone}. Given a function $f\in L^2(\R)$, the Floquet transform of $f$ is defined as
\begin{equation*}\label{eq:floquet}
\F[f](x,\alpha) := \sum_{m\in \Z} f(x-Lm) e^{\iu L\alpha m}.
\end{equation*}
$\F[f]$ is always $\alpha$-quasiperiodic in $x$ and periodic in $\alpha$. Let $Y_0 = [-L/2,L/2)$ be the unit cell for the $\alpha$-quasiperiodicity in $x$. The Floquet transform is an invertible map $\F:L^2(\R) \rightarrow L^2(Y_0\times Y^*)$, with inverse (see, for instance, \cite{MaCMiPaP, kuchment})
\begin{equation*}
\F^{-1}[g](x) = \frac{1}{2\pi}\int_{Y^*} g(x,\alpha) \dx \alpha, \quad x\in \R,
\end{equation*}
where $g(x,\alpha)$ is the quasiperiodic extension of $g$ for $x$ outside of the unit cell $Y_0$.

We will consider a three-dimensional problem which is periodic in one dimension. Define the unit cell $Y$ as $Y := Y_0\times \R^2$. The quasiperiodic Green's function $G^{\alpha,k}(x,y)$, for $x,y\in\R^3$, is formally defined as the Floquet transform of $G^k(x,y)$ in the $x_1$ direction with fixed $y$, \ie{},
$$G^{\alpha,k}(x,y) := -\sum_{m \in \Z} \frac{e^{\iu k|x-y-(Lm,0,0)|}}{4\pi|x-y-(Lm,0,0)|}e^{\iu \alpha Lm}.$$
If $k\neq |\alpha+\frac{2\pi}{L}m|$ for all $m\in \Z$, it is known that this series converges uniformly for $x$ in compact sets of $\R^3$, $x\neq 0$ (see \emph{e.g} \cite[Section 2.12]{MaCMiPaP}).

Let $\Omega$ be as above but with the additional assumption that $\Omega\Subset Y$. The quasiperiodic single layer potential $\mathcal{S}_\Omega^{\alpha,k}$ is defined analogously to $\S_\Omega^k$, by
$$\mathcal{S}_\Omega^{\alpha,k}[\phi](x) := \int_{\partial \Omega} G^{\alpha,k} (x,y) \phi(y) \dx\sigma(y),\quad x\in \mathbb{R}^3.$$
It is known that $\mathcal{S}_\Omega^{\alpha,0} : L^2(\p \Omega) \rightarrow H^1(\p \Omega)$ is invertible if $\alpha \neq  0$ \cite{MaCMiPaP}. There are also jump relations for the quasiperiodic single layer potential, given by
\begin{equation} \label{eq:jump1}
\S_\Omega^{\alpha,k}[\phi]\big|_+ = \S_\Omega^{\alpha,k}[\phi]\big|_-,
\end{equation}
and
\begin{equation} \label{eq:jump2}
\frac{\p}{\p\nu} \mathcal{S}_\Omega^{\alpha,k}[\phi] \Big|_{\pm} = \left( \pm \frac{1}{2} I +( \mathcal{K}_\Omega^{-\alpha,k} )^*\right)[\phi]\quad \mbox{on}~ \p \Omega,
\end{equation}
where $(\mathcal{K}_\Omega^{-\alpha,k})^*$ is the quasiperiodic Neumann-Poincaré operator, given by
$$ (\mathcal{K}_\Omega^{-\alpha, k} )^*[\phi](x):= \int_{\p \Omega} \frac{\p}{\p\nu_x} G^{\alpha,k}(x,y) \phi(y) \dx\sigma(y).$$
For small $k$, we have the following expansions \cite{MaCMiPaP}:
\begin{equation} \label{eq:Sexp}
\S_\Omega^{\alpha,k} =  \S_\Omega^{\alpha,0} + O(k^2), \quad (\mathcal{K}_\Omega^{-\alpha, k} )^* = (\mathcal{K}_\Omega^{-\alpha, 0} )^* + O(k^2).
\end{equation}
As before, the error terms are with respect to the operator norms in $\B\left(L^2(\p D),H^1(\p D) \right)$ and $\B\left(L^2(\p D)\right)$, respectively.

 	\section{Infinite dislocated system} \label{sec:infinite}
 	We will now study the problem of the dislocation of an infinite array of resonators. We will show that, in the case corresponding to non-zero bulk index, there are two mid-gap frequencies. These cover an interval in the middle of the band gap as the dislocation is varied. In \Cref{sec:periodic} we study the periodic system, \ie{} the system without dislocation, and prove that it has a subwavelength band gap. In \Cref{sec:smalldis} we study the dislocated system in the asymptotic case when the dislocation $d$ is arbitrarily small. We show that as the dislocation increases from zero, two mid-gap frequencies appear, one from each edge of the band gap. In \Cref{sec:compact} we study the case when the dislocation size is an integer number of unit cell lengths $L$, using the fact that this special case is equivalent to removing a finite number of resonators from the periodic structure. Here, we prove the existence of two mid-gap frequencies in the simplest case $d=L$, which corresponds to removing two resonators. We also show that in the limit when $d\rightarrow \infty$, any mid-gap frequency corresponds to two, hybridized, frequencies when $d$ is finite. Finally, in \Cref{sec:no//}, we study the dislocated system for a general dislocation that is larger than the width of one resonator. These values of $d$ include those in \Cref{sec:compact}, but the corresponding integral operator is significantly harder to analyse. The main goal of this section is to prove that all mid-gap frequencies will be bounded away from the edges of the band gap. In \Cref{sec:mainthm}, we combine the results of \Cref{sec:compact} and \Cref{sec:no//} to conclude that the two mid-gap frequencies found in \Cref{sec:compact} will converge to a single point as $d$ increases and therefore fill an interval in the middle of the band gap.

 	\begin{figure}[tbh]
 	\centering
 	\begin{tikzpicture}[scale=2]
 	\pgfmathsetmacro{\rb}{0.25pt}
 	\pgfmathsetmacro{\rs}{0.2pt}
 	\coordinate (a) at (0.25,0);
 	\coordinate (b) at (1.05,0);

 	\draw[dashed, opacity=0.5] (-0.5,0.85) -- (-0.5,-1);
 	\draw[dashed, opacity=0.5]  (1.8,0.85) -- (1.8,-1)node[yshift=4pt,xshift=-7pt]{};
 	\draw[{<[scale=1.5]}-{>[scale=1.5]}, opacity=0.5] (-0.5,-0.6) -- (1.8,-0.6)  node[pos=0.5, yshift=-7pt,]{$L$};
 	\draw plot [smooth cycle] coordinates {($(a)+(30:\rb)$) ($(a)+(90:\rs)$) ($(a)+(150:\rb)$) ($(a)+(210:\rs)$) ($(a)+(270:\rb)$) ($(a)+(330:\rs)$) } node[xshift=5pt, yshift=-5pt]{$D_1$};
 	\draw plot [smooth cycle] coordinates {($(b)+(30:\rb)$) ($(b)+(90:\rs)$) ($(b)+(150:\rb)$) ($(b)+(210:\rs)$) ($(b)+(270:\rb)$) ($(b)+(330:\rs)$) }  node[xshift=5pt, yshift=-5pt]{$D_2$};

 	\draw[{<[scale=1.5]}-{>[scale=1.5]}, opacity=0.5] (0.25,0.6) -- (1.05,0.6) node[pos=0.5, yshift=-5pt,]{$l$};
 	\draw[dotted,opacity=0.5] (0.25,0.7) -- (0.25,-0.8) node[at end, yshift=-0.2cm]{$p_1$};
 	\draw[dotted,opacity=0.5] (1.05,0.7) -- (1.05,-0.8) node[at end, yshift=-0.2cm]{$p_2$};

 	\begin{scope}[xshift=-2.3cm]
 	\coordinate (a) at (0.25,0);
 	\coordinate (b) at (1.05,0);
 	\draw plot [smooth cycle] coordinates {($(a)+(30:\rb)$) ($(a)+(90:\rs)$) ($(a)+(150:\rb)$) ($(a)+(210:\rs)$) ($(a)+(270:\rb)$) ($(a)+(330:\rs)$) };
	\draw plot [smooth cycle] coordinates {($(b)+(30:\rb)$) ($(b)+(90:\rs)$) ($(b)+(150:\rb)$) ($(b)+(210:\rs)$) ($(b)+(270:\rb)$) ($(b)+(330:\rs)$) };
 	\begin{scope}[xshift = 1.2cm]
 	\draw (-1.6,0) node{$\cdots$};
 	\end{scope};
 	\end{scope}
 	\begin{scope}[xshift=2.3cm]
 	\coordinate (a) at (0.25,0);
 	\coordinate (b) at (1.05,0);
 	\draw plot [smooth cycle] coordinates {($(a)+(30:\rb)$) ($(a)+(90:\rs)$) ($(a)+(150:\rb)$) ($(a)+(210:\rs)$) ($(a)+(270:\rb)$) ($(a)+(330:\rs)$) };
	\draw plot [smooth cycle] coordinates {($(b)+(30:\rb)$) ($(b)+(90:\rs)$) ($(b)+(150:\rb)$) ($(b)+(210:\rs)$) ($(b)+(270:\rb)$) ($(b)+(330:\rs)$) };
 	\begin{scope}[xshift = 1.1cm]
 	\end{scope}
 	\draw (1.7,0) node{$\cdots$};
 	\end{scope}

 	\begin{scope}[yshift=0.9cm]
 	\draw [decorate,opacity=0.5,decoration={brace,amplitude=10pt}]
 	(-0.5,0) -- (1.8,0) node [black,midway]{};
 	\node[opacity=0.5] at (0.67,0.35) {$Y$};
 	\end{scope}
 	\end{tikzpicture}
 	\caption{Example of the array in the case $d=0$. The resonators are drawn to illustrate the symmetry assumptions.} \label{fig:SSH}
 \end{figure}
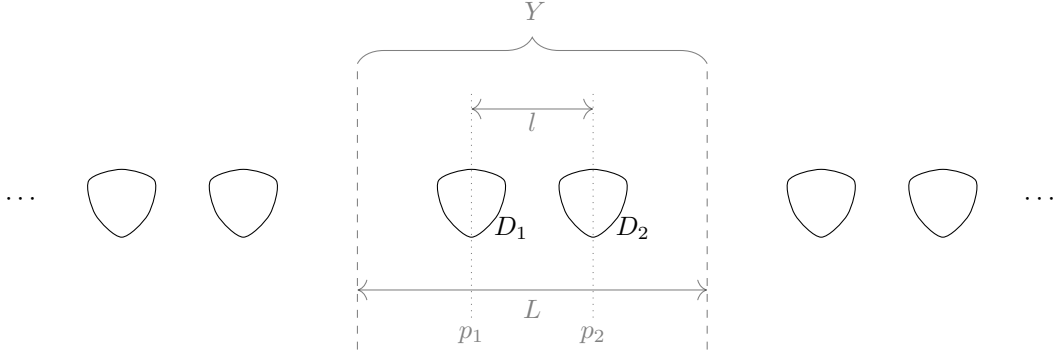

We first describe the geometry of the periodic structure, \ie{} the case without dislocation, depicted in \Cref{fig:SSH}. Let $Y=[-L/2,L/2]\times \R^2$ be the unit cell, $Y_1 = [-L/2,0]\times \R^2$ and $Y_2 = [0,L/2]\times \R^2$. For $j=1,2$, we assume that $Y_j$ contains a resonator $D_j$, which is a bounded domain $D_j\subset Y_j$ such that $\p D_j \in \C^{1,s}$ for some $0<s<1$. We denote a pair of resonators, a so-called dimer, by $D = D_1 \cup D_2$. We assume that the resonators in each dimer are separated by distance $l$ and that each individual resonator has reflection symmetry. More precisely, we assume that
\begin{equation} \label{eq:symmetry}
R_1 D_1 = D_1, \qquad R_0 D = D,
\end{equation}
where $R_1$ is the reflection in the plane $\{-l/2\} \times \R^2$ and $R_0$ is the reflection in the plane $\{0\} \times \R^2$. Observe that $R_2 := R_0R_1R_0$ describes reflection in the plane $\{l/2\} \times \R^2$ and therefore the assumptions \eqref{eq:symmetry} also imply that
$$R_2D_2 = D_2.$$

Starting from the periodic system, we assume that half of this structure is dislocated along the $x_1$-axis. Let $\mathbf{v} = (1,0,0)$ and let $d$ denote the dislocation size. We then define the periodic and dislocated systems, respectively, as
$$\C_0 = \bigcup_{\substack{m\in\Z \\ j = 1,2} } D_j^m, \qquad \C_d = \Bigg(\bigcup_{\substack{m\in\Z^- \\ j = 1,2} } D_j^m\Bigg)\cup\Bigg(\bigcup_{\substack{m\in\N \\ j = 1,2} } D_j^m+d \mathbf{v}\Bigg).$$
Here, we use the notation
$$ D_j^m = D_j + mL\mathbf{v}, \quad j=1,2, \ m\in\Z,$$
for the resonators in the $m^\text{th}$ unit cell. We introduce the notation $l_0 = l/L$, \ie{} $l_0$ is the ratio of the separation of the resonators to the unit cell length. There are two fundamentally different cases: $l_0 < 1/2$ and $l_0 > 1/2$. In the first case, the dislocation occurs between dimers of resonators, keeping each pair of resonators intact. The second case corresponds to the dislocation occurring within a dimer, splitting one pair of resonators into two ``edge'' resonators. The case $l_0 > 1/2$ was illustrated in \Cref{fig:infinite}, which, we will show, is the only case with mid-gap frequencies.

Wave propagation inside the infinite dislocated system is modelled by the Helmholtz problem
\begin{equation} \label{eq:scattering_translated}
\left\{
\begin{array} {ll}
\ds \Delta {u}+ \omega^2 {u}  = 0 & \text{in } \R^3 \setminus \p \C_d, \\
\nm
\ds  {u}|_{+} -{u}|_{-}  =0  & \text{on } \partial \C_d, \\
\nm
\ds  \delta \frac{\partial {u}}{\partial \nu} \bigg|_{+} - \frac{\partial {u}}{\partial \nu} \bigg|_{-} =0 & \text{on } \partial \C_d, \\
\nm
\ds u(x_1,x_2,x_3) & \text{satisfies the outgoing radiation condition as } \sqrt{x_2^2+x_3^2} \rightarrow \infty.
\end{array}
\right.
\end{equation}
Here, $\p/\p \nu$ denotes the outward normal derivative and $|_\pm$ indicates the limits from outside and inside $D$, respectively. Moreover, $\omega$ corresponds to the frequency of the waves. We refer to \cite{AmmariBao, Gerard1989} for the definition of the outgoing radiation condition for the scattering from compactly perturbed periodic structures. For non-compactly perturbed structures, the outgoing radiation condition amounts to choosing the \emph{outgoing} Green's function in \Cref{sec:smalldis} and thereafter \cite{thompson2018direct,thompson2008interaction}. The material parameter $\delta$ represents the contrast between the material inside the resonators and the background medium.

We assume that $\delta$ satisfies the high-contrast condition
\begin{equation*} \label{data2}
\delta \ll 1.
\end{equation*}
This assumption is crucial for subwavelength resonant modes to exist, see \emph{e.g} \cite{first,ammari2021functional}. Physically, it means that the resonators interact strongly with waves whose wavelength is considerably larger than the resonators themselves.

We denote the spectrum corresponding to the problem \eqref{eq:scattering_translated} by $\Lambda(d)$, and $\omega$ such that $\omega^2\in \Lambda(d)$ are called \emph{resonant frequencies}.  We say that a resonant frequency $\omega$ is \textit{subwavelength} if $\omega$ scales as $O(\sqrt{\delta})$ as $\delta \rightarrow 0$. By a \textit{mid-gap frequency}, we mean a value $\omega > 0$ that is in the subwavelength regime and is such that $\omega^2 \in \Lambda(d)$ but $\omega^2 \notin \Lambda(0)$.  Here, the condition $\omega^2 \notin \Lambda(0)$ means that $\omega$ is within the band gap of the periodic system. Corresponding edge mode $u$ is $L^2$-localized in $x_1$, \ie{} $\int_{\R}|u(x_1,x_2,x_3)|^2\dx x_1<\infty$, and satisfies the outgoing radiation condition as $\sqrt{x_2^2+x_3^2} \rightarrow \infty$.

It is worth emphasizing that, due to radiation in $x_2$- and $x_3$-directions, the resonant frequencies are complex with negative imaginary parts when the quasiperiodicity $\alpha$ is around the origin. Nevertheless, as we will see in Theorem \ref{thm:char_approx_infinite}, the resonant frequencies are real at leading order and, moreover, the mid-gap frequencies are real.

%in view of estimate (\ref{add2}) in Theorem \ref{thm:char_approx_infinite}, their imaginary parts are of higher order in $\delta$ than their real parts. Therefore, throughout this paper, we consider the resonant eigenvalues as real.

 	\subsection{Periodic system} \label{sec:periodic}
 	This section concerns the infinite system in the case of no dislocation. We first state some preliminary results from \cite{ammari2019topological} concerning the capacitance matrix. In \Cref{sec:Asing} we prove the existence of a band gap between the first and the second band, which is a strengthening of a result from \cite{ammari2019topological}. Moreover, we derive the asymptotic behaviour of the integral operator corresponding to the periodic problem as the frequency $\omega$ approaches the first or the second band.

 Taking the Floquet transform of the solution $u$ to \eqref{eq:scattering_translated}, the $\alpha$-quasiperiodic component $u^\alpha$ satisfies the Helmholtz problem
 \begin{equation} \label{eq:scattering_quasi_d=0}
 \left\{
 \begin{array} {ll}
 \ds \Delta u^\alpha+ \omega^2 u^\alpha  = 0 &\text{in } Y \setminus \p D, \\
 \nm
 \ds  {u^\alpha}|_{+} -{u^\alpha}|_{-}  =0  & \text{on } \partial D, \\
 \nm
 \ds  \delta \frac{\partial {u^\alpha}}{\partial \nu} \bigg|_{+} - \frac{\partial {u^\alpha}}{\partial \nu} \bigg|_{-} =0 & \text{on } \partial D, \\
 \nm
 \ds e^{-\iu  \alpha x_1}  u^\alpha(x_1,x_2,x_3)  \,\,\,&  \mbox{is periodic in } x_1, \\
 \nm
 \ds u^\alpha(x_1,x_2,x_3)& \text{satisfies the $\alpha$-quasiperiodic outgoing radiation} \\ &\hspace{0.5cm} \text{condition as } \sqrt{x_2^2+x_3^2} \rightarrow \infty.
 \end{array}
 \right.
 \end{equation}
It is well-known (see, \emph{e.g.} \cite{honeycomb, ammari2018honeycomb, highfrequency}), that \eqref{eq:scattering_quasi_d=0} has two subwavelength resonant frequencies $\omega_j^\alpha$, $j=1,2$.
 We refer to \cite{MaCMiPaP} for the definition of the $\alpha$-quasiperiodic outgoing radiation condition. Crucially, when $\omega$ is real and $|\alpha| > \omega$, the $\alpha$-quasiperiodic outgoing radiation condition states that the function is exponentially decaying as $\sqrt{x_2^2+x_3^2} \rightarrow \infty$. Recall that we are studying the subwavelength regime $\omega = O(\sqrt{\delta})$. Therefore, when $|\alpha|>K>0$ for some constant $K$, and for small enough $\delta$, the subwavelength resonant problem \eqref{eq:scattering_quasi_d=0} can be viewed as the spectral problem for a self-adjoint operator. Consequently, the subwavelength resonant frequencies $\omega_j^\alpha$, $j=1,2$, are real-valued for $|\alpha| > K >0$ (see also \cite{ammari2021functional} for a detailed discussion of this).

 Next, we formulate the quasiperiodic resonance problem \eqref{eq:scattering_quasi_d=0} as an integral equation. The solution $u^\alpha$ of \eqref{eq:scattering_quasi_d=0} can be represented as
 \begin{equation*} \label{eq:helm-solution_quasi}
 u^\alpha(x) = \begin{cases} \S_{D_1}^\omega[\phi_1^{\alpha,i}](x), \quad &x\in D_1, \\
 \S_{D_2}^\omega[\phi_2^{\alpha,i}](x), \quad &x\in D_2, \\
 \S_{D}^{\alpha,\omega}[\phi^{\alpha,o}](x), \quad &x\in Y\setminus D,
 \end{cases}
 \end{equation*}
 for some densities $\phi_1^{\alpha,i} \in  L^2(\p D_1), \ \phi_2^{\alpha,i} \in L^2(\p D_2)$ and $\phi^{\alpha,o} \in L^2(\p D)$ (here, the superscripts $i$ and $o$ indicate \textit{inside} and \textit{outside}, respectively). Throughout, we will identify $L^2(\p D) = L^2(\p D_1) \times L^2(\p D_2)$. With this identification, we write $\phi^{\alpha,i} = (\phi_1^{\alpha,i}, \phi_2^{\alpha,i})$.

 Using the jump relations \eqref{eq:jump1} and \eqref{eq:jump2}, it can be shown that~\eqref{eq:scattering_quasi_d=0} is equivalent to the boundary integral equation
 \begin{equation*}  \label{eq:boundary_quasi}
 \mathcal{A}^\alpha(\omega, \delta)[\Phi^\alpha] =0,
 \end{equation*}
 where
 \begin{equation} \label{eq:A_quasi_defn}
 \A^\alpha(\omega,\delta) =
 \begin{pmatrix}
 \hat\S_{D}^{\omega} & -\S_{D}^{\alpha,\omega} \\
 -\frac{1}{2}I + \hat\K_{D}^{\omega,*} & -\delta\left(\frac{1}{2}I + \left(\mathcal{K}_D^{ -\alpha,\omega}\right)^*\right) \\
 \end{pmatrix}, \qquad \Phi^\alpha =  \begin{pmatrix}
 \phi^{\alpha,i} \\
 \phi^{\alpha,o}
 \end{pmatrix},
 \end{equation}
 and the operators $\hat{\S}_D^\omega$ and $\hat\K_{D}^{\omega,*}$ are defined on $L^2(\p D) = L^2(\p D_1) \times L^2(\p D_2)$ as
 \begin{equation} \label{add1} \hat{\S}_D^\omega = \begin{pmatrix}
 \S_{D_1}^{\omega} & 0 \\ 0 & \S_{D_2}^{\omega} \end{pmatrix}, \qquad \hat{\K}_D^{\omega,*} = \begin{pmatrix}
 \K_{D_1}^{\omega,*} & 0 \\ 0 & \K_{D_2}^{\omega,*} \end{pmatrix}.\end{equation}
Here, and throughout this work, the block-matrix definition is used to reconcile the domains of the operators; both $\hat{\S}_D^\omega$ and $\S_D^{\alpha,\omega}$ are operators on $L^2(\p D)$, however $\hat{\S}_D^\omega$ is defined piecewise on $L^2(\p D_1)$ and $L^2(\p D_2)$ through \eqref{add1}. 

\begin{rmk}
	Here, we use the standard single-layer potential to represent the solution inside the resonators. This leads to a block $2\times 2$ integral equation, which might seem more complicated than the scalar integral equation studied in \cite{ammari2019topological}. However, this representation will, in fact, simplify the analysis of the fictitious sources used later in this paper. Another advantage of this representation is that it easily generalizes to the case of different wave speeds inside and outside the resonators.
\end{rmk}

 	\subsubsection{Quasiperiodic capacitance matrix}\label{sec:cap_quasi}
 	In this section, we state some results from \cite{ammari2019topological} on the quasiperiodic capacitance matrix.
 	Let $V_j^\alpha$ be the solution to
 	\begin{equation*} \label{eq:V_quasi}
 	\begin{cases}
 	\ds \Delta V_j^\alpha =0 \quad &\mbox{in } \quad Y\setminus D,\\
 	\ds V_j^\alpha = \delta_{ij} \quad &\mbox{on } \quad \partial D_i,\\
 	\ds e^{-\iu  \alpha_1 x_1}   V_j^\alpha (x_1,x_2,x_3)  \,\,\,&  \mbox{is periodic in } x_1, \\
 	\ds V_j^\alpha(x_1,x_2,x_3) = O\left(\tfrac{1}{\sqrt{x_2^2+x_3^2}}\right) \quad &\text{as } \sqrt{x_2^2+x_3	^2}\to\infty, \text{ uniformly in } x_1,
 	\end{cases}
 	\end{equation*}
 	where $\delta_{ij}$ is the Kronecker delta.
 	We then define the quasiperiodic capacitance matrix $C^\alpha=(C_{ij}^\alpha)$ by
 	\begin{equation*} \label{eq:qp_capacitance}
 	C_{ij}^\alpha := \int_{Y\setminus D}  \overline{\nabla V_i^\alpha}\cdot{\nabla V_j^\alpha}  \dx x,\quad i,j=1, 2.
 	\end{equation*}
 	The main motivation for studying the capacitance matrix is given in the following theorem, proved in \cite{honeycomb, bandgap}.

\begin{thm} \label{thm:char_approx_infinite}
	The subwavelength resonant frequencies $\omega_j^\alpha=\omega_j^\alpha(\delta),~j=1,2$, of the operator $\mathcal{A}^{\alpha}(\omega,\delta)$, defined in \eqref{eq:A_quasi_defn}, can be approximated as
	\begin{equation*} %\label{add2}
	\omega_j^\alpha= \sqrt{\frac{\delta \lambda_j^\alpha }{|D_1|}}  + O(\delta), \end{equation*}
	where $\lambda_j^\alpha,~j=1,2$, are eigenvalues of the quasiperiodic capacitance matrix $C^\alpha$ and $|D_1|$ is the volume of each individual resonator.
\end{thm}
 	In other words, this theorem says that the continuous spectral problem \eqref{eq:scattering_quasi_d=0} can be approximated, to leading order in $\delta$, by the discrete eigenvalue problem for $C^\alpha$.

 	\begin{lem} \label{lem:quasi_matrix_form}
 		The matrix $C^\alpha$ is Hermitian with constant diagonal, \ie{},
 		$$C_{11}^\alpha = C_{22}^\alpha \in \R, \quad C_{12}^\alpha = \overline{C_{21}^\alpha} \in \mathbb{C}.$$
 		%$$C^\alpha = \begin{pmatrix}c_1^\alpha & c_2^\alpha \\ \overline{c_2^\alpha} & c_1^\alpha \end{pmatrix},$$for $c_1^\alpha\in\mathbb{R}$ and $c_2^\alpha\in\mathbb{C}$.
 	\end{lem}
 	Using the jump conditions, in the case $\alpha\neq0$, it can be shown that the capacitance coefficients $C_{ij}^\alpha$ are also given by
 	$$ C_{ij}^\alpha = - \int_{\partial D_i} \psi_j^\alpha \dx \sigma,\quad i,j=1, 2,$$
 	where $\psi_j^\alpha$ are defined by
 	$$\psi_j^\alpha = (\S_D^{\alpha,0})^{-1}[\chi_{\p D_j}].$$

 	Since $C^\alpha$ is Hermitian, the following lemma follows directly.
 	\begin{lem} \label{lem:evec}
 		The eigenvalues and corresponding eigenvectors of the quasiperiodic capacitance matrix are given by
 		\begin{align*}
 		\lambda_1^\alpha &= C_{11}^\alpha - \left|C_{12}^\alpha \right|, \qquad
 		\begin{pmatrix}
 		a_1  \\ b_1
 		\end{pmatrix} = \frac{1}{\sqrt{2}}\begin{pmatrix}
 		- e^{\iu \theta_\alpha}  \\ 1
 		\end{pmatrix}, \\
 		\lambda_2^\alpha &= C_{11}^\alpha + \left|C_{12}^\alpha \right|, \qquad
 		\begin{pmatrix}
 		a_2  \\ b_2
 		\end{pmatrix} = \frac{1}{\sqrt{2}}\begin{pmatrix}
 		e^{\iu \theta_\alpha}  \\ 1
 		\end{pmatrix},
 		\end{align*}
 		where, for $\alpha$ such that $C_{12}^\alpha\neq0$, $\theta_\alpha\in[0,2\pi)$ is defined to be such that
 		\begin{equation*}
 		e^{\iu \theta_\alpha} = \frac{C_{12}^\alpha}{|C_{12}^\alpha|}.
 		\end{equation*}
 	\end{lem}
 	Using these eigenvectors, we define bases $\{u_1^\alpha$, $u_2^\alpha\}$, $\{\chi_1^\alpha$, $\chi_2^\alpha\}$ of $\ker{\left(-\frac{1}{2}I + \left(\K_D^{-\alpha,0}\right)^*\right)}$ \\ and~$\ker{\left(-\frac{1}{2}I + \K_D^{\alpha,0}\right)}$, respectively, as
 	\begin{align*}
 	u_1^\alpha &= \frac{1}{\sqrt{2}}\left(-e^{\iu \theta_\alpha}\psi_1^\alpha + \psi_2^\alpha\right),
	&u_2^\alpha &= \frac{1}{\sqrt{2}}\left(e^{\iu \theta_\alpha}\psi_1^\alpha + \psi_2^\alpha\right), \\
  	\chi_1^\alpha &= \frac{1}{\sqrt{2}}\left(-e^{\iu \theta_\alpha}\chi_{\p D_1} + \chi_{\p D_2}\right),
	&\chi_2^\alpha &= \frac{1}{\sqrt{2}}\left(e^{\iu \theta_\alpha}\chi_{\p D_1} + \chi_{\p D_2}\right).
 	\end{align*}
 	Observe that $\langle \chi_i^\alpha, u_j^\alpha \rangle = -\delta_{i,j} \lambda_i^\alpha$ for $i,j=1,2$. Here, $\langle \cdot, \cdot \rangle$ denotes the $L^2(\p D)$ inner product
 	$$\langle u, v \rangle = \int_{\p D} \overline{u(y)}v(y)\dx \sigma (y).$$

 	In the so-called \emph{dilute} regime, the capacitance coefficients can be computed explicitly. This regime is defined by expressing the two resonators $D_1$ and $D_2$ as rescalings of the two fixed domains $B_1$ and $B_2$:
 	\begin{equation}\label{eq:dilute}
 	D_1=\epsilon B_1 - \frac{l}{2}\mathbf{v}, \quad  D_2=\epsilon B_2 + \frac{l}{2}\mathbf{v},
 	\end{equation}
 	for some small parameter $\epsilon > 0$.

 	We define the capacitance $\textrm{Cap}_{B_i}$ of the fixed domains as
 	$$
 	\textrm{Cap}_{B_i} := -\int_{\p B_i}\psi_{B_i} \dx \sigma,
 	$$
 	where $\psi_{B_i} := (\mathcal{S}_{B_i}^0)^{-1}[\chi_{\p B_i}]$. Due to symmetry, the capacitance is the same for the two domains and therefore will simply be denoted by $\textrm{Cap}_{B}$;
 	$$\textrm{Cap}_{\epsilon B_i} = \textrm{Cap}_{\epsilon B_i} =: \textrm{Cap}_{\epsilon B}.$$
 	Rescaling the domain, we have that
 	\begin{equation*}\label{eq:cap_scale}
 	\textrm{Cap}_{\epsilon B_i} = \epsilon \textrm{Cap}_{B} \quad i=1,2.
 	\end{equation*}
 	Similarly, by rescaling, we find that the capacitance coefficients satisfy
 	\begin{equation}\label{eq:capcoeff_scale}
 	|C_{i,j}^\alpha| \leq \epsilon C \quad i,j=1,2,
 	\end{equation}
 	for some constant $C$ independent of $\alpha \in Y^*$.

 	\begin{lem}\label{lem:cap_estim_quasi}
 		We assume that the resonators are in the dilute regime specified by \eqref{eq:dilute}. Then, for every $\epsilon_0 > 0$ and $p\in \N$ there exists a constant $A_p$ such that we have the following asymptotics of the capacitance matrix $C_{ij}^\alpha$ for $\epsilon < \epsilon_0$:
 		\begin{align*}
 		C_{11}^\alpha &= \epsilon \mathrm{Cap}_B - \frac{(\epsilon \mathrm{Cap}_B)^2}{4\pi}\sum_{m \neq 0}  \frac{e^{\iu m \alpha L}}{  |mL| } + o(\epsilon^2), %\label{eq:c1q}
 		\\
 		C_{12}^\alpha &= -\frac{(\epsilon \mathrm{Cap}_B)^2}{4\pi}\sum_{m =-\infty}^\infty \frac{e^{\iu m \alpha L} }{  |mL+l| } + o(\epsilon^2), %\label{eq:c2q}
 		\end{align*}
 		uniformly in $\alpha$ for $|\alpha| \geq A_p\epsilon^p$.
 	\end{lem}
 \Cref{lem:cap_estim_quasi} is a generalisation of a result from \cite{ammari2019topological} and shows, essentially, that for smaller $\epsilon$, the asymptotic formulas are valid for $\alpha$ closer to 0. \Cref{lem:cap_estim_quasi} can be proved by following the steps in \cite{ammari2019topological} under the additional observation that the sums have a logarithmic behaviour as $\alpha \rightarrow 0$:
 $$
 \sum_{m \neq 0}\frac{e^{\iu m\alpha L}}{|m|} = -\log\big(2-2\cos(\alpha L)\big).
 $$
\subsubsection{Bandgap opening and singularity of $\A^\alpha$} \label{sec:Asing}
The next theorem describes the subwavelength band gap opening and the edge points of the bands.
\begin{thm}\label{thm:bandgap}
	In the dilute regime, we have
	$$\max_{\alpha \in Y^*} \lambda_1^\alpha = \lambda_1^{\pi/L}, \qquad \min_{\alpha \in Y^*} \lambda_2^\alpha = \lambda_2^{\pi/L},$$
	for $\epsilon$ small enough.
\end{thm}
\begin{proof}
	Observe first that if $l_0 > 1/2$, we can redefine the unit cell so that $l_0 < 1/2$, without changing the band structure. Therefore, it is enough to consider the case $l_0 \leq 1/2$. We have
	\begin{align*}
	\lambda_1^\alpha &= C_{11}^\alpha - |C_{12}^\alpha| \\
	& \leq C_{11}^\alpha + \Re\left( C_{12}^\alpha\right) \\
	&= \frac{1}{2}\text{Cap}_D^{\alpha},
	\end{align*}
	where $\text{Cap}_D^{\alpha}$ is the capacitance of $D$ defined by
	$$\text{Cap}_D^{\alpha} = \int_{\p D}\left(\S_D^{\alpha,0} \right)^{-1} [\chi_{\p D}] \dx \sigma.$$
	Using the variational characterisation of $\text{Cap}_D^{\alpha}$, in the same way as in \cite{highfrequency}, it is shown that the maximum of $\text{Cap}_D^{\alpha}$ is attained at $\alpha=\pi/L$. Moreover, in the dilute regime, $C_{12}^{\pi/L}$ is a non-positive real number \cite{ammari2019topological}. We therefore have
	$$\lambda_1^{\pi/L} = \frac{1}{2}\text{Cap}_D^{\pi/L},$$
	so the maximum of $\lambda_1^\alpha$ is attained at $\alpha=\pi/L$.

	We now turn to the second eigenvalue $\lambda_2^\alpha$. Similarly, we have
	\begin{align}
	\lambda_2^\alpha &= C_{11}^\alpha + |C_{12}^\alpha| \nonumber \\
	& \geq C_{11}^\alpha - \Re  \left( C_{12}^\alpha\right). \label{eq:lambda2}
	\end{align}
	We can formulate a variational characterisation for $C_{11}^\alpha - \text{Re} \left( C_{12}^\alpha\right)$ in terms of the Dirichlet energy.	Let $\mathcal{C}_\alpha^\infty$ be the set of functions in $\C^\infty(Y)$ that can be extended to $\alpha$-quasiperiodic functions in $\C^\infty(\R^3)$ decaying as $O\left((x_2^2+x_3^2)^{-1/2}\right) \ \mathrm{as} \ \sqrt{x_2^2+x_3^2}\rightarrow \infty$. Let
	$$\Hc = \left\{ v\in H^1_{\text{loc}}(Y) \ \Big| \ v(x_1,x_2,x_3) = O\left(\frac{1}{\sqrt{x_2^2+x_3^2}}\right) \ \mathrm{as} \ \sqrt{x_2^2+x_3^2}\rightarrow \infty \right\}$$
	and let $\Hc_\alpha$ be the closure of $\mathcal{C}_\alpha^\infty$ in $\Hc$. Then define (see, for instance, \cite{nedelec})
	$$\V_\alpha = \left\{v\in \Hc_\alpha \ \Big| \ v=-\frac{1}{\sqrt{2}} \text{ on } \p D_1, v=\frac{1}{\sqrt{2}}\text{ on } \p D_2\right\}.$$
	We then have the variational characterisation
	\begin{equation}\label{eq:var2}
	C_{11}^\alpha - \text{Re} \left( C_{12}^\alpha\right) = \min_{v\in \V_\alpha} \int_{Y\setminus D}|\nabla v|^2 \dx x.
	\end{equation}
	Indeed, the minimiser $v_0$ satisfies $\Delta v_0 = 0$ in $Y\setminus D$ and therefore $v_0 = \frac{1}{\sqrt{2}}\left(-V_1^\alpha + V_2^\alpha\right)$. Equation \eqref{eq:var2} then follows by expanding the integral.

	Define $\V = \cup_{\alpha\in Y^*}\V_\alpha$. From \eqref{eq:var2} we find
	$$\min_{\alpha\in Y^*} \bigg[ C_{11}^\alpha - \text{Re}\left( C_{12}^\alpha\right) \bigg]= \min_{v\in \V} \int_{Y\setminus D}|\nabla v|^2 \dx x.$$
	Because of the symmetry of $D$, the corresponding minimizer $v_1$ is an odd function in $x_1$. In other words, $v_1$	is a $\pi/L$-quasiperiodic function, so
	\begin{equation} \label{eq:min2}
	\min_{\alpha\in Y^*} \bigg[ C_{11}^\alpha - \Re\left( C_{12}^\alpha\right) \bigg] = C_{11}^{\pi/L} - \Re\left( C_{12}^{\pi/L}\right).
	\end{equation}
	At $\alpha=\pi/L$, \eqref{eq:lambda2} is an equality. This, together with \eqref{eq:min2}, proves that the minimum of $\lambda_2^\alpha$ is attained at $\alpha=\pi/L$.
\end{proof}

\begin{corollary} \label{cor:bandgap}
	In the dilute regime and with $\delta$ sufficiently small, there exists a subwavelength band gap between the first two bands if $l_0\neq 1/2$, \ie{}
	\begin{equation*}
		\max_{\alpha \in Y^*} \Re(\omega_1^\alpha) = \omega_1^{\pi/L} <  \omega_2^{\pi/L} = \min_{\alpha \in Y^*} \Re(\omega_2^\alpha),
	\end{equation*}
	for $\epsilon$ and $\delta$ small enough.
\end{corollary}
\begin{proof}
	Again, it is enough to consider the case $l_0 < 1/2$.
	Given a constant $K>0$, we know that we can choose $\delta$ small enough so that $\omega_1^\alpha$ and $\omega_2^\alpha$ are real-valued for $|\alpha|>K$. Corresponding Bloch modes are exponentially decaying away from the resonators. From \Cref{thm:bandgap} and \Cref{thm:char_approx_infinite}, it follows that $$\max_{\alpha \in Y^*} \Re(\omega_1^\alpha) = \max_{|\alpha|>K} \omega_1^{\alpha}, \quad \min_{\alpha \in Y^*} \Re(\omega_2^\alpha) = \min_{|\alpha|>K} \omega_2^{\alpha}.$$
	
	Let $\C_\alpha^\infty$ be the set of functions in $\C^\infty(Y)$ that can be extended to $\alpha$-quasiperiodic functions in $\C^\infty(\R^3)$ decaying exponentially as $\sqrt{x_2^2+x_3^2}\rightarrow \infty$. Let $\Hc$ be the set of functions in $H^1(Y)$ which decay exponentially as $\sqrt{x_2^2+x_3^2}\rightarrow \infty$
	and let $\Hc_\alpha$ be the closure of $\mathcal{C}_\alpha^\infty$ in $\Hc$. Let $\rho(x) = 1+ (\delta^{-1} - 1)\chi_D$, and let $R(v)$ denote the Rayleigh quotient
	$$R(v) =  \frac{\int_Y \rho |\nabla v|^2 \dx x}{\int_Y \rho |v|^2\dx x }.$$
	The Bloch mode $v_1^{\pi/L}$ at $\alpha=\pi/L$ is an even function and hence $v_1^{\pi/L}|_{\p Y} = 0$. Therefore $v_1^{\pi/L}\in \Hc_\alpha$ for each $\alpha$ with $|\alpha| >K$, and so
	$$\omega_1^\alpha =  \min_{v\in \Hc_\alpha} R(v) \leq  \min_{v\in \Hc_{\pi/L}}  R(v) = \omega_1^{\pi/L}.$$
	Next, take $v\in \operatorname{span}(v_1^\alpha)^\perp$ and write
	$$v = w_1 + w_2, \qquad w_1\in \operatorname{span}(v_1^{\pi/L})^{(\perp,\alpha)}, \ w_2 \in \operatorname{span}(v_1^{\pi/L}),$$
	where $^{(\perp,\alpha)}$ denotes the orthogonal complement with respect to $\Hc_\alpha$. Since $R(v) \geq R(w_1)$ we have
	$$\omega_2^\alpha =  \min_{v\in \operatorname{span}(v_1^\alpha)^\perp} R(v) \geq \min_{v\in \operatorname{span}(v_1^{\pi/L})^{(\perp,\alpha)}}  R(v)\geq   \min_{v\in \operatorname{span}(v_1^{\pi/L})^{(\perp,\pi/L)}}  R(v) = \omega_2^{\pi/L}.$$
	Finally, from \cite{ammari2019topological}, we know that if $l_0\neq 1/2$ then $\lambda_1^{\pi/L}< \lambda_2^{\pi/L}$. Hence, \Cref{thm:char_approx_infinite} gives us that $\omega_1^{\pi/L}< \omega_2^{\pi/L}$, which concludes the proof.
	\begin{comment}
	From \cite{ammari2019topological}, we know that if $l_0\neq 1/2$ then $\lambda_1^{\pi/L}< \lambda_2^{\pi/L}$. Hence, \Cref{thm:bandgap} gives us that
	\begin{equation*}
	\max_{\alpha \in Y^*} \lambda_1^\alpha < \min_{\alpha \in Y^*} \lambda_2^\alpha.
	\end{equation*}
	The result then follows from \Cref{thm:char_approx_infinite}, provided that $\delta$ is sufficiently small.% which says that, at leading order in $\delta$, $\omega_1^\alpha$ and $\omega_2^\alpha$ are determined by $\lambda_1^\alpha$ and $\lambda_2^\alpha$.
	\end{comment}
\end{proof}
\begin{rmk}
	If $l_0 = 1/2$, it follows using arguments analogous to those in \cite{honeycomb} that the band gap closes at $\alpha = \pi/L$:
	$$\omega_1^{\pi/L} = \omega_2^{\pi/L}.$$
\end{rmk}

Next, we will explicitly describe the behaviour of $\left(\A^\alpha(\omega,\delta)\right)^{-1}$ as $\omega$ approaches the edge of the first or the second band. The results are similar to Lemmas 4.1 and 4.2 of \cite{linedefect}, but generalized to the case when $D$ consists of two connected domains of general shape. Throughout the remainder of this section, we assume that $|\alpha| > K > 0$ for some $K$.

Using $u_1^\alpha, u_2^\alpha, \chi_1^\alpha$ and $\chi_2^\alpha$ as defined in \Cref{sec:cap_quasi}, we decompose the operator $\frac{1}{2} I + (\K_D^{-\alpha,0})^*$ as
$$\frac{1}{2} I + (\K_D^{-\alpha,0})^* = P_\alpha + Q_\alpha,$$
where
$$P_\alpha = -\frac{\langle \chi_1^\alpha, \cdot \rangle}{\lambda_1^\alpha}u_1^\alpha -\frac{\langle \chi_2^\alpha, \cdot\rangle}{\lambda_2^\alpha}u_2^\alpha, \qquad Q_\alpha = \frac{1}{2} + (\K_D^{-\alpha,0})^* - P_\alpha.$$
Then it follows that $Q_\alpha[u_i^\alpha] = 0$ and $Q_\alpha^*[\chi_i^\alpha] = 0$ for $i = 1,2$. Here, $Q_\alpha^*$ denotes the $L^2(\p D)$-adjoint of $Q_\alpha$. As we will see, the reason for using this decomposition is that $Q_\alpha$ will only contribute to higher-order terms when computing the inverse $\left(\A^\alpha(\omega,\delta)\right)^{-1}$.

We consider the limit as $\delta$ goes to zero. Recall that for $\omega$ inside the corresponding band gap, we have $\omega = O(\sqrt{\delta})$. Using the operators $P_\alpha$ and $Q_\alpha$, along with the expansions in \eqref{eq:Sk0exp} and \eqref{eq:Sexp}, we can decompose the operator $\A^\alpha(\omega,\delta)$ as
$$\A^\alpha(\omega,\delta) =
\begin{pmatrix}
\hat\S_{D}^{\omega} & -\S_{D}^{\alpha,\omega} \\
-\frac{1}{2}I + \hat\K_{D}^{\omega,*} & 0 \\
\end{pmatrix}  - \delta \begin{pmatrix} 0 & 0 \\ 0 & P_\alpha \end{pmatrix} - \delta \begin{pmatrix} 0 & 0 \\ 0 & Q_\alpha \end{pmatrix} + O(\delta^{3/2}),
$$
with respect to the operator norm in $\B\left(  \left(L^2(\p D) \right)^2,L^2(\p D) \times H^2(\p D)\right)$. We define
$$A_0(\omega) =
\begin{pmatrix}
\hat\S_{D}^{\omega} & -\S_{D}^{\alpha,\omega} \\
-\frac{1}{2}I + \hat\K_{D}^{\omega,*} & 0 \\
\end{pmatrix}, \qquad
A_1(\omega,\delta) = I - \delta A_0^{-1} \begin{pmatrix} 0 & 0 \\ 0 & P_\alpha \end{pmatrix}.
$$

We introduce the basis $\{u_1, u_2\}$ of $\ker{\left(-\frac{1}{2}I + \hat{\K}_D^{0,*} \right)} \subset L^2(\p D)$ as
$$u_1 = \frac{1}{\sqrt{2}}\left(-e^{\iu \theta_\alpha} \psi_1 + \psi_2\right), \qquad u_2 = \frac{1}{\sqrt{2}}\left(e^{\iu \theta_\alpha} \psi_1 + \psi_2\right),$$
where $\psi_j$ are defined by
$$\psi_j = (\S_{D_j}^{0})^{-1}[\chi_{\p D_j}].$$
We then have the following result.
\begin{lem}
	\begin{itemize}
		\item[(i)] For $\omega \neq 0$, $A_0: \left(L^2(\p D)\right)^2 \rightarrow \left(L^2(\p D)\right)^2$ is invertible and, as $\omega \rightarrow 0$,
		$$A_0^{-1} = \begin{pmatrix}
		0 &  \ds -\frac{\langle  \chi_{\p D_1}, \cdot\rangle \psi_1 + \langle  \chi_{\p D_2}, \cdot\rangle \psi_2}{\omega^2|D_1|}  +O\left(\frac{1}{\omega}\right) \\
		\nm
		-\left(\mathcal{S}_D^{\alpha,0}\right)^{-1} + O(\omega^2) & \ds -\frac{\langle \chi_{\p D_1}, \cdot \rangle \psi_1^\alpha + \langle \chi_{\p D_2}, \cdot \rangle \psi_2^\alpha}{\omega^2|D_1|}  +O\left(\frac{1}{\omega}\right)
		\end{pmatrix},$$
		with respect to the operator norm in $\B\left(L^2(\p D)\right)$, where $|D_1|$ denotes the volume of $D_1$.
		\item[(ii)]  For $\omega \neq \sqrt{\frac{\delta\lambda_j^\alpha}{|D_1|}}$, $A_1: L^2(\p D) \rightarrow L^2(\p D)$ is invertible. As $\delta \rightarrow 0$ and $\omega = C\sqrt\delta$ for $C\neq \sqrt{\frac{\lambda_j^\alpha}{|D_1|}}$ we have
		$$A_1^{-1} = \begin{pmatrix}
		I & -P\left(P_\alpha^\perp\right)^{-1}\\
		0 & \left(P_\alpha^\perp\right)^{-1}
		\end{pmatrix}+O(\omega),$$
		with respect to the operator norm in $\B\left(\left(L^2(\p D)\right)^2\right)$, where $$ \ds P = \frac{\delta}{\omega^2|D_1|}\big( \langle \chi_1^\alpha, \cdot \rangle u_1  + \langle \chi_2^\alpha, \cdot\rangle u_2 \big),\qquad P_\alpha^\perp =I+\frac{\delta}{\omega^2|D_1|}\big( \langle \chi_1^\alpha, \cdot\rangle u_1^\alpha  + \langle \chi_2^\alpha, \cdot \rangle u_2^\alpha \big).$$
	\end{itemize}
\end{lem}
\noindent \textit{Proof of (i).}
Using block matrix inversion, we find that
\begin{equation}\label{eq:Alem1}A_0^{-1} = \begin{pmatrix}
0 &  \ds \left(-\frac{1}{2}I+ \hat{\K}_{D}^{\omega, *}\right)^{-1} \\
\nm
-\left(\S_D^{\alpha,\omega}\right)^{-1} &  \ds \left(\S_D^{\alpha,\omega}\right)^{-1} \hat{\S}_D^{\omega} \left(-\frac{1}{2}I+ \hat{\K}_D^{\omega, *}\right)^{-1}
\end{pmatrix},\end{equation}
which is well-defined since $-\frac{1}{2}I+ \mathcal{K}_{D_i}^{\omega, *}: L^2(\p D) \rightarrow L^2(\p D)$ is invertible for $\omega\neq0$ for both $i=1,2,$ see, for instance, \cite{MaCMiPaP}. Here, $\hat{\S}_D^{\omega}$ and $\hat{\K}_D^{\omega, *}$ are defined in (\ref{add1}).

From the low-frequency expansion \eqref{eq:Sexp} of $\mathcal{S}_D^{\alpha,\omega}$  we have
\begin{align} \label{eq:Alem2}
\left(\mathcal{S}_D^{\alpha,\omega}\right)^{-1} &=\left(\mathcal{S}_D^{\alpha,0}\right)^{-1} +O(\omega^2)
\end{align}
in the operator norm. The operator $\left(-\frac{1}{2}I+ \mathcal{K}_{D_i}^{\omega,*}\right)^{-1}$ is known to be singular at $\omega= 0$, see \cite{MaCMiPaP}. Explicitly, we have
$$\left(-\frac{1}{2}I+ \mathcal{K}_{D_i}^{\omega,*}\right)^{-1} = -\frac{\langle \chi_{\p D_i}, \cdot  \rangle}{\omega^2 |D_i|}\psi_i + R_i(\omega),$$
where $R_i(\omega) = O(1)$ as $\omega \rightarrow 0$. Since $|D_1| = |D_2|$, we have
\begin{equation} \label{eq:Alem3}
\left(-\frac{1}{2}I+ \hat{\K}_D^{\omega, *}\right)^{-1} = -\frac{\langle   \chi_{\p D_1}, \cdot \rangle \psi_1 + \langle \chi_{\p D_2}, \cdot \rangle \psi_2}{\omega^2|D_1|}  +O\left(1\right),
\end{equation}
with respect to the operator norm in $\B(L^2(\p D))$,
where we, as before, identify $L^2(\p D) = L^2(\p D_1) \times L^2(\p D_2)$. Moreover, we know that $\S_{D_i}^{\omega}[\psi_i] = \chi_{\p D_i} + O(\omega)$ and so
\begin{equation}\label{eq:Alem4}
\left(\mathcal{S}_D^{\alpha,\omega}\right)^{-1} \hat{\S}_D^{\omega} \left(-\frac{1}{2}I+ \hat{\K}_D^{\omega, *}\right)^{-1} =  -\frac{\langle  \chi_{\p D_1}, \cdot \rangle \psi_1^\alpha + \langle  \chi_{\p D_2}, \cdot \rangle \psi_2^\alpha}{\omega^2|D_1|}  +O\left(\frac{1}{\omega}\right).
\end{equation}
Combining equations \eqref{eq:Alem1}, \eqref{eq:Alem2}, \eqref{eq:Alem3} and \eqref{eq:Alem4} proves \textit{(i)}.
\qed
\bigskip

\noindent \textit{Proof of (ii).}
From \eqref{eq:Alem3} we have
$$
\left(-\frac{1}{2}I+ \hat{\K}_D^{\omega, *}\right)^{-1}P_\alpha  = -\frac{\langle \chi_1^\alpha, \cdot \rangle u_1  + \langle  \chi_2^\alpha, \cdot \rangle u_2 }{\omega^2|D_1|} +O\left(\frac{1}{\omega}\right).
$$
Similarly, we have
$$
\left(\mathcal{S}_D^{\alpha,\omega}\right)^{-1} \hat{\S}_D^{\omega} \left(-\frac{1}{2}I+ \hat{\K}_D^{\omega, *}\right)^{-1}P_\alpha = -\frac{\langle  \chi_1^\alpha, \cdot \rangle u_1^\alpha  + \langle  \chi_2^\alpha , \cdot \rangle u_2^\alpha }{\omega^2|D_1|} +O\left(\frac{1}{\omega}\right).
$$
We then find that
$$A_1 = \begin{pmatrix}
I &  \ds \frac{\delta}{\omega^2|D_1|}\big( \langle \chi_1^\alpha, \cdot \rangle u_1  + \langle  \chi_2^\alpha , \cdot \rangle u_2 \big)\\
\nm
0 & \ds I+\frac{\delta}{\omega^2|D_1|}\big( \langle \chi_1^\alpha, \cdot \rangle u_1^\alpha  + \langle \chi_2^\alpha , \cdot \rangle u_2^\alpha \big)
\end{pmatrix}  +O(\omega) .$$
Define
$$P = \frac{\delta}{\omega^2|D_1|}\big( \langle  \chi_1^\alpha, \cdot \rangle u_1  + \langle \chi_2^\alpha , \cdot \rangle u_2 \big),$$
and
$$P_\alpha^\perp =I+\frac{\delta}{\omega^2|D_1|}\big( \langle  \chi_1^\alpha, \cdot \rangle u_1^\alpha  + \langle \chi_2^\alpha, \cdot  \rangle u_2^\alpha \big).$$
The leading order of $A_1$ is invertible precisely when $P_\alpha^\perp$ is invertible. This occurs precisely when $P_\alpha^\perp u_i^\alpha \neq 0$ for $i=1,2$, \ie{} when
$$\omega\neq \sqrt{\frac{\delta \lambda_i^\alpha}{|D_1|}} \quad \mbox{for } i=1,2.$$
Moreover, we have
$$A_1^{-1} = \begin{pmatrix}
I & -P\left(P_\alpha^\perp\right)^{-1}\\
0 & \left(P_\alpha^\perp\right)^{-1}
\end{pmatrix}+O(\omega).$$
This shows \textit{(ii)}.
\qed

The following result can be proved by using the same arguments as those in \cite{linedefect}.
\begin{lem} \label{lem:invAa}
	As $\delta \rightarrow 0$ and $\omega = C\sqrt\delta$ for $C\neq \sqrt{\frac{\lambda_j^\alpha}{|D_1|}}$, we have
	\begin{align*}
	(\A^\alpha(\omega, \delta))^{-1} &= A_1^{-1}A_0^{-1}\big(I + O(\delta)\big),
	\end{align*}
	where the error term is with respect to the operator norm in $\B\left(\left(L^2(\p D)\right)^2\right)$.
\end{lem}
Based on this lemma, we can explicitly compute $\left(\A^\alpha(\omega, \delta)\right)^{-1}$. We have as $\delta \to 0$ with $\omega = C\delta^{1/2}$,
$$\begin{array}{l}
\left(\A^\alpha(\omega, \delta)\right)^{-1} = A_1^{-1}A_0^{-1}\big(I + O(\delta)\big) \\
\nm
= \begin{pmatrix}
I & -P\left(P_\alpha^\perp\right)^{-1}\\
0 & \left(P_\alpha^\perp\right)^{-1}
\end{pmatrix}
\begin{pmatrix}
0 &  \ds -\frac{\langle  \chi_{\p D_1}, \cdot  \rangle \psi_1 + \langle \chi_{\p D_2}, \cdot  \rangle \psi_2}{\omega^2|D_1|}  +O\left(\frac{1}{\omega}\right) \\
\nm
-\left(\mathcal{S}_D^{\alpha,0}\right)^{-1} + O(\omega^2) & \ds -\frac{\langle  \chi_{\p D_1}, \cdot  \rangle \psi_1^\alpha + \langle \chi_{\p D_2}, \cdot \rangle \psi_2^\alpha}{\omega^2|D_1|}  +O\left(\frac{1}{\omega}\right)
\end{pmatrix} \\
\nm
=
\begin{pmatrix}
P\left(P_\alpha^\perp\right)^{-1}\left(\mathcal{S}_D^{\alpha,0}\right)^{-1} + O(\omega) & \ds -\frac{\langle \chi_{\p D_1}, \cdot \rangle \psi_1 + \langle \chi_{\p D_2}, \cdot  \rangle \psi_2}{\omega^2|D_1|}   + P\left(P_\alpha^\perp\right)^{-1}\frac{\langle \chi_{\p D_1}, \cdot \rangle \psi_1^\alpha + \langle  \chi_{\p D_2}, \cdot \rangle \psi_2^\alpha}{\omega^2|D_1|} +O\left(\frac{1}{\omega}\right) \\
\nm
-\left(P_\alpha^\perp\right)^{-1}\left(\mathcal{S}_D^{\alpha,0}\right)^{-1} + O(\omega) & \ds -\left(P_\alpha^\perp\right)^{-1}\frac{\langle  \chi_{\p D_1}, \cdot \rangle \psi_1^\alpha + \langle  \chi_{\p D_2}, \cdot \rangle \psi_2^\alpha}{\omega^2|D_1|}  +O\left(\frac{1}{\omega}\right)
\end{pmatrix},
\end{array}$$
where the error terms are now with respect to the operator norm in $\B(L^2(\p D))$. We will simplify the elements in the right column in the above expression, which is the part of $\left(\A^{\alpha}\right)^{-1}$ that is relevant for the rest of the work. Define
$$\left(\A^\alpha(\omega, \delta)\right)^{-1} = \begin{pmatrix} A_{11} & A_{12} \\ A_{21} & A_{22} \end{pmatrix}.$$
We can compute
\begin{align*}
\left(P_\alpha^\perp\right)^{-1} \psi_1^\alpha &= -\frac{e^{-\iu \theta_\alpha}}{\sqrt{2}}\left( \frac{\omega^2}{\omega^2-\left(\omega_1^\alpha\right)^2}\right) u_1^\alpha + \frac{e^{-\iu \theta_\alpha}}{\sqrt{2}}\left( \frac{\omega^2}{\omega^2-\left(\omega_2^\alpha\right)^2}\right) u_2^\alpha + O(\omega), \\
\left(P_\alpha^\perp\right)^{-1} \psi_2^\alpha &= \frac{1}{\sqrt{2}}\left( \frac{\omega^2}{\omega^2-\left(\omega_1^\alpha\right)^2}\right) u_1^\alpha + \frac{1}{\sqrt{2}}\left( \frac{\omega^2}{\omega^2-\left(\omega_2^\alpha\right)^2}\right) u_2^\alpha + O(\omega),
\end{align*}
with respect to the $L^2(\p D)$-norm. Then we obtain
\begin{align*}
P u_1^\alpha = -\frac{\left(\omega_1^\alpha\right)^2}{\omega^2} u_1 + O(\omega), \qquad
P u_2^\alpha = -\frac{\left(\omega_2^\alpha\right)^2}{\omega^2} u_2 + O(\omega).
\end{align*}
Consequently, we have

%\begin{multline}
%A_{12} = -\frac{\langle \chi_{\p D_1}, \cdot \rangle \psi_1 + \langle  \chi_{\p D_2}, \cdot \rangle \psi_2}{\omega^2|D_1|}   + P\left(P_\alpha^\perp\right)^{-1}\frac{\langle \cdot,  \chi_{\p D_1} \rangle \psi_1^\alpha + \langle \chi_{\p D_2}, \cdot \rangle \psi_2^\alpha}{\omega^2|D_1|} +O\left(\frac{1}{\omega}\right) \\ =
%-\frac{\langle  \chi_{\p D_1}, \cdot \rangle }{\omega^2|D_1|}\left(\psi_1 + \frac{e^{-\iu \theta_\alpha}}{\sqrt{2}}\left( \frac{\left(\omega_1^\alpha\right)^2}{\omega^2-\left(\omega_1^\alpha\right)^2}\right) u_1 - \frac{e^{-\iu \theta_\alpha}}{\sqrt{2}}\left( \frac{\left(\omega_2^\alpha\right)^2}{\omega^2-\left(\omega_2^\alpha\right)^2}\right) u_2\right) \\
%-  \frac{\langle \chi_{\p D_2}, \cdot \rangle }{\omega^2|D_1|} \left(\psi_2 - \frac{1}{\sqrt{2}}\left( \frac{\left(\omega_1^\alpha\right)^2}{\omega^2-\left(\omega_1^\alpha\right)^2}\right) u_1 - \frac{1}{\sqrt{2}}\left( \frac{\left(\omega_2^\alpha\right)^2}{\omega^2-\left(\omega_2^\alpha\right)^2}\right) u_2  \right) + O\left(\frac{1}{\omega}\right).
%\end{multline}

\begin{align} \label{eq:A12}
A_{12} & \ds = -\frac{\langle \chi_{\p D_1}, \cdot \rangle \psi_1 + \langle  \chi_{\p D_2}, \cdot \rangle \psi_2}{\omega^2|D_1|}   + P\left(P_\alpha^\perp\right)^{-1}\frac{\langle \cdot,  \chi_{\p D_1} \rangle \psi_1^\alpha + \langle \chi_{\p D_2}, \cdot \rangle \psi_2^\alpha}{\omega^2|D_1|} +O\left(\frac{1}{\omega}\right) \nonumber \\
\nm
& \ds = -\frac{\langle \chi_{\p D_1}, \cdot \rangle \psi_1 + \langle  \chi_{\p D_2}, \cdot \rangle \psi_2}{\omega^2|D_1|}
- \frac{\langle \chi_1^\alpha, \cdot \rangle u_1}{\omega^2|D_1|} \left(\frac{\left(\omega_1^\alpha\right)^2}{\omega^2-\left(\omega_1^\alpha\right)^2}\right)- \frac{\langle \chi_2^\alpha, \cdot \rangle u_2}{\omega^2|D_1|}\left(\frac{\left(\omega_2^\alpha\right)^2}{\omega^2-\left(\omega_2^\alpha\right)^2}\right) \nonumber \\
\nm & \ds +O\left(\frac{1}{\omega}\right),
\end{align}
and
\begin{align}\label{eq:A22}
A_{22} & \ds = -\left(P_\alpha^\perp\right)^{-1}\frac{\langle  \chi_{\p D_1}, \cdot \rangle \psi_1^\alpha + \langle  \chi_{\p D_2}, \cdot \rangle \psi_2^\alpha}{\omega^2|D_1|} +O\left(\frac{1}{\omega}\right) \nonumber \\
\nm
& \ds =- \frac{\langle \chi_1^\alpha, \cdot \rangle u_1^\alpha}{\omega^2|D_1|} \left(\frac{\omega^2}{\omega^2-\left(\omega_1^\alpha\right)^2}\right) - \frac{\langle \chi_2^\alpha, \cdot \rangle u_2^\alpha}{\omega^2|D_1|}\left(\frac{\omega^2}{\omega^2-\left(\omega_2^\alpha\right)^2}\right)  +O\left(\frac{1}{\omega}\right),
\end{align}
with respect to the norm in $\B(L^2(\p D))$. The singularity of $\A^\alpha$ as $\omega \rightarrow \omega_1^\alpha$ or $\omega \rightarrow \omega_2^\alpha$ is, to leading order, described by the operator $P_\alpha^\perp$. Defining
\begin{equation*}\label{eq:psiphi}
\Psi^\alpha_j = \begin{pmatrix} u_j \\ u^\alpha_j \end{pmatrix} \qquad   \Phi^\alpha_j = \begin{pmatrix} -\delta u_j^\alpha \\ \chi_j^\alpha \end{pmatrix},
\end{equation*}
the above computations imply the following result.
\begin{prop}\label{prop:Asing}
	As $\omega\rightarrow \omega_j^\alpha, j=1,2$, we have
	$$\big(\A^\alpha(\omega,\delta)\big)^{-1} = -\frac{1}{2\omega_j^\alpha|D_1|} \frac{\langle\Phi_j^\alpha,\cdot \rangle \Psi_j^\alpha}{\omega-\omega_j^\alpha} + \mathcal{R}_j^\alpha(\omega),$$
	where $\mathcal{R}_j^\alpha(\omega)$ is holomorphic for $\omega$ in a neighbourhood of $\omega_j^\alpha$.
\end{prop}

\subsection{Dislocated system for small dislocation}\label{sec:smalldis}
In this section, we study the problem when a dislocation is introduced so that half of the array of resonators is translated in the $x_1$-direction. We will model the defect problem using the fictitious source superposition method \cite{defectSIAM}.

\subsubsection{Fictitious sources for dislocated resonator with a small dislocation} \label{sec:fic_small}
Here, we briefly describe the method of fictitious sources for a single translated resonator, in the asymptotic limit when the translation $d \rightarrow 0$. This will be developed for use on a dislocated array in \Cref{sec:inteq_smalldis}. Throughout this subsection, $\Omega$ denotes a bounded domain such that $\p \Omega \in \C^{1,s}$,  $\Omega_d := \Omega + d\mathbf{v}$ and $U$ is a neighbourhood of $\Omega\cup \Omega_d$.
Although this subsection is phrased for a general domain $\Omega$, we think of $\Omega$ as a pair of resonators in the dislocated array.

\begin{figure}[tbh]
	\centering
	\begin{tikzpicture}[scale=3]
	\draw[dotted, opacity=0.5] (-0.25,0.5) -- (-0.25,-0.7);
	\draw[dotted, opacity=0.5]  (1.5,0.5) -- (1.5,-0.7)node[yshift=4pt,xshift=-7pt]{};

	\pgfmathsetmacro{\rb}{0.25pt}
	\pgfmathsetmacro{\rs}{0.2pt}
	\coordinate (a) at (0.2,0);
	\coordinate (b) at (.29,0);

	\draw[double] plot [smooth cycle] coordinates {($(a)+(30:\rb)$) ($(a)+(90:\rs)$) ($(a)+(150:\rb)$) ($(a)+(210:\rs)$)  ($(a)+(270:\rb)$) ($(a)+(330:\rs)$) };
	\draw  ($(a)+(210:\rs)$) node[xshift=-6pt, yshift=-6pt]{$\Omega$};
	\node at (0.2,-0.35) {$f,g$};
	\draw[dashed] plot [smooth cycle] coordinates {($(b)+(30:\rb)$) ($(b)+(90:\rs)$) ($(b)+(150:\rb)$) ($(b)+(210:\rs)$) ($(b)+(270:\rb)$) ($(b)+(330:\rs)$) };
	\draw[|->,opacity=0.5] ($(a)+(0,0.3)$) -- ($(b)+(0.05,0.3)$) node[pos=0.5,above]{$d$};

	\begin{scope}[xshift=0.8cm]
	\coordinate (a) at (0.2,0);
	\coordinate (b) at (.29,0);

	\draw[double] plot [smooth cycle] coordinates {($(a)+(30:\rb)$) ($(a)+(90:\rs)$) ($(a)+(150:\rb)$) ($(a)+(210:\rs)$)  ($(a)+(270:\rb)$) ($(a)+(330:\rs)$) };
	\draw  ($(a)+(210:\rs)$);
	\node at (0.2,-0.35) {$f,g$};
	\draw[dashed] plot [smooth cycle] coordinates {($(b)+(30:\rb)$) ($(b)+(90:\rs)$) ($(b)+(150:\rb)$) ($(b)+(210:\rs)$) ($(b)+(270:\rb)$) ($(b)+(330:\rs)$) }  node[xshift=6pt, yshift=-6pt]{$\Omega_d$};
	\draw[|->,opacity=0.5] ($(a)+(0,0.3)$) -- ($(b)+(0.05,0.3)$) node[pos=0.5,above]{$d$};
	\end{scope}

	%\draw  plot [smooth cycle] coordinates {(-0.1,0.1) (0.2,0) (0.5,0.1) (0.3,-0.4) (0.1,-0.4)};
	%\begin{scope}[xshift=0.13cm]
	%\draw[dashed]  plot [smooth cycle] coordinates {(-0.1,0.1) (0.2,0) (0.5,0.1) (0.3,-0.4) (0.1,-0.4)} node[xshift=13pt, yshift=-10pt]{$\Omega_d$};
	%\end{scope}
	\begin{scope}[yshift=0.6cm]
	\draw [decorate,opacity=0.5,decoration={brace,amplitude=10pt}]
	(-0.25,0) -- (1.5,0) node [midway,yshift=0.6cm]{$U$};
%	\node[opacity=0.5] at (0.2,0.35) {$U$};
	\end{scope}
	\end{tikzpicture}
	\caption[Single]{A dislocated pair of resonators in the case of a small dislocation $d$.
	Legend:
	\raisebox{2pt}{\tikz{\draw[double,thick] (0,0) -- (0.5,0)}} resonator with fictitious sources,
	\raisebox{2pt}{\tikz{\draw[dashed] (0,0) -- (0.5,0)}} dislocated resonator.
	} \label{fig:DDeps}
\end{figure}
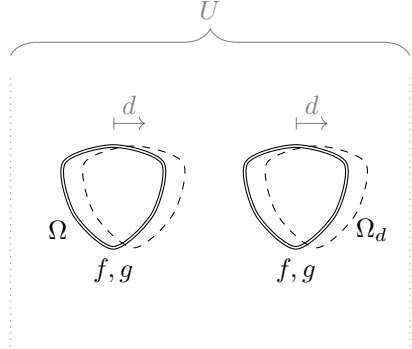

We define the maps
\begin{equation}\label{eq:pq}
	p : \p \Omega \rightarrow \p \Omega_d, \ x \mapsto x+ d \mathbf{v}, \qquad q : L^2(\p \Omega) \rightarrow L^2(\p \Omega_d), \ q(\phi)(y) = \phi(p^{-1}(y)).
\end{equation}
Since the gradient of the single-layer potential potential has a jump across $\p \Omega$, we introduce the notation
$$\nabla \S_\Omega^\omega[\phi] \big |_\pm  = \begin{cases} \nabla \S_\Omega^\omega[\phi] \big |_+  & \nu_x\cdot \mathbf{v}  \geq 0, \\ \nabla \S_\Omega^\omega[\phi] \big |_-  & \nu_x\cdot \mathbf{v}  < 0,\end{cases}$$
where $\nu_x$ is the outward unit normal to $\p \Omega$ at $x$. We remark that if $\nu_x \cdot \mathbf{v} = 0$ we have that \cite{ColtonKress}
\begin{equation}\label{eq:jumpT}
	\mathbf{v}\cdot \nabla \S_\Omega^\omega[\phi] \big |_+(x) = \mathbf{v}\cdot \nabla \S_\Omega^\omega \big|_-[\phi](x).
\end{equation}
\begin{lem} \label{lem:expTrans}
	Let $x\in \p \Omega$ and let $p$ be defined as in \eqref{eq:pq}. For $\phi \in L^2(\p \Omega)$, and for $d$ small enough, we have
		$$\S_{\Omega}^\omega[\phi](p(x)) = \S_\Omega^\omega[\phi](x) + d \mathbf{v}\cdot \nabla \S_\Omega^\omega[\phi] \big|_\pm(x) + O(d^2).$$
This estimate is valid in $L^2(\p \Omega)$ in the sense that there is a constant $C$, independent of $d$, such that
$$\left\| \S_{\Omega}^\omega[\phi]\circ p - \big(\S_\Omega^\omega[\phi] + d \mathbf{v}\cdot \nabla \S_\Omega^\omega[\phi] \big |_\pm \big)\right\|_{L^2(\p \Omega)} \leq Cd^2$$
for $d$ small enough.
\end{lem}
\begin{proof}
	We let $U_-\subset \p \Omega$ be the set of points $x$ such that $x+d_0 \mathbf{v}\in \Omega$ for all $d_0\leq d$, and $U_+$ be the set of points $x$ such that $x+d_0 \mathbf{v} \notin \Omega$ for all $d_0\leq d$. Moreover, we let $V = \p \Omega \setminus(U_+ \cup U_- )$.  Since $\p \Omega\in \C^1$ we have the following implications:
	$$x\in U_+ \implies \nu_x\cdot \mathbf{v} \geq 0 ,\quad x\in U_- \implies \nu_x\cdot \mathbf{v} \leq 0.$$
	In $U_+$ (and $U_-$), we have Taylor expansions in the $L^2$-sense (see, \emph{e.g}, \cite[Theorem 3.4.2]{weakdifffcn}), so that
	\begin{equation}\label{eq:est}
		\left\| \S_{\Omega}^\omega[\phi]\circ p - \big(\S_\Omega^\omega[\phi] + d \mathbf{v}\cdot \nabla \S_\Omega^\omega[\phi] \big |_\pm \big)\right\|_{L^2(U_\pm)} \leq C_\pm d^2
	\end{equation}
	for some constants $C_+$ and $C_-$. Moreover, we have that $\nu_x \cdot \mathbf{v} = O(d)$ uniformly for $x\in V$. From \eqref{eq:jumpT} we therefore have
	$$ \mathbf{v}\cdot \nabla \S_\Omega^\omega[\phi] \big |_+(x) = \mathbf{v}\cdot \nabla \S_\Omega^\omega[\phi] \big|_-(x) + O(d),$$
	with respect to the norm in $L^2(V)$. Therefore,
	for some constant $C_0$,
	$$\left\| \S_{\Omega}^\omega[\phi]\circ p - \big(\S_\Omega^\omega[\phi] + d \mathbf{v}\cdot \nabla \S_\Omega^\omega[\phi] \big |_\pm \big)\right\|_{L^2(V)} \leq C_0d^2,$$
	which, together with \eqref{eq:est} proves the claim.
	%We let $U_+\subset \p \Omega$ be the set of points $x$ such that $p(x) \in \Omega$ and $\nu_x\cdot \mathbf{v} > 0$, and let $U_-\subset \p \Omega$ be the set of points $x$ such that $p(x) \notin \Omega$ and $\nu_x\cdot \mathbf{v}<0$. Then, for all $d_0 < d$ we have that	Moreover, we let $V = \p \Omega \setminus(U_+ \cup U_- )$. Since $\p \Omega\in \C^1$, we have that $\nu_x \cdot \mathbf{v} = O(d)$ uniformly for $x\in V$.
	%Then, for small $d$ we have that $V$ is a neighbourhood of the points $x$ such that $\nu_x \cdot \mathbf{v} = 0$
	%For a fixed $x$, the cases $\nu_x \cdot \mathbf{v} > 0$ or $\nu_x \cdot \mathbf{v} < 0$ correspond, respectively, to the cases when $p(x)$ is inside or outside $\Omega$ for all $d$ small enough. In these cases, pointwise in $x$, the estimates follow by Taylor series expansions.
	%For a fixed $d$, we define $U\subset \p \Omega$ as the set of points $x$ such that $p(x) \notin \Omega$ but $\nu_x\cdot \mathbf{v} < 0$. Similarly, $V\subset \p \Omega$ is the set of points $x$ such that $p(x) \in  \Omega$ but $\nu_x\cdot \mathbf{v} > 0$ for some $d_0 < d$. Since $\p \Omega\in \C^1$, we have $\nu_x \cdot \mathbf{v} = O(d)$ uniformly for $x\in U\cup V$, and so
	%$$ \mathbf{v}\cdot \nabla \S_\Omega^\omega[\phi] \big |_+(x) = \mathbf{v}\cdot \nabla \S_\Omega^\omega[\phi] \big|_-(x) + O(d),$$
	%uniformly for $x\in U\cup V$. This proves the claim.
\end{proof}

We now assume $\Omega = \Omega_1 \cup \Omega_2$ for two connected domains $\Omega_i, i=1,2$. To study the problem for the dislocated resonator, we consider the problem when the resonator $\Omega$ has its original position, along with fictitious sources $f, g$ on the boundary. Explicitly, we consider the problem
\begin{equation} \label{eq:scattering_sources}
\left\{
\begin{array} {ll}
\ds \Delta \widetilde{u}+ \omega^2 \widetilde{u}  = 0 & \text{in } U\setminus \p \Omega, \\
\nm
\ds  \widetilde{u}|_{+} -\widetilde{u}|_{-}  = f  & \text{on } \partial \Omega, \\
\nm
\ds  \delta \frac{\partial \widetilde{u}}{\partial \nu} \bigg|_{+} - \frac{\partial \widetilde{u}}{\partial \nu} \bigg|_{-} =g & \text{on } \partial \Omega.
\end{array}
\right.
\end{equation}
We assume we have a reference solution $u$ satisfying
\begin{equation} \label{eq:scattering}
\left\{
\begin{array} {ll}
\ds \Delta {u}+ \omega^2 {u}  = 0 & \text{in } U\setminus \p \Omega_d, \\
\nm
\ds  {u}|_{+} -{u}|_{-}  =0  & \text{on } \p \Omega_d, \\
\nm
\ds  \delta \frac{\partial {u}}{\partial \nu} \bigg|_{+} - \frac{\partial {u}}{\partial \nu} \bigg|_{-} =0 & \text{on } \p \Omega_d.
\end{array}
\right.
\end{equation}
We want to determine the fictitious sources $f,g$ such that
\begin{align}
u&=\widetilde u \quad \text{in } U\setminus \left(\Omega\cup \Omega_d\right), \label{eq:id1} \\
u&=\widetilde u \quad \text{in } \Omega\cap \Omega_d. \label{eq:id2}
\end{align}
Inside $U$, the two solutions $u$ and  $\widetilde u$ can be respectively represented  as
\begin{equation}\label{eq:u}
u = \begin{cases} \hat\S_{\Omega_d}^{\omega}[\phi^{i,d}] & \text{in } \Omega_d, \\[0.3em]
\S_{\Omega_d}^{\omega}[\phi^{o,d}] + H & \text{in } U\setminus \Omega_d, \end{cases}
\end{equation}
and
\begin{equation} \label{eq:utilde}
\widetilde u = \begin{cases} \hat\S_{\Omega}^{\omega}[\phi^i] & \text{in } \Omega, \\[0.3em]
\S_\Omega^{\omega}[\phi^o] + \widetilde H & \text{in } U\setminus \Omega, \end{cases}
\end{equation}
for some functions $H$ and  $\widetilde H$ satisfying $\Delta H + \omega^2 H = 0$ and $\Delta \widetilde H + \omega^2 \widetilde H = 0$ in $U$. $H$ and $\widetilde H$ can be thought of as background solutions, while the single layer potentials account for the local effect of the resonators. From \eqref{eq:id1} it follows that $H=\widetilde H$. Using the jump relations and the boundary conditions in \eqref{eq:scattering} and \eqref{eq:scattering_sources} we find that
\begin{equation} \label{eq:A_dA}
\A_d(\omega,\delta) \Phi_d = \begin{pmatrix} H \big|_{\p \Omega_d} \\ \delta \p_\nu H \big|_{\p \Omega_d} \end{pmatrix}, \qquad \A(\omega,\delta) \Phi = \begin{pmatrix} H \big|_{\p \Omega} \\ \delta \p_\nu H \big|_{\p \Omega} \end{pmatrix} - \begin{pmatrix} f \\ g\end{pmatrix},
\end{equation}
where
$$\A_d(\omega,\delta) = \begin{pmatrix}
\hat\S_{\Omega_d}^{\omega} & - \S_{\Omega_d}^{\omega} \\ -\frac{1}{2}I + \hat\K_{\Omega_d}^{\omega,*} & -\delta\left(\frac{1}{2}I + (\K_{\Omega_d}^{\omega})^* \right)
\end{pmatrix}, \quad
\A(\omega,\delta) = \begin{pmatrix}
\hat\S_\Omega^{\omega} & - \S_\Omega^{\omega} \\ -\frac{1}{2}I + \hat\K_\Omega^{\omega,*}  & -\delta\left(\frac{1}{2}I + (\K_\Omega^{\omega})^* \right)
\end{pmatrix},$$
and
$$\Phi_d = \begin{pmatrix} \phi^{i,d} \\ \phi^{o,d} \end{pmatrix}, \quad \Phi = \begin{pmatrix} \phi^i \\ \phi^o \end{pmatrix}.$$
\begin{comment} OBSERVE that convexity is not actually needed.
To simplify the arguments, we make following assumption on $\Omega$:
\begin{assump} \label{assump:convex}
	We assume that $\Omega\subset \R^3 $ is convex.
\end{assump}
\Cref{assump:convex} helps to simplify the arguments. The method presented next can be extended to handle non-convex resonators by further decomposing $\p \Omega$ and $\p Omega_d$.
\begin{comment}
\begin{cases}
\S_{\Omega_d}^{\omega}[\phi^{o,d}](\widetilde x) = \S_\Omega^{\omega}[\phi^o](\widetilde x), \quad &\widetilde x \in \p \Omega_d \setminus \Omega,  \\
\S_{\Omega_d}^{\omega}[\phi^{o,d}](x) = \S_\Omega^{\omega}[\phi^o](x), \quad & x \in \p \Omega \setminus \Omega_d.
\end{cases}
From this we find
$$
\hat\S_{\Omega_d}^{\omega}[\phi^{i,d}](\widetilde x) = \begin{cases} \hat\S_\Omega^{\omega}[\phi^i](\widetilde x), \quad &\widetilde x \in \p \Omega_d^i,  \\ \S_\Omega^{\omega}[\phi^o](\widetilde x), \quad &\widetilde x \in \p \Omega_d^o.  \end{cases}
$$
Here, we have defined
$$\chi_d^i(x) = \begin{cases} 1, \quad & x\in \p \Omega^i, \\ 0, \quad & x\in \p \Omega^o,\end{cases} \qquad \chi_d^o(x) = \begin{cases} 0, \quad & x\in \p \Omega^i, \\ 1, \quad & x\in \p \Omega^o.\end{cases}$$
\end{comment}

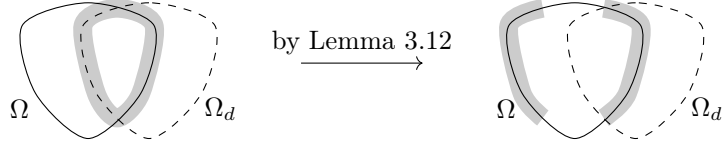
\begin{figure}%[tbh]
	\centering
	\begin{tikzpicture}[scale=4]
	\pgfmathsetmacro{\rb}{0.25pt}
	\pgfmathsetmacro{\rs}{0.2pt}
	\pgfmathsetmacro{\rspec}{0.21pt}
	\coordinate (a) at (0.2,0);
	\coordinate (b) at (.4,0);

	\draw plot [smooth cycle] coordinates {($(a)+(30:\rb)$) ($(a)+(90:\rs)$) ($(a)+(150:\rb)$) ($(a)+(210:\rs)$)  ($(a)+(270:\rb)$) ($(a)+(330:\rs)$) };
	\draw  ($(a)+(210:\rs)$) node[xshift=-6pt, yshift=-6pt]{$\Omega$};
	\draw[dashed] plot [smooth cycle] coordinates {($(b)+(30:\rb)$) ($(b)+(90:\rs)$) ($(b)+(150:\rb)$) ($(b)+(210:\rs)$) ($(b)+(270:\rb)$) ($(b)+(330:\rs)$) } node[xshift=7pt, yshift=-6pt]{$\Omega_d$};
%	\draw[|->,opacity=0.5] ($(a)+(0,0.3)$) -- ($(b)+(0.05,0.3)$) node[pos=0.5,above]{$d$};

	\draw[opacity=0.2,line width=2.5mm] plot [smooth cycle] coordinates {($(b)+(120:\rspec)$) ($(b)+(150:\rb)$) ($(b)+(210:\rs)$) ($(b)+(240:\rspec)$) ($(a)+(330:\rs)$) ($(a)+(30:\rb)$)};

	\draw[->] (0.9,0) -- (1.3,0) node[pos=0.5,above]{by \Cref{lem:expTrans}};

	\begin{scope}[xshift=1.6cm]
	\coordinate (a) at (0.2,0);
	\coordinate (b) at (.4,0);

	\draw plot [smooth cycle] coordinates {($(a)+(30:\rb)$) ($(a)+(90:\rs)$) ($(a)+(150:\rb)$) ($(a)+(210:\rs)$)  ($(a)+(270:\rb)$) ($(a)+(330:\rs)$) };
	\draw  ($(a)+(210:\rs)$) node[xshift=-6pt, yshift=-6pt]{$\Omega$};
	\draw[dashed] plot [smooth cycle] coordinates {($(b)+(30:\rb)$) ($(b)+(90:\rs)$) ($(b)+(150:\rb)$) ($(b)+(210:\rs)$) ($(b)+(270:\rb)$) ($(b)+(330:\rs)$) } node[xshift=7pt, yshift=-6pt]{$\Omega_d$};

	\draw[opacity=0.2,line width=2.5mm] plot [smooth] coordinates { ($(b)+(240:\rspec)$) ($(a)+(330:\rs)$) ($(a)+(30:\rb)$) ($(b)+(120:\rspec)$)};
	\draw[opacity=0.2,line width=2.5mm] plot [smooth] coordinates {($(a)+(120:\rspec)$) ($(a)+(150:\rb)$) ($(a)+(210:\rs)$)  ($(a)+(240:\rspec)$) };
	\end{scope}

	\end{tikzpicture}
	\caption[Single]{In the fictitious sources approach, for the case of a small dislocation, we seek solutions that match on the shaded region. In \eqref{eq:1} and \eqref{eq:2}, equality is imposed on the region highlighted in the left image. Using \Cref{lem:expTrans} this is mapped to a subset of $\p\Omega$. After this transformation, the length of the part of $\p\Omega$ not included will be $O(d)$, where $d$ is the size of the dislocation.
		Legend:
		\raisebox{2pt}{\tikz{\draw (0,0) -- (0.5,0)}} original resonator,
		\raisebox{2pt}{\tikz{\draw[dashed] (0,0) -- (0.5,0)}} dislocated resonator,
		\raisebox{0pt}{\tikz{\draw[opacity=0.2,line width=2.5mm] (0,0) -- (0.7,0)}} region of enforced equality.
	} \label{fig:disloc_labels}
\end{figure}

By equations \eqref{eq:id1} and \eqref{eq:id2}, we have
\begin{align}
\hat\S_{\Omega_d}^{\omega}[\phi^{i,d}](\widetilde x) = \hat\S_\Omega^{\omega}[\phi^i](\widetilde x), \quad &\widetilde x \in \p \Omega_d \cap \Omega,  \label{eq:1} \\
\hat\S_{\Omega_d}^{\omega}[\phi^{i,d}](x) = \hat\S_\Omega^{\omega}[\phi^i](x), \quad & x \in \p \Omega \cap \Omega_d.  \label{eq:2}
\end{align}
We decompose the boundaries of the resonators as $\p \Omega_d^i = \p \Omega_d \cap \Omega$ and $\p \Omega_d^o = \p \Omega_d \setminus \p \Omega_d^i$ and define $\p \Omega^i = \p \Omega_d^i - d \mathbf{v}$ and $\p \Omega^o = \p \Omega_d^o - d \mathbf{v}$. Because of translation invariance, we have $\hat\S_{\Omega_d}^{\omega}\left[\phi^{i,d}\right](\widetilde x) = \hat\S_\Omega^{\omega}\left[q^{-1}(\phi^{i,d})\right](x),$ where $\widetilde x = p(x)$. Therefore, using \Cref{lem:expTrans}, we obtain
$$\begin{cases}
\hat\S_\Omega^{\omega}\left[q^{-1}(\phi^{i,d})\right] =  \hat\S_\Omega^{\omega}[\phi^i] + d \mathbf{v}\cdot \nabla \hat\S_\Omega^{\omega}[\phi^i]\big |_- + O(d^2) \quad &\mathrm{on} \ \p \Omega^i, \\
\nm
\hat\S_\Omega^{\omega}\left[q^{-1}(\phi^{i,d})\right] - d \mathbf{v}\cdot \nabla \hat\S_\Omega^{\omega}\left[q^{-1}(\phi^{i,d})\right]\big |_- = \hat\S_\Omega^{\omega}\left[\phi^{i}\right] + O(d^2) \quad &\mathrm{on} \ \p \Omega \cap \Omega_d, \end{cases}
$$
where $q$ is defined in \eqref{eq:pq} and the error terms are with respect to corresponding $L^2$-norm. This transformation is depicted in \Cref{fig:disloc_labels}. The boundary $\p \Omega$ is decomposed into disjoint parts $\p \Omega^i$ and $\p \Omega^o$, and (since $\p \Omega$ is $\C^1$) the length of the ``missing'' part of the boundary, $\p \Omega^o \setminus \left( \p \Omega \cap \Omega_d\right)$, scales as $O(d)$. Moreover, on this part \eqref{eq:2} holds to order $O(d)$.
Using the Neumann series, we can invert the second equation to obtain
$$q^{-1}(\phi^{i,d}) =  \phi^i + d \left(\hat\S_\Omega^{\omega}\right)^{-1} \mathbf{v}\cdot \nabla \hat\S_\Omega^{\omega}[\phi^i]\big |_- + O(d^2),
$$
with respect to the $L^2(\p \Omega)$-norm. We define $Q: L^2(\p \Omega)^2\rightarrow L^2(\p \Omega_d)^2$ by
$$Q(u,v) = \begin{pmatrix}q(u) \\ q(v)\end{pmatrix}.$$
By analogous computations for $\S_\Omega^{\omega}[\phi^{o,d}](x)$ as those for $\hat\S_\Omega^{\omega}[\phi^{i,d}](x)$, we find that
\begin{equation} \label{eq:P1_B}
Q^{-1}\Phi_d = \P_1 \Phi, \qquad \P_1 = I + d\begin{pmatrix}
\left(\hat\S_\Omega^{\omega}\right)^{-1} \mathbf{v}\cdot \nabla \hat\S_\Omega^{\omega}\big |_- &  0  \\
0  & \left(\S_\Omega^{\omega}\right)^{-1} \mathbf{v}\cdot \nabla \S_\Omega^{\omega}\big |_+ \end{pmatrix} + O(d^2),
\end{equation}
where $\P_1: L^2(\p \Omega)^2 \rightarrow L^2(\p \Omega)^2$. We denote the linear term in $d$ by $\P_1^{(1)}$, so that $\P_1 = I + d\P_1^{(1)} + O(d^2)$ with respect to the operator norm in $\B(L^2(\p \Omega)^2)$.

We now use Taylor series expansions to relate $ H |_{\partial \Omega}$ and $ H |_{\partial \Omega_d}$. If we let $\frac{\partial}{\partial T} := (\mathbf{v}- (\mathbf{v}\cdot \nu)\nu) \cdot \nabla$ denote the tangential derivative in the direction specified by $\mathbf{v}$ we have that
\begin{align*}
H |_{\partial \Omega} &=  H |_{\partial \Omega_d} - d \mathbf{v} \cdot \nabla  H |_{\partial \Omega_d} +O(d^2) \\
 &= H |_{\partial \Omega_d} - d \left( (\mathbf{v}\cdot \nu) \frac{\partial}{\partial \nu}  H |_{\partial \Omega_d} + \frac{\partial}{\partial T}  H |_{\partial \Omega_d} \right) +O(d^2),
\end{align*}
where the error term is a continuous function of $x$ in the compact domain $\p \Omega$, and hence valid uniformly in $x$. Moreover,
$$\frac{\partial}{\partial \nu}  H |_{\partial \Omega} = \frac{\partial}{\partial \nu}  H |_{\partial \Omega_d} - d \left( (\mathbf{v}\cdot \nu)\frac{\partial^2}{\partial \nu^2}H |_{\partial \Omega_d} + \frac{\partial^2}{\partial T \partial \nu}H |_{\partial \Omega_d} \right)+O(d^2).$$
The Laplacian in local coordinates defined by the normal and tangential directions of $\partial \Omega_d$ can be written as
\begin{equation*}\label{eq:lapcurve}
\Delta = \frac{\partial^2}{\partial \nu^2} + 2\tau(\widetilde x)\frac{\partial}{\partial \nu} + \Delta_{\partial \Omega_d},
\end{equation*}
where $\tau$ denotes the mean curvature of $\p \Omega_d$ and $\Delta_{\partial \Omega_d}$ denotes the Laplace-Beltrami operator on $\p \Omega_d$. Since $H$ satisfies the Helmholtz equation $(\Delta + \omega^2)H=0$,  we get
$$ \frac{\partial^2}{\partial \nu^2}H |_{\partial \Omega_d} = -\left(\omega^2 + \Delta_{\partial \Omega_d}\right) H |_{\partial \Omega_d} - 2\tau \frac{\partial}{\partial \nu}  H |_{\partial \Omega_d}.$$
Hence, we have
\begin{equation} \label{eq:H_B}
\begin{pmatrix} H \big|_{\p \Omega} \\ \delta \p_\nu H \big|_{\p \Omega} \end{pmatrix}
=  \mathcal{P}_2 Q^{-1}
\begin{pmatrix} H \big|_{\p \Omega_d} \\ \delta \p_\nu H \big|_{\p \Omega_d} \end{pmatrix},
\end{equation}
where the operator $\mathcal{P}_2: L^2(\p \Omega)^2\rightarrow L^2(\p \Omega)^2 $ is given by
$$
\mathcal{P}_2 = I + d \P_2^{(1)} + O(d^2), \qquad \P_2^{(1)} =  \begin{pmatrix}
\ds - \p_{T}
& -\frac{(\mathbf{v}\cdot \nu)}{\delta}
\\
\delta(\mathbf{v}\cdot \nu)\left(\omega^2 + \Delta_{\partial \Omega}\right) & 2\tau - \p_{T}
\end{pmatrix},
$$
with respect to the norm in $\B(L^2(\p \Omega)^2)$. Since $\Omega_d$ and $\Omega$ only differ by a translation, we have that
\begin{equation} \label{eq:A_B}
\A_d = Q \A Q^{-1}.
\end{equation}
Combining \eqref{eq:A_dA}, \eqref{eq:P1_B}, \eqref{eq:H_B} and \eqref{eq:A_B}, we arrive at the following result.
\begin{prop} \label{prop:B}
	The layer densities $\phi^i$ and $\phi^o$ and the fictitious sources $f$ and $g$ satisfy
\begin{equation*} \label{add3}
\begin{pmatrix}
f \\ g
\end{pmatrix} = B(\omega,\delta,d) \begin{pmatrix} \phi^i \\ \phi^o \end{pmatrix}, \qquad B(\omega,\delta,d) = \P_2 \A \P_1 - \A.
\end{equation*}
\end{prop}

\subsubsection{Integral equation for the dislocated system} \label{sec:inteq_smalldis}
In this section, we use \Cref{prop:B} to derive an integral equation for the dislocated system when the dislocation size is small.

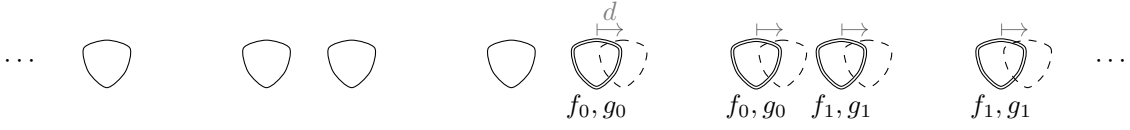
\begin{figure}%[tbh]
	\centering
	\begin{tikzpicture}[scale=1.4]
	\pgfmathsetmacro{\rb}{0.25pt}
	\pgfmathsetmacro{\rs}{0.2pt}
	\coordinate (a) at (0.25,0);
	\coordinate (b) at (1.05,0);

	\draw plot [smooth cycle] coordinates {($(a)+(30:\rb)$) ($(a)+(90:\rs)$) ($(a)+(150:\rb)$) ($(a)+(210:\rs)$) ($(a)+(270:\rb)$) ($(a)+(330:\rs)$) };
	\draw plot [smooth cycle] coordinates {($(b)+(30:\rb)$) ($(b)+(90:\rs)$) ($(b)+(150:\rb)$) ($(b)+(210:\rs)$) ($(b)+(270:\rb)$) ($(b)+(330:\rs)$) };

	\begin{scope}[xshift=-2.3cm]
	\coordinate (a) at (0.25,0);
	\coordinate (b) at (1.05,0);
	\draw plot [smooth cycle] coordinates {($(b)+(30:\rb)$) ($(b)+(90:\rs)$) ($(b)+(150:\rb)$) ($(b)+(210:\rs)$) ($(b)+(270:\rb)$) ($(b)+(330:\rs)$) };
	\draw (a) node{$\cdots$};
	\end{scope}
	\begin{scope}[xshift=2.3cm]
	\coordinate (a) at (0.25,0);
	\coordinate (b) at (1.05,0);
	\draw plot [smooth cycle] coordinates {($(a)+(30:\rb)$) ($(a)+(90:\rs)$) ($(a)+(150:\rb)$) ($(a)+(210:\rs)$) ($(a)+(270:\rb)$) ($(a)+(330:\rs)$) };
	\draw[double] plot [smooth cycle] coordinates {($(b)+(330:\rs)$) ($(b)+(30:\rb)$) ($(b)+(90:\rs)$) ($(b)+(150:\rb)$) ($(b)+(210:\rs)$) ($(b)+(270:\rb)$) } node[xshift=0pt, yshift=-7pt]{$f_0,g_0$};
	\draw[|->,opacity=0.5] ($(b)+(0,0.3)$) -- ($(b)+(0.25,0.3)$) node[pos=0.5,above]{$d$};
	\end{scope}
	\begin{scope}[xshift=4.6cm]
	\coordinate (a) at (0.25,0);
	\coordinate (b) at (1.05,0);
	\draw[double] plot [smooth cycle] coordinates {($(a)+(330:\rs)$) ($(a)+(30:\rb)$) ($(a)+(90:\rs)$) ($(a)+(150:\rb)$) ($(a)+(210:\rs)$) ($(a)+(270:\rb)$) } node[xshift=0pt, yshift=-7pt]{$f_0,g_0$};
	\draw[|->,opacity=0.5] ($(a)+(0,0.3)$) -- ($(a)+(0.25,0.3)$);
	\draw[double] plot [smooth cycle] coordinates {($(b)+(330:\rs)$) ($(b)+(30:\rb)$) ($(b)+(90:\rs)$) ($(b)+(150:\rb)$) ($(b)+(210:\rs)$) ($(b)+(270:\rb)$) } node[xshift=0pt, yshift=-7pt]{$f_1,g_1$};
	\draw[|->,opacity=0.5] ($(b)+(0,0.3)$) -- ($(b)+(0.25,0.3)$);
	\end{scope}
	\begin{scope}[xshift=6.9cm]
	\coordinate (a) at (0.25,0);
	\coordinate (b) at (1.05,0);
	\draw[double] plot [smooth cycle] coordinates {($(a)+(330:\rs)$) ($(a)+(30:\rb)$) ($(a)+(90:\rs)$) ($(a)+(150:\rb)$) ($(a)+(210:\rs)$) ($(a)+(270:\rb)$) } node[xshift=0pt, yshift=-7pt]{$f_1,g_1$};
	\draw[|->,opacity=0.5] ($(a)+(0,0.3)$) -- ($(a)+(0.25,0.3)$);
	\end{scope}

	%%% Translated versions
	\begin{scope}[xshift=0.25cm]
	\begin{scope}[xshift=2.3cm]
	\coordinate (a) at (0.25,0);
	\coordinate (b) at (1.05,0);
%	\draw plot [smooth cycle] coordinates {($(a)+(30:\rb)$) ($(a)+(90:\rs)$) ($(a)+(150:\rb)$) ($(a)+(210:\rs)$) ($(a)+(270:\rb)$) ($(a)+(330:\rs)$) };
	\draw[dashed] plot [smooth cycle] coordinates {($(b)+(30:\rb)$) ($(b)+(90:\rs)$) ($(b)+(150:\rb)$) ($(b)+(210:\rs)$) ($(b)+(270:\rb)$) ($(b)+(330:\rs)$) };
	\end{scope}
	\begin{scope}[xshift=4.6cm]
	\coordinate (a) at (0.25,0);
	\coordinate (b) at (1.05,0);
	\draw[dashed] plot [smooth cycle] coordinates {($(a)+(30:\rb)$) ($(a)+(90:\rs)$) ($(a)+(150:\rb)$) ($(a)+(210:\rs)$) ($(a)+(270:\rb)$) ($(a)+(330:\rs)$) };
	\draw[dashed] plot [smooth cycle] coordinates {($(b)+(30:\rb)$) ($(b)+(90:\rs)$) ($(b)+(150:\rb)$) ($(b)+(210:\rs)$) ($(b)+(270:\rb)$) ($(b)+(330:\rs)$) };
	%	\draw (1.7,0) node{$\cdots$};
	\end{scope}
	\begin{scope}[xshift=6.9cm]
	\coordinate (a) at (0.25,0);
	\coordinate (b) at (1.05,0);
	\draw[dashed] plot [smooth cycle] coordinates {($(a)+(330:\rs)$) ($(a)+(30:\rb)$) ($(a)+(90:\rs)$) ($(a)+(150:\rb)$) ($(a)+(210:\rs)$) ($(a)+(270:\rb)$) };
	%	\draw plot [smooth cycle] coordinates {($(b)+(30:\rb)$) ($(b)+(90:\rs)$) ($(b)+(150:\rb)$) ($(b)+(210:\rs)$) ($(b)+(270:\rb)$) ($(b)+(330:\rs)$) };
	\draw (b) node{$\cdots$};
	\end{scope}
	\end{scope}

%	\begin{scope}[xshift=-1.5cm,yshift=0.8cm]
%	\draw (0,0) -- (0.5,0) node[pos=1,right] {\small Untouched resonator};
%	\draw[double] (3,0) -- (3.5,0) node[pos=1,right] {\small Resonator with sources};
%	\draw[dashed] (6.3,0) -- (6.8,0) node[pos=1,right] {\small Dislocated resonator};
%	\end{scope}

	\end{tikzpicture}
	\caption[Dislocated]{The dislocated system is equivalent to the original array with the addition of so-called fictitious sources $f_m$, $g_m$, on the boundary of $D^m$ for $m\in\N$. Legend:
		\raisebox{2pt}{\tikz{\draw (0,0) -- (0.5,0)}} untouched resonator,
		\raisebox{2pt}{\tikz{\draw[double,thick] (0,0) -- (0.5,0)}} resonator with fictitious sources,
		\raisebox{2pt}{\tikz{\draw[dashed] (0,0) -- (0.5,0)}} dislocated resonator.
	 } \label{fig:disloc_fictsources}
\end{figure}

To study the dislocated problem \eqref{eq:scattering_translated}, we consider the problem with periodic geometry, along with fictitious sources $f_m,g_m$ placed on the boundary of $D^m = D_1^m \cup D_2^m$. Explicitly, we consider the problem
\begin{equation} \label{eq:scattering_translated_sources}
\left\{
\begin{array} {ll}
\ds \Delta \widetilde{u}+ \omega^2 \widetilde{u}  = 0 & \text{in } \R^3 \setminus \p \C_0, \\
\nm
\ds  \widetilde{u}|_{+} -\widetilde{u}|_{-}  =f_m  & \text{on } \p D^m, m\in\N, \\
\nm
\ds  \delta \frac{\partial \widetilde{u}}{\partial \nu} \bigg|_{+} - \frac{\partial \widetilde{u}}{\partial \nu} \bigg|_{-} =g_m & \text{on } \partial D^m, m\in\N, \\
\nm
\ds \widetilde u(x_1,x_2,x_3) & \text{satisfies the outgoing radiation condition as } \sqrt{x_2^2+x_3^2} \rightarrow \infty.
\end{array}
\right.
\end{equation}
Assume we have a non-zero solution $u$ to \eqref{eq:scattering_translated}. Inside $Y^m := Y + md\mathbf{v}$, we can represent the solution as in \eqref{eq:u} with the choices $\Omega = D^m$ and $U=Y^m$. In this way we define the layer densities $\phi^{i,d}$ and $\phi^{o,d}$. Since $\P_1$ is invertible for small enough $d$, we can define the layer densities $\phi^i_m$ and  $\phi^o_m$ as
$$
\begin{pmatrix} \phi^i_m \\ \phi^o_m \end{pmatrix}
= \left(\P_1\right)^{-1} Q^{-1} \begin{pmatrix} \phi^{i,d} \\ \phi^{o,d} \end{pmatrix}.
$$
We then set the fictitious sources as
\begin{equation}\label{eq:sources}
\begin{pmatrix}
f_m \\ g_m
\end{pmatrix} =\begin{cases} 0, \quad &m < 0,\\
 B^m\begin{pmatrix} \phi^i_m \\ \phi^o_m \end{pmatrix}, \quad &m \geq 0,\end{cases}
\end{equation}
where $B^m$ is defined as in \Cref{prop:B} with the choice $\Omega = D^m$. We then define the solution $\widetilde u$ by \eqref{eq:utilde} with $\widetilde H = H$, and because of \eqref{eq:sources} this coincides with $u$ in $\Big(Y^m \setminus \big(D^m \cup (D^m +d\mathbf{v})\big) \Big)\cup \big(D^m \cap (D^m +d\mathbf{v})\big)$.

Conversely, if we have a non-zero solution $\widetilde u$ to \eqref{eq:scattering_translated_sources}, represented as \eqref{eq:utilde} in $Y^m$ and with sources satisfying \eqref{eq:sources}, we can define  $\phi^{i,d}$ and $\phi^{o,d}$ to get a non-zero solution $u$ to \eqref{eq:scattering_translated} coinciding with $\widetilde u$ in $\Big(Y^m \setminus \big(D^m \cup (D^m +d\mathbf{v})\big) \Big)\cup \big(D^m \cap (D^m +d\mathbf{v})\big)$.

From the above arguments, it follows that the spectral problem \eqref{eq:scattering_translated} is equivalent to \eqref{eq:scattering_translated_sources}. So, in the remainder of this subsection we will only study the latter problem. For simplicity, since the solutions coincide, we will omit the superscript $\ \widetilde{} \ $ and simply write $u$ for $\widetilde u$.

We define $u^\alpha$ as the Floquet transform of $u$, \ie{},
$$
u^\alpha = \sum_{m\in \Z} u(x-mL\mathbf{v}) e^{\iu \alpha m}.
$$
The transformed solution $u^\alpha$ satisfies
\begin{equation} \label{eq:scattering_quasi}
\left\{
\begin{array} {ll}
\ds \Delta u^\alpha+ \omega^2 u^\alpha  = 0 & \text{in } Y \setminus \p D, \\
\nm
\ds  u^\alpha|_{+} -u^\alpha|_{-}  =f^\alpha  & \text{on } \p D, \\
\nm
\ds  \delta \frac{\partial u^\alpha}{\partial \nu} \bigg|_{+} - \frac{\partial u^\alpha}{\partial \nu} \bigg|_{-} =g^\alpha & \text{on } \partial D, \\
\nm
\ds e^{-\iu  \alpha x_1}  u^\alpha(x_1,x_2,x_3)  \,\,\,&  \mbox{is periodic in } x_1, \\
\nm
\ds u^\alpha(x_1,x_2,x_3)& \text{satisfies the $\alpha$-quasiperiodic outgoing radiation condition} \\ &\hspace{0.5cm} \text{as } \sqrt{x_2^2+x_3^2} \rightarrow \infty,
\end{array}
\right.
\end{equation}
where
\begin{equation}\label{eq:fg}
f^\alpha =  \sum_{m\in \Z} f_me^{-\iu \alpha m}, \qquad g^\alpha =  \sum_{m\in \Z} g_me^{-\iu \alpha m}.
\end{equation}
From now on, we identify functions $u_m \in L^2(\p D^m)$, for any $m$, with $u_0\in L^2(\p D)$ by translating the argument, \ie{}, by $u_0(x) = u_m(x+mL\mathbf{v})$. Observe that under this identification, all operators $B^m, m \in \N$ coincide and will be denoted by $\B_0$.

The solution $u^\alpha$ can be represented using quasiperiodic layer potentials as
\begin{align*}
u^\alpha
=
\begin{cases}
\hat\S_D^{\omega}[\phi^{\alpha,i}] &\quad \mbox{in } D,
\\
\S_D^{\alpha,\omega}[\phi^{\alpha,o}] &\quad \mbox{in } Y\setminus \overline{D},
\end{cases}
\end{align*}
where the pair $(\phi^{\alpha,i},\phi^{\alpha,o})\in L^2(\p D)^2$ is the solution to
\begin{equation}\label{eq:phipsialpha}
\mathcal{A}^\alpha(\omega,\delta)
\begin{pmatrix}
	\phi^{\alpha,i}
	\\[0.3em]
	\phi^{\alpha,o}
\end{pmatrix}
=
\begin{pmatrix}
\hat\S_D^{\omega} & - \S_D^{\alpha,\omega}
\\[0.3em]
\ds -\frac{1}{2}I +\hat\K_D^{\omega,*}
&
\ds -\delta \left( \frac{1}{2}I +\left(\K_D^{-\alpha,\omega}\right)^* \right)
\end{pmatrix}
\begin{pmatrix}
\phi^{\alpha,i}
\\[0.3em]
\phi^{\alpha,o}
\end{pmatrix}
=
\begin{pmatrix}
- f^\alpha
\\[0.3em]
- g^\alpha
\end{pmatrix}.
\end{equation}
Then the original solution $u$ can be recovered by the inverse Floquet transform,
\begin{equation*}
u(x)=\frac{1}{2\pi}\int_{Y^*} u^\alpha(x) \dx \alpha.
\end{equation*}
Because of the quasiperiodicity of $u^\alpha$, the solution $u$ inside the region $D^m$ satisfies
\begin{equation} \label{eq:psijm}
u= \hat\S_{D^m}^{\omega}\left[\frac{1}{2\pi}\int_{Y^*}e^{\iu \alpha m}\phi^{\alpha,i} \dx \alpha \right].
\end{equation}
Similarly, inside the region $Y^m\setminus \overline{D^m}$, we have
\begin{align} \label{eq:phijm}
\ds
u&=\frac{1}{2\pi}\int_{Y^*}\S_D^{\alpha,\omega}[\phi^{\alpha,o}] \dx \alpha \nonumber \\
&=\S_{D^m}^{\omega}\left[\frac{1}{2\pi}\int_{Y^*}e^{\iu \alpha m}\phi^{\alpha,o}  \dx \alpha\right]  + \frac{1}{2\pi}\int_{Y^*} \sum_{n\in\mathbb{Z},n\neq m}\S_D^{\omega}[\phi^{\alpha,o}](\,\cdot - nL\mathbf{v})e^{\iu n\alpha} \dx \alpha.
\end{align}
The last term in the right-hand side of (\ref{eq:phijm}) satisfies the homogeneous Helmholtz equation $(\Delta + \omega^2 )u=0$ in $Y^m$. Therefore, combining \eqref{eq:psijm} and \eqref{eq:phijm} together with  \eqref{eq:utilde}, we can identify $\phi^i=\phi^i_m$, $\phi^o=\phi^o_m$ and $\widetilde H$ as follows:
\begin{equation}\label{eq:psimphim}
\phi_m^{i} = \frac{1}{2\pi}\int_{Y^*}e^{\iu \alpha m} \phi^{\alpha,i} \dx \alpha,
\quad
\phi_m^{o} = \frac{1}{2\pi}\int_{Y^*}e^{\iu \alpha m} \phi^{\alpha,o} \dx \alpha,
\end{equation}
and
$$\widetilde H = \frac{1}{2\pi}\int_{Y^*} \sum_{n\in\mathbb{Z},n\neq m}\S_D^{\omega}[\phi^{\alpha,o}](\,\cdot - nL\mathbf{v})e^{\iu n\alpha} \dx \alpha.$$
We define the operator $I_m: L^2(\p D\times Y^*) \rightarrow L^2(\p D),$ by
$$I_m[\varphi](x) = \frac{1}{2\pi} \int_{Y^*}\varphi(x,\alpha) e^{\iu \alpha m} \dx \alpha.$$

Since the operator $\mathcal{A}_\alpha$ is invertible for $\omega$ in the band gap,
we have from \eqref{eq:phipsialpha} that
$$
\begin{pmatrix}
\phi^{\alpha,i}
\\[0.3em]
\phi^{\alpha,o}
\end{pmatrix}
= \mathcal{A}^\alpha(\omega,\delta)^{-1}\begin{pmatrix}
- f^\alpha
\\[0.3em]
- g^\alpha
\end{pmatrix}.
$$
Combining this together with \eqref{eq:psimphim} and \eqref{eq:fg}, we obtain the following result.
\begin{prop}
	For small enough $d > 0 	$, the mid-gap frequencies of \eqref{eq:scattering_translated} are precisely the values $\omega$ such that there is a non-zero solution $\phi^{\alpha,i}, \phi^{\alpha,o} \in L^2(\p D\times Y^*)$ to the equation
\begin{equation} \label{eq:int_eq}
\begin{pmatrix}
\phi^{\alpha,i}
\\[0.3em]
\phi^{\alpha,o}
\end{pmatrix}
= -\big(\mathcal{A}^\alpha(\omega,\delta)\big)^{-1}\left(\sum_{m=0}^\infty e^{-\iu m\alpha} \B_0 I_m\right)
\begin{pmatrix}
\phi^{\alpha,i}
\\[0.3em]
\phi^{\alpha,o}
\end{pmatrix}.
\end{equation}
\end{prop}
It is clear that $\B_0 = O(d)$. As $d\rightarrow 0$, it follows from \Cref{prop:Asing} that any subwavelength resonant frequency $\omega = \omega(d)$ satisfies $\omega(d) \rightarrow \omega_j^\alpha$ for some $\omega_j^\alpha$. Denote
$$\omega_1^\diamond = \max_{\alpha \in Y^*} \Re(\omega_1^\alpha), \qquad \omega_2^\diamond = \min_{\alpha \in Y^*} \Re(\omega_2^\alpha).$$
The following lemma follows from \Cref{thm:bandgap}.
\begin{lem}\label{prop:band}
	The critical values $\omega_1^\diamond$ and $\omega_2^\diamond$ are attained at $\alpha^\diamond = \pi/L$. Further, for $\alpha$ close to $\alpha^\diamond$ we have
	$$\omega_1^\alpha = \omega_1^\diamond - c_1(\alpha-\alpha^\diamond)^2 + O\left(|\alpha-\alpha^\diamond|^3 \right), \qquad \omega_2^\alpha = \omega_2^\diamond - c_2(\alpha-\alpha^\diamond)^2 + O\left(|\alpha-\alpha^\diamond|^3\right),$$
	for some constants $c_1, c_2$.
\end{lem}
In what follows, we will consistently use the superscript $^\diamond$ to denote corresponding quantity evaluated at the critical point $\alpha^\diamond = \pi/L$.
\begin{lem}\label{lem:C}
	Assume that $D_1$ and $D_2$ are strictly convex. Then, in the dilute regime, we have the following:
	\begin{align*}
		\text{Case} \ l_0 < 1/2: \qquad \langle \Phi_1^\diamond, \B_0 \Psi_1^\diamond\rangle < 0 \quad \text{and} \quad \langle \Phi_2^\diamond, \B_0 \Psi_2^\diamond\rangle > 0, \\
		\text{Case} \ l_0 > 1/2: \qquad \langle \Phi_1^\diamond, \B_0 \Psi_1^\diamond\rangle > 0 \quad \text{and} \quad \langle \Phi_2^\diamond, \B_0 \Psi_2^\diamond\rangle < 0,
	\end{align*}
	for small enough $\epsilon, \delta$ and $d$.
\end{lem}
We refer to \Cref{app:lem1} for the proof of Lemma \ref{lem:C}. We will also need the following lemma.
\begin{lem}\label{lem:intsum}
	We have
	$$\mathrm{Re}\left(\frac{1}{2\pi}\sum_{m=0}^\infty e^{-\iu m\alpha}\int_0^{2\pi} \frac{e^{\iu m\alpha'}}{1+c^2(\alpha'-\pi)^2}\dx \alpha'\right) = \frac{1}{2}\left(\frac{1}{1+c^2(\alpha-\pi)^2} + \frac{1}{\pi c}\arctan(\pi c)\right).$$
\end{lem}
\begin{proof}
	Define $I(\alpha)$ as
	$$
	I(\alpha) = \frac{1}{2\pi}\sum_{m=0}^\infty e^{-\iu m\alpha}\int_0^{2\pi} \frac{e^{\iu m\alpha'}}{1+c^2(\alpha'-\pi)^2}\dx \alpha'.
	$$
	Completing the Fourier series, we have
	$$I(\alpha) + \overline{I(\alpha)} - \frac{1}{2\pi}\int_0^{2\pi} \frac{1}{1+c^2(\alpha'-\pi)^2}\dx \alpha' = \frac{1}{1+c^2(\alpha-\pi)^2}.$$
	Since $I(\alpha) + \overline{I(\alpha)} = 2\mathrm{Re} (I(\alpha))$, and since
	$$ \frac{1}{2\pi}\int_0^{2\pi} \frac{1}{1+c^2(\alpha'-\pi)^2}\dx \alpha' = \frac{1}{\pi c} \arctan(\pi c),$$
	the lemma follows.
\end{proof}From \Cref{lem:intsum} we find that
\begin{equation} \label{eq:intsum}
\frac{1}{2\pi}\sum_{m=0}^\infty (-1)^m\int_0^{2\pi} \frac{e^{\iu m\alpha'}}{1+c^2(\alpha'-\pi)^2}\dx \alpha' = \frac{1}{2} + \frac{1}{2\pi c}\arctan(\pi c).\end{equation}
%Define $$\B^m = \begin{pmatrix}\left(B_1^m\right)_{11} & 0 & \left(B_1^m\right)_{12} & 0 \\[0.3em]0 & \left(B_2^m\right)_{11} & 0 & \left(B_2^m\right)_{12} \\[0.3em]\left(B_1^m\right)_{21} & 0 & \left(B_1^m\right)_{22}& 0 \\[0.3em]0 & \left(B_2^m\right)_{21} & 0 & \left(B_2^m\right)_{22} \\[0.3em]\end{pmatrix},$$  The following proposition is the main result of this section.

The next theorem, which is the main result of this section, describes how the mid-gap frequencies emerge from the edges of the band gap. At this point, we observe that any mid-gap frequency is necessarily real-valued. This can be seen from \eqref{eq:scattering_quasi}: a mid-gap frequency $\omega$ is a solution to this equation for any $\alpha \in Y^*$. At $\alpha = \pi/L$, and small enough $\delta$, this correspond to a self-adjoint spectral problem, and it is clear that any subwavelength resonant frequency is real-valued.
\begin{thm}\label{prop:smalld}
Assume that $D_1$ and $D_2$ are strictly convex. For small enough $d$ and $\delta$, and in the case $l_0 > 1/2$, there are two mid-gap frequencies $\omega_1(d), \omega_2(d)$ such that $\omega_j(d) \rightarrow \omega_j^\diamond, j=1,2$ as $d\rightarrow 0$. In the case $l_0 < 1/2$, there are no mid-gap frequencies as $d, \delta \rightarrow 0$.
\end{thm}
\begin{proof}
We seek solutions to \eqref{eq:int_eq} as $d \rightarrow 0$, corresponding to solutions $\omega$ in a small neighbourhood of  $\omega_j^\diamond$ for $j=1$ or $j=2$. By \Cref{prop:Asing} and \Cref{prop:band}, $\left(\A^\alpha\right)^{-1}$ has a pole at $\omega^\diamond$. Recall that we seek solutions $\omega =O(\sqrt{\delta})$ as $\delta \rightarrow 0$. At $\delta = 0$ and $\omega = 0$, the problem \eqref{eq:scattering_translated} decouples into a Neumann problem on each resonator, with constant solution inside each resonator. Since $\hat\S_D^0[u_j]$ is constant inside $D$, we find that
$$ \phi^{\alpha,i} = c_1(\alpha)u_1 + c_2(\alpha)u_2$$
for some coefficients $c_1(\alpha)$ and $c_2(\alpha)$. It follows that the root function is such that the singularity of $(\A^\alpha)^{-1}$ does not vanish. Hence, from \Cref{prop:Asing} and \Cref{prop:band}, we can find a non-zero $h = h(\omega,\delta,d)$ such that the solution can be written, for $\alpha$ close to $\alpha^\diamond$, as
$$\begin{pmatrix}
\phi^{\alpha,i}\\\phi^{\alpha,o}
\end{pmatrix} = \frac{\Psi_j^\diamond}{\omega - \omega_j^\diamond + c_j|\alpha-\alpha^\diamond|^2}h(\omega,\delta,d) + K_1(\omega,\alpha,\delta,d),$$
where $K_1(\omega,\alpha)$ is bounded uniformly in $d$ for $(\omega,\alpha)$ in a neighbourhood of $(\omega_j^\diamond,\alpha^\diamond)$. Applying \eqref{eq:intsum}, we then find that
$$\sum_{m=0}^\infty e^{-\iu \alpha m} \B_0 I_m\begin{pmatrix}
\phi^{\alpha,i}\\\phi^{\alpha,o}
\end{pmatrix} =  \frac{\B_0\Psi_j^\diamond}{2(\omega-\omega_j^\diamond)}h(\omega,\delta,d) + K_2$$
for some $K_2$ with norm of order $O(d)$ in a neighbourhood of $(\omega_j^\diamond,\alpha^\diamond)$. We then have
$$
-\big(\mathcal{A}^\alpha(\omega,\delta)\big)^{-1}\sum_{m=0}^\infty e^{\iu \alpha m}
\B_0 I_m
\begin{pmatrix}
\phi^{\alpha,i}\\\phi^{\alpha,o}
\end{pmatrix} = \frac{\Psi_j^\diamond}{\omega - \omega_j^\diamond + c_j|\alpha-\alpha^\diamond|^2}\frac{\langle \Phi_j^\diamond, \B_0\Psi_j^\diamond\rangle  }{4\omega_j^\diamond|D_1|(\omega-\omega_j^\diamond)}h(\omega,\delta,d) + K_3.
$$
Equation \eqref{eq:int_eq} then reads
$$\frac{\langle \Phi_j^\diamond, \B_0\Psi_j^\diamond\rangle}{4\omega_j^\diamond|D_1|(\omega-\omega_j^\diamond)} = 1 + O\left(\frac{d}{\sqrt{\omega-\omega_j^\diamond}}\right),$$
which has precisely one solution $\omega=\omega_j(d)$, expanded as
$$\omega_j(d)= \omega_j^\diamond + \frac{\langle \Phi_j^\diamond, \B_0\Psi_j^\diamond\rangle}{4\omega_j^\diamond|D_1|} + O(d^{3/2}).$$
From \Cref{lem:C} it follows that $\omega_j(d)$ is inside the band gap precisely in the case $l_0 > 1/2$.
\end{proof}

\begin{rmk}
	It should be noted that the assumption of convexity made in this section is not an intrinsic part of the fictitious source method. This assumption was only needed for the arguments in the proof of \Cref{lem:C}. Indeed, the fictitious source method is repeatedly used in the rest of this work without any assumption of convexity.
\end{rmk}

\subsection{Integer unit length dislocation}\label{sec:compact}
In this section, we study the problem when the dislocation is an integer number of unit cell lengths. This is equivalent to the case when an integer multiple of dimers are removed from the original, periodic structure, thus creating a cavity.  We will model this defect cavity problem using the fictitious source superposition method \cite{defectSIAM}.

\subsubsection{Fictitious sources for a removed resonator}
Here, we describe the method of fictitious sources when a single resonator is removed. Throughout this subsection, $\Omega$ denotes a connected, bounded domain such that $\p \Omega \in \C^{1,s}$ and $U$ denotes a neighbourhood of $\Omega$. Although the argument can be made for general $\Omega$, we assume that $\Omega$ consists of two connected components $\Omega = \Omega_1 \cup \Omega_2$.

To study this problem, we consider the problem when the removed resonator $\Omega$ is reintroduced, along with fictitious dipole sources $g$ on the boundary. %We will see that no monopole sources are required to model this problem. Explicitly, we consider the problem
We assume we have a reference solution $u$ satisfying
$$\Delta {u}+ \omega^2 {u}  = 0 \qquad \text{in } U.$$
Let $\widetilde{u}$ satisfy the fictitious source problem
\begin{equation*} \label{eq:scattering_sources2}
\left\{
\begin{array} {ll}
\ds \Delta \widetilde{u}+ \omega^2 \widetilde{u}  = 0 & \text{in } U \setminus \p \Omega, \\
\nm
\ds  \widetilde{u}|_{+} -\widetilde{u}|_{-}  =0  & \text{on } \partial \Omega, \\
\nm
\ds  \delta \frac{\partial \widetilde{u}}{\partial \nu} \bigg|_{+} - \frac{\partial \widetilde{u}}{\partial \nu} \bigg|_{-} =g & \text{on } \p \Omega.
\end{array}
\right.
\end{equation*}
We want to determine the fictitious sources $g$ such that $u=\widetilde u$ inside $U$. Any solution $\widetilde u$ can be represented as
\begin{equation}\label{eq:utilde2}
\widetilde u = \begin{cases} \hat\S_\Omega^{\omega}[\phi^i] & \text{in } \Omega, \\ \S_\Omega^{\omega}[\phi^o] + H & \text{in } U\setminus \Omega, \end{cases}
\end{equation}
for some $H$ satisfying $\Delta  H+ \omega^2 H  = 0$ in $U$. Imposing $\widetilde{u} = u$ in $U$ is equivalent to
$$
\phi^i = \left(\hat\S_\Omega^\omega\right)^{-1}[u|_{\p \Omega}], \qquad \phi^o = 0, \qquad H = u.
$$
Moreover, using the jump conditions, we find the following expression of $g$.
\begin{prop} \label{prop:Bc}
	The fictitious sources $g$ and the layer density $\phi^i$ satisfy
	\begin{equation*}
	g = B(\omega,\delta) \phi^i, \quad B(\omega,\delta) = \left(\delta - 1\right) \left( -\frac{1}{2}I + \hat\K_\Omega^{\omega,*} \right).
	%g = B(\omega,\delta) \psi, \quad B(\omega,\delta) = \delta \left( -\frac{1}{2}I + \K_D^{k,*} \right)\left(\S_D^k\right)^{-1} \S_\Omega^{k_b} -  \left(-\frac{1}{2}I + \K_\Omega^{k_b,*}\right).
	\end{equation*}
\end{prop}
Conversely, by the unique continuation property of the Helmholtz equation, if $g$ satisfies \Cref{prop:Bc}, then $\widetilde{u} = u$ in $U$.

\subsubsection{Integral equation for the dislocated system}\label{sec:inteq_compact}

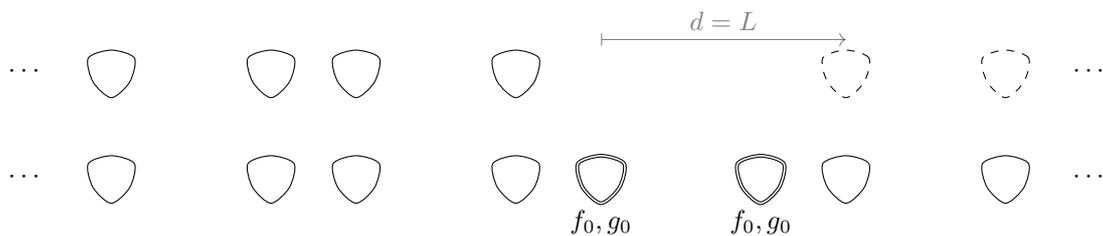
\begin{figure}[tbh]
	\centering
	\begin{tikzpicture}[scale=1.4]
	\pgfmathsetmacro{\rb}{0.25pt}
	\pgfmathsetmacro{\rs}{0.2pt}
	\coordinate (a) at (0.25,0);
	\coordinate (b) at (1.05,0);

	\draw plot [smooth cycle] coordinates {($(a)+(30:\rb)$) ($(a)+(90:\rs)$) ($(a)+(150:\rb)$) ($(a)+(210:\rs)$) ($(a)+(270:\rb)$) ($(a)+(330:\rs)$) };
	\draw plot [smooth cycle] coordinates {($(b)+(30:\rb)$) ($(b)+(90:\rs)$) ($(b)+(150:\rb)$) ($(b)+(210:\rs)$) ($(b)+(270:\rb)$) ($(b)+(330:\rs)$) };

	\begin{scope}[xshift=-2.3cm]
	\coordinate (a) at (0.25,0);
	\coordinate (b) at (1.05,0);
	\draw plot [smooth cycle] coordinates {($(b)+(30:\rb)$) ($(b)+(90:\rs)$) ($(b)+(150:\rb)$) ($(b)+(210:\rs)$) ($(b)+(270:\rb)$) ($(b)+(330:\rs)$) };
	\draw (a) node{$\cdots$};
	\end{scope}
	\begin{scope}[xshift=2.3cm]
	\coordinate (a) at (0.25,0);
	\coordinate (b) at (1.05,0);
	\draw plot [smooth cycle] coordinates {($(a)+(30:\rb)$) ($(a)+(90:\rs)$) ($(a)+(150:\rb)$) ($(a)+(210:\rs)$) ($(a)+(270:\rb)$) ($(a)+(330:\rs)$) };
%	\draw plot [smooth cycle] coordinates {($(b)+(330:\rs)$) ($(b)+(30:\rb)$) ($(b)+(90:\rs)$) ($(b)+(150:\rb)$) ($(b)+(210:\rs)$) ($(b)+(270:\rb)$) };
	\draw[|->,opacity=0.5] ($(b)+(0,0.3)$) -- ($(b)+(2.3,0.3)$) node[pos=0.5,above]{$d=L$};
	\end{scope}
	\begin{scope}[xshift=4.6cm]
	\coordinate (a) at (0.25,0);
	\coordinate (b) at (1.05,0);
%	\draw plot [smooth cycle] coordinates {($(a)+(330:\rs)$) ($(a)+(30:\rb)$) ($(a)+(90:\rs)$) ($(a)+(150:\rb)$) ($(a)+(210:\rs)$) ($(a)+(270:\rb)$) };
	\draw[dashed] plot [smooth cycle] coordinates {($(b)+(330:\rs)$) ($(b)+(30:\rb)$) ($(b)+(90:\rs)$) ($(b)+(150:\rb)$) ($(b)+(210:\rs)$) ($(b)+(270:\rb)$) };
	\end{scope}
	\begin{scope}[xshift=6.9cm]
	\coordinate (a) at (0.25,0);
	\coordinate (b) at (1.05,0);
	\draw[dashed] plot [smooth cycle] coordinates {($(a)+(330:\rs)$) ($(a)+(30:\rb)$) ($(a)+(90:\rs)$) ($(a)+(150:\rb)$) ($(a)+(210:\rs)$) ($(a)+(270:\rb)$) };
	\draw (b) node{$\cdots$};
	\end{scope}

	%%%
	\begin{scope}[yshift=-1cm]
	\coordinate (a) at (0.25,0);
	\coordinate (b) at (1.05,0);

	\draw plot [smooth cycle] coordinates {($(a)+(30:\rb)$) ($(a)+(90:\rs)$) ($(a)+(150:\rb)$) ($(a)+(210:\rs)$) ($(a)+(270:\rb)$) ($(a)+(330:\rs)$) };
	\draw plot [smooth cycle] coordinates {($(b)+(30:\rb)$) ($(b)+(90:\rs)$) ($(b)+(150:\rb)$) ($(b)+(210:\rs)$) ($(b)+(270:\rb)$) ($(b)+(330:\rs)$) };

	\begin{scope}[xshift=-2.3cm]
	\coordinate (a) at (0.25,0);
	\coordinate (b) at (1.05,0);
	\draw plot [smooth cycle] coordinates {($(b)+(30:\rb)$) ($(b)+(90:\rs)$) ($(b)+(150:\rb)$) ($(b)+(210:\rs)$) ($(b)+(270:\rb)$) ($(b)+(330:\rs)$) };
	\draw (a) node{$\cdots$};
	\end{scope}
	\begin{scope}[xshift=2.3cm]
	\coordinate (a) at (0.25,0);
	\coordinate (b) at (1.05,0);
	\draw plot [smooth cycle] coordinates {($(a)+(30:\rb)$) ($(a)+(90:\rs)$) ($(a)+(150:\rb)$) ($(a)+(210:\rs)$) ($(a)+(270:\rb)$) ($(a)+(330:\rs)$) };
	\draw[double] plot [smooth cycle] coordinates {($(b)+(330:\rs)$) ($(b)+(30:\rb)$) ($(b)+(90:\rs)$) ($(b)+(150:\rb)$) ($(b)+(210:\rs)$) ($(b)+(270:\rb)$) } node[xshift=0pt, yshift=-7pt]{$f_0,g_0$};
	\end{scope}
	\begin{scope}[xshift=4.6cm]
	\coordinate (a) at (0.25,0);
	\coordinate (b) at (1.05,0);
	\draw[double] plot [smooth cycle] coordinates {($(a)+(330:\rs)$) ($(a)+(30:\rb)$) ($(a)+(90:\rs)$) ($(a)+(150:\rb)$) ($(a)+(210:\rs)$) ($(a)+(270:\rb)$) } node[xshift=0pt, yshift=-7pt]{$f_0,g_0$};
	\draw plot [smooth cycle] coordinates {($(b)+(330:\rs)$) ($(b)+(30:\rb)$) ($(b)+(90:\rs)$) ($(b)+(150:\rb)$) ($(b)+(210:\rs)$) ($(b)+(270:\rb)$) };
	\end{scope}
	\begin{scope}[xshift=6.9cm]
	\coordinate (a) at (0.25,0);
	\coordinate (b) at (1.05,0);
	\draw plot [smooth cycle] coordinates {($(a)+(330:\rs)$) ($(a)+(30:\rb)$) ($(a)+(90:\rs)$) ($(a)+(150:\rb)$) ($(a)+(210:\rs)$) ($(a)+(270:\rb)$) };
	\draw (b) node{$\cdots$};
	\end{scope}
	\end{scope}

%	\begin{scope}[xshift=-1.5cm,yshift=0.8cm]
%	\draw (0,0) -- (0.5,0) node[pos=1,right] {\small Untouched resonator};
%	\draw[double] (3,0) -- (3.5,0) node[pos=1,right] {\small Resonator with sources};
%	\draw[dashed] (6.3,0) -- (6.8,0) node[pos=1,right] {\small Dislocated resonator};
%	\end{scope}

	\end{tikzpicture}
	\caption[Unit dislocation]{The dislocated system with dislocation equal to a multiple of the length of the unit cell (\ie{} $d=NL$) is equivalent to the original array with the addition of so-called fictitious sources $f_m$, $g_m$, on the boundary of $D^m$ for $m=0,\dots,N-1$. The case $N=1$ is depicted here. Legend:
		\raisebox{2pt}{\tikz{\draw (0,0) -- (0.5,0)}} untouched resonator,
		\raisebox{2pt}{\tikz{\draw[double,thick] (0,0) -- (0.5,0)}} resonator with fictitious sources,
		\raisebox{2pt}{\tikz{\draw[dashed] (0,0) -- (0.5,0)}} dislocated resonator. } \label{fig:disloc_removed}
\end{figure}

We now assume that $2N$ resonators are removed, so that $u$ satisfies \eqref{eq:scattering_translated} with
\begin{equation} \label{eq:CN}
\C_d = \C_0 \setminus \left(\bigcup_{m=0}^{N-1} D^m \right).
\end{equation}
Again, we model this using the fictitious source method as in \eqref{eq:scattering_translated_sources}, following the approach of \Cref{sec:inteq_smalldis}. We put $f_m = 0$ for all $m$. Moreover, $g_m$ will be defined as in \Cref{prop:Bc} for all the removed resonators.

Assume we have a non-zero solution $u$ to \eqref{eq:scattering_translated}. Inside $Y^m, m = 0,1,\cdots,N-1$, we can define the layer density $\phi^i_m$ as
$$
\phi^i_m  = \big(\hat\S_\Omega^\omega\big)^{-1}\left[u|_{\p D^m}\right].$$
We then set the fictitious sources as
\begin{equation} \label{eq:sources2}
g_m = B_{D}\phi^i_m, \quad 0<m<N-1,
\end{equation}
and $g_m = 0$ otherwise. Here, $B_D$ are the operators defined in \Cref{prop:B} with the choice $\Omega = D$. Then, putting $\phi^o = 0$ and $H = u$, we obtain a solution $\widetilde u$ defined by \eqref{eq:utilde2}, which coincides with $u$ on $Y^m \setminus D^m$.

Conversely, if we have a non-zero solution $\widetilde u$ to \eqref{eq:scattering_translated_sources}, represented as \eqref{eq:utilde2} in $Y^m$ and with sources satisfying \eqref{eq:sources2}, then we can define a non-zero solution $u=\widetilde u$ to \eqref{eq:scattering_translated} coinciding with $\widetilde u$ on $Y^m \setminus D^m$.

We introduce the extended operator on $\left(L^2(\p D)\right)^2$,
$$B = \begin{pmatrix} 0 & 0 \\  B_D & 0 \end{pmatrix}.$$
For $\alpha\in Y^*$, define $\B^\alpha: \left(L^2(\p D) \right)^{2N} \rightarrow\left(L^2(\p D) \right)^{2}$ block-wise as
$$\B^\alpha  = \Big(\begin{matrix} B & e^{-\iu \alpha}B & \cdots & e^{-(N-1)\iu\alpha}B \end{matrix}\Big),$$
and define $E^\alpha: \left(L^2(\p D) \right)^{2} \rightarrow \left(L^2(\p D) \right)^{2N}$ block-wise as
$$ E^\alpha = \begin{pmatrix} I \\ e^{\iu \alpha}I \\ e^{2\iu\alpha}I \\ \vdots \\ e^{(N-1)\iu\alpha}I\end{pmatrix}.$$

Next, we follow the approach of \Cref{sec:inteq_smalldis} to derive the integral equation for the dislocated system. By taking the Fourier transform, we obtain \eqref{eq:scattering_quasi} together with the relation \eqref{eq:psimphim} for $\phi^i_m$ and $\phi^o_m$. Putting
$$
\Phi_N = \begin{pmatrix} \phi^i_0\\ \phi^o_0 \\ \vdots \\ \phi^i_{N-1} \\ \phi^o_{N-1} \end{pmatrix},
$$
we then obtain the following result.
\begin{prop}
	For $\C_d$ as in \eqref{eq:CN}, the mid-gap frequencies of \eqref{eq:scattering_translated}  are precisely the values $\omega$ such that there is a non-zero solution $\Phi_N \in \left(L^2(\p D)\right)^{2N}$ to the equation
	\begin{equation} \label{eq:int_eq_N}
	\Phi_N = -\frac{1}{2\pi}\left( \int_{Y^*}E^\alpha \big(\A^{\alpha}(\omega,\delta)\big)^{-1} \B^\alpha\dx \alpha \right) \Phi_N.
	\end{equation}
\end{prop}
In order to analyse \eqref{eq:int_eq_N}, we will need the following lemma, which is an immediate consequence of the structure of $B$.
\begin{lem}\label{lem:AB}
	We have
	\begin{equation}\label{eq:AB}
	\big(\A^{\alpha}(\omega,\delta)\big)^{-1}B = \begin{pmatrix} A_{12}B_{D}& 0 \\ A_{22}B_D & 0 \end{pmatrix}.
	\end{equation}
	As $\delta \rightarrow 0$ and $\omega = O(\sqrt{\delta})$, the operator $A_{12}$ can be approximated by \eqref{eq:A12} and \eqref{eq:A22}, respectively.
	\begin{comment}
	\begin{align*}
	A_j^{(1)} &= -\frac{\langle  \chi_{\p D_j} , \cdot\rangle }{\omega^2|D_1|}\left(\psi_j + \frac{b_j^{(1)}}{2}\left( \frac{\left(\omega_1^\alpha\right)^2}{\omega^2-\left(\omega_1^\alpha\right)^2}\right) u_1 - \frac{b_j^{(2)}}{2}\left( \frac{\left(\omega_2^\alpha\right)^2}{\omega^2-\left(\omega_2^\alpha\right)^2}\right) u_2\right) + O\left(\frac{1}{\omega}\right)  \\
	A_j^{(2)} &= -\frac{\langle \chi_{\p D_j} , \cdot \rangle }{\omega^2|D_1|}\left(-\frac{b_j^{(1)}}{2}\left( \frac{\left(\omega_1^\alpha\right)^2}{\omega^2-\left(\omega_1^\alpha\right)^2}\right) u_1^\alpha + \frac{b_j^{(2)}}{2}\left( \frac{\left(\omega_2^\alpha\right)^2}{\omega^2-\left(\omega_2^\alpha\right)^2}\right) u_2^\alpha\right) + O\left(\frac{1}{\omega}\right) ,
	\end{align*}
	for $j = 1,2$. Here $b_1^{(1)} = e^{-\iu \theta_\alpha}, b_2^{(1)} = -1, b_1^{(2)} = e^{-\iu \theta_\alpha}, b_2^{(2)} = 1$.
	\end{comment}
\end{lem}
Due to the zero column in \eqref{eq:AB}, it is clear that \eqref{eq:int_eq_N} reduces to an equation for $\phi^i_m, \cdots, \phi^i_{N-1}$ only. In fact, from \eqref{eq:int_eq_N}, it follows that
\begin{equation}\label{eq:psi}
\Phi^i_N = -\frac{1}{2\pi} \int_{Y^*} \begin{pmatrix} 1& e^{\iu \alpha} & \cdots & e^{-(N-1)\iu\alpha}\\ e^{\iu \alpha} & 1 & \cdots & e^{-(N-2)\iu\alpha} \\ \vdots & \vdots & \ddots  & \vdots \\ e^{(N-1)\iu\alpha} & e^{(N-2)\iu\alpha}& \cdots & 1 \end{pmatrix}\begin{pmatrix} A_{12}B_{D}& 0 & \cdots & 0\\ 0 & A_{12}B_{D} & \cdots & 0 \\ \vdots & \vdots & \ddots  & \vdots \\ 0 & 0& \cdots & A_{12}B_{D} \end{pmatrix}\Phi^i_N\dx \alpha,
\end{equation}
where
$$\quad \Phi^i_N = \begin{pmatrix} \phi^i_0\\ \phi^i_1 \\ \vdots \\ \phi^i_{N-1} \end{pmatrix}.$$
From \Cref{lem:AB}, we obtain that, to leading order, $\phi^i_m$ is a linear combination of $\psi_1$ and $\psi_2$: %More precisely, we have
$$\phi^i_m = c_m \psi_1 + d_m \psi_2 + O(\omega),$$ with respect to the $L^2(\p D)$-norm. Define, for $j=1,2$,
$$ t_{i,j}^m = \frac{1}{2\pi}\int_{Y^*}e^{\iu m\alpha}\left\langle \chi_{\p D_i}, \left(I+ A_{12}B_{D}\right)[\psi_j]\right\rangle \dx \alpha.$$
Then, taking inner products $\langle \chi_{\p D_i}, \cdot \rangle$ in equation \eqref{eq:psi} we find
$$\frac{1}{2\pi} \int_{Y^*} e^{\iu m\alpha}\begin{pmatrix} \left\langle \chi_{\p D_1},I+ A_{12}B_{D}[\phi^i_n] \right \rangle \\\left\langle \chi_{\p D_2}, I+ A_{12}B_{D}[\phi^i_n] \right\rangle\end{pmatrix} \dx \alpha = T_m \begin{pmatrix} c_n \\ d_n\end{pmatrix},$$
where $T_m$ denotes the $2\times 2$ matrix $\big(t_{i,j}^m\big)$.
We thus have
\begin{equation}\label{eq:main_N}
\T_N(\omega)C_N = 0,
\end{equation}
where we have defined
$$\T_N(\omega) = \begin{pmatrix} T_0& T_{-1} & \cdots & T_{-(N-1)}\\ T_{1} & T_0 & \cdots & T_{-(N-2)} \\ \vdots & \vdots & \ddots  & \vdots \\ T_{N-1} & T_{N-2} & \cdots & T_0  \end{pmatrix}, \qquad C_N = \begin{pmatrix} c_0 \\ d_0 \\ c_1 \\ \vdots \\ d_{N-1}\end{pmatrix}.$$
Observe that $\T_N$ is a block Toeplitz matrix generated by the symbol $\varphi$,
$$\varphi = \varphi(\alpha) = \begin{pmatrix}\varphi_{1,1} & \varphi_{1,2} \\ \varphi_{2,1} & \varphi_{2,2}\end{pmatrix}, \qquad \varphi_{i,j} = \big\langle \chi_{\p D_i}, \left(I+ A_{12}B_{D}\right)[\psi_j]\big\rangle, \quad i,j = 1,2.$$
In the following lemma, we compute $\varphi$.
\begin{lem}
	We have
	\begin{align*}
	\varphi(\alpha) = -\frac{\mathrm{Cap}_{D_1}}{2} \begin{pmatrix}
	\eta_1+\eta_2 & -e^{\iu \theta_\alpha}\left( \eta_1-\eta_2 \right) \\ -e^{-\iu \theta_\alpha}\left( \eta_1-\eta_2 \right) & \eta_1+\eta_2
	\end{pmatrix}, \qquad \det \varphi(\alpha) = \left(\mathrm{Cap}_{D_1}\right)^2\eta_1\eta_2,
	\end{align*}
	where
	$$\eta_j  =\frac{\left(\omega_j^\alpha\right)^2}{\omega^2-\left(\omega_j^\alpha\right)^2}, \quad j = 1,2. $$
\end{lem}
\begin{proof}
	As computed in \cite{first}, we have
	$$\left\langle \chi_{\p D_j}, \left(-\frac{1}{2}I +\K_{D_j}^{\omega,*}\right)\psi_j \right\rangle  = -\omega^2|D_1| + O(\omega^3),$$
	and therefore,
	$$\left\langle \chi_{\p D_i}, B_{D}\psi_j \right\rangle  = \omega^2|D_1|\delta_{i,j} + O(\omega^3), \ i,j = 1,2.$$
	From this, using \Cref{lem:AB} and \eqref{eq:A12} we find that
	\begin{align*}
	A_{12}B_{D}[\psi_1] &= \left(-1+\frac{\eta_1 + \eta_2}{2}\right)\psi_1 -e^{-\iu \theta_\alpha}\frac{\eta_1-\eta_2 }{2}\psi_2 + O(\omega),\\
	A_{12}B_{D}[\psi_2] &= -e^{\iu \theta_\alpha}\frac{\eta_1-\eta_2 }{2}\psi_1 +  \left(-1+\frac{\eta_1 + \eta_2}{2}\right)\psi_2 + O(\omega),
	\end{align*}
 	with respect to the $L^2(\p D)$-norm. The result now follows from the facts that $\left\langle \chi_{\p D_i}, \psi_j\right\rangle = -\mathrm{Cap}_{D_i} \delta_{i,j}$ and  $\mathrm{Cap}_{D_1} = \mathrm{Cap}_{D_2}$.
\end{proof}
Observe, in particular, that $\varphi$ is a Hermitian matrix and, therefore, the Toeplitz matrices $\T_N$ are also Hermitian. We define the ``exchange'' matrix $J_m\in \R^{2m}, m\in \N,$ as
$$J_m = \begin{pmatrix} 0 & \cdots & 0 & 1\\ 0 & \cdots & 1 & 0 \\ \vdots & \adots  & \vdots & \vdots \\ 1 & \cdots & 0 & 0 \end{pmatrix}.$$
The following lemma describes the centrosymmetry property of Hermitian Toeplitz matrices.
\begin{lem}\label{lem:centrosym}
We have
$$T_m = J_1\overline{T}_{-m} J_1, \qquad \T_N = J_N\overline{\T_N} J_N.$$
\end{lem}
\begin{proof}
We have $J_1\varphi J_1 = \overline{\varphi}$ and therefore
\begin{align*}
T_m &= \frac{1}{2\pi}\int_{Y^*}e^{\iu m\alpha}\varphi(\alpha) \dx \alpha \\
&= \frac{1}{2\pi}\int_{Y^*}J_1\overline{e^{-\iu m\alpha}\varphi(\alpha)} J_1\dx \alpha \\
&= J_1\overline{T}_{-m} J_1.
\end{align*}
The second equality of the statement follows from the first one together with the Toeplitz structure of $\T_N$.
\end{proof}

We will study the solutions to \eqref{eq:main_N} in the two cases $N=1$ and $N\rightarrow \infty$. The following proposition characterizes the solutions in the case $N=1$, corresponding to two removed resonators.
\begin{prop}\label{prop:N=1}
	If $N=1$, the equation \eqref{eq:main_N} has a non-zero solution if and only if $\omega$ is a solution to one of the two equations
	\begin{equation}\label{eq:N=1}
	\frac{1}{2\pi} \int_{Y^*}\Big(\eta_1\left(1\pm e^{\iu \theta_\alpha}\right)+\eta_2\left(1\mp e^{\iu \theta_\alpha}\right)\Big)\dx \alpha = 0.
	\end{equation}
	If $l_0 < 1/2$, there are no solution to the equations \eqref{eq:N=1}, while if $l_0 > 1/2$, each equation has exactly one solution.
\end{prop}
\begin{proof}
	In the case $N=1$, equation \eqref{eq:main_N} reads
	$$
	T_0\begin{pmatrix}  c_0 \\ d_0 \end{pmatrix} = 0,
	$$
	which has a non-zero solution if and only if $\det T_0 = 0$. We have
	$$
	\det T_0 = \frac{\left(\mathrm{Cap}_{D_1} \right)^2}{4} \left( \left(I_1 \right)^2 - |I_2|^2  \right),
	$$
	where
	$$
	I_1 = \frac{1}{2\pi} \int_{Y^*}\left(\eta_1+\eta_2\right)\dx \alpha, \qquad I_2 = \frac{1}{2\pi} \int_{Y^*}e^{\iu \theta_\alpha}\left(\eta_1-\eta_2\right)\dx \alpha.
	$$
	By time-reversal symmetry, we have $\omega_j^{-\alpha} = \omega_j^\alpha, j=1,2$, which implies $I_2 \in \R$. Hence $\det T_0 = 0$ is equivalent to
	$$
	I_1 - I_2 = 0, \ \mathrm{or} \ I_1 + I_2 = 0.
	$$
	The remaining part of the proof is given in \Cref{app:prop}. It is shown that each of these equations has a unique solution in the case $l_0 > 1/2$, while no solutions in the case $l_0 > 1/2$.
\end{proof}
Denote by $\T(\omega)$ the infinite Toeplitz matrix corresponding to $\T_N(\omega)$, \ie{},
$$\T(\omega)= \begin{pmatrix} T_0& T_{-1} & T_{-2} &  \cdots \\ T_{1} & T_0 & T_{-1}  &\cdots \\ T_2 & T_1 & T_0&\cdots \\ \vdots & \vdots  & \vdots  & \ddots \end{pmatrix},$$
which defines a bounded operator on the space $l^2_2$ of sequences of two-dimensional vectors. More precisely, $l_2^2$ consists of sequences $\{x_n\}_{n=0}^\infty\in l^2_2$ of vectors $x_n\in \R^2$ such that
$$ \left(\sum_{n=0}^\infty \|x_n\|^2\right)^{1/2}  < \infty, $$
where $\|\cdot \|$ denotes the Euclidean norm.

\begin{prop} \label{prop:inf}
	Given $\omega_\infty$ inside the band gap such that $\T(\omega_\infty)$ has eigenvalue $0$, there are two frequencies $\omega_1(N), \omega_2(N) \rightarrow \omega_\infty$ as $N\rightarrow \infty$, such that $\T_N$ is not invertible at $\omega_1(N), \omega_2(N)$.
\end{prop}
\begin{proof}
Let $X = \{x_n\}_{n=0}^\infty \in l^2_2$ be an eigenvector with $\T(\omega_\infty) X = 0$ and let $x\in \R^{2N}$ be a truncation of $X$. Since $\T(\omega_\infty) X = 0$, we have
\begin{equation}\label{eq:null}
\sum_{n=0}^\infty T_{k-n} x_n = 0
\end{equation}
for all $k\in \N$. Define $z_1, z_2 \in \R^{4N}$,
$$z_1 = \begin{pmatrix} x \\ J_{N}\overline{x} \end{pmatrix}, \qquad z_2 = \begin{pmatrix} x \\ -J_{N}\overline{x} \end{pmatrix}.$$
Then, using \Cref{lem:centrosym} we have
$$\T_{2N}(\omega_\infty) z_1 =
\begin{pmatrix} \vdots \\[0.3em]
\ds \sum_{n=0}^{N-1} T_{k-n}x_n + \sum_{n=0}^{N-1} T_{k-N-n}J_N\overline{x}_{N-1-n} \\[0.3em]
\vdots
\end{pmatrix} = \begin{pmatrix} \vdots \\[0.3em]
\ds \sum_{n=0}^{N-1} T_{k-n}x_n + J_N\overline{\sum_{n=0}^{N-1} T_{2N-1-k-n}x_{n}} \\[0.3em]
\vdots
\end{pmatrix} $$
for $\ k = 0,\cdots,2N-1.$

In view  of \eqref{eq:null}, given $\epsilon >0$ we can choose $N$ such that
$$\| \T_{2N} z_1 \| < \epsilon,$$
which implies that $0$ is in the $\epsilon$-pseudospectrum of $\T_{2N}(\omega_\infty)$ (see, for example, \cite{trefethen} for a thorough discussion on the definition and properties of pseudospectra). Since $\T_{2N}(\omega_\infty)$ is Hermitian, it follows that there is an eigenvalue $\mu_1$ of $\T_{2N}(\omega_\infty)$ with $|\mu_1| < \epsilon$. From this, it follows that there is a value $\omega_1$ such that $\T_{2N}(\omega_1)$ is not invertible, satisfying $|\omega_1 - \omega_\infty| < K\epsilon$ for some $K$ independent on $\epsilon$ \cite{MaCMiPaP}.

In the same way, we can show that given $\epsilon >0$ we can choose $N$ such that
$$\| \T_{2N} z_2 \| < \epsilon,$$
and therefore there is a value $\omega_2$ such that $\T_{2N}(\omega_2)$ is not invertible, satisfying $|\omega_2 - \omega_\infty| < K\epsilon$ for some $K$ independent on $\epsilon$.

The above argument shows that $\omega_i(2N)\rightarrow \omega_\infty$ as $N\rightarrow \infty$. The case of the sequence $\omega_i(2N-1)$, corresponding to odd indices, follows similarly by choosing the truncation $x\in \R^{2N-1}$ and constructing $z_1, z_2$ analogously.
\end{proof}
\begin{comment}
$$\T_{2N} z_1 =
\begin{pmatrix} \ds \sum_{n=0}^{N-1} T_{-n}x_n + \sum_{n=0}^{N-1} T_{-n-N}J\overline{x}_{N-1-n} \\[0.3em]
\ds \sum_{n=0}^{N-1} T_{-n+1}x_n + \sum_{n=0}^{N-1} T_{-n+1-N}J\overline{x}_{N-1-n} \\[0.3em]
\vdots \\[0.3em]
\ds \sum_{n=0}^{N-1} T_{-n+2N-1}x_n + \sum_{n=0}^{N-1} T_{-n+N-1}J\overline{x}_{N-1-n}
\end{pmatrix}$$
\end{comment}
\begin{rmk}
The values of the non-zero solutions $C_N$ to \eqref{eq:main_N} correspond to the values attained by the mid-gap modes inside the dislocation region. The two pseudomodes $z_1$ and $z_2$ can be interpreted as approximations of the monopole and dipole modes, respectively, arising from the hybridization of the two semi-infinite half-structures. As the dislocation increases, \ie{}, as $N\rightarrow \infty$, the strength of the hybridization decreases and the frequencies corresponding to these modes converge to the same value $\omega_\infty$.
\end{rmk}
\begin{rmk}
The work in this section shows the intimate connection between localized edge modes and the fact that Toeplitz matrices with sufficiently smooth symbols have eigenvectors which are exponentially localized to the ``edges'' (\ie{} the first and last entries) of the vector \cite{trefethen}.
\end{rmk}

\subsection{Dislocation larger than resonator width} \label{sec:no//}
In this section, we assume that the size of the dislocation is larger than the width of one resonator. In other words, this means that each dislocated resonator does not overlap with corresponding original, undislocated, resonator.

We begin by stating some facts from \cite{AndoKang} on the eigenfunctions of the Neumann-Poincar\'e operator $\K_\Omega^{0,*}$ for a domain $\Omega$ with $\p \Omega\in \C^{1,s}, 0<s<1$. Here, we additionally assume that $\Omega$ is connected, which means that $\Omega$ can be thought of as a single resonator $D_j^m$ in the dislocated array. The operator $\K_\Omega^{0,*}$ is known to be self-adjoint in the inner product $\langle \cdot, \cdot\rangle_{-1/2}$ on $H^{-1/2}(\p \Omega)$ defined by
$$\langle u, v\rangle_{-1/2} = -\left\langle u, \S_\Omega^0[v]\right\rangle_{-1/2,1/2},$$
where $\langle \cdot, \cdot\rangle_{-1/2,1/2}$ denotes the duality pairing of $H^{-1/2}(\p \Omega)$ and $H^{1/2}(\p \Omega)$. Then, by the spectral theorem, the eigenfunctions $\psi_{\Omega}^j, j=1,2,3, \cdots,$ of $\K_\Omega^{0,*}$ form a basis of $H^{-1/2}(\p \Omega)$ that is orthonormal with respect to $\langle \cdot, \cdot\rangle_{-1/2}$, while the functions $\S_\Omega^0[\psi_{\Omega}^j]$ form a basis of $H^{1/2}(\p \Omega)$ that is orthogonal with respect to the inner product $\langle \cdot, \cdot\rangle_{1/2}$ defined by
$$\langle u, v\rangle_{1/2} = -\left\langle \left(\S_\Omega^0\right)^{-1}[u], v\right\rangle_{-1/2,1/2}.$$

The following addition theorem gives an expansion of  Green's function $G^\omega(x,z)$, with the origin shifted by $z\notin \p \Omega$, in terms of $\S_\Omega^\omega[\psi_{\Omega}^j](x)$.
\begin{prop} \label{thm:AndoKang}
	For $x\in \p \Omega$, $z\notin \p \Omega$ and $\omega$ small enough, we have
	\begin{equation*}\label{eq:add}
	G^\omega(x,z) = - \sum_{i=1}^\infty \S^\omega_\Omega[\xi_{\Omega}^i](z)  \S^\omega_\Omega[\psi_{\Omega}^i](x),
	\end{equation*}
	where $\xi_\Omega^i = \left(\S_\Omega^\omega\right)^{-1}\S_\Omega^0[\psi_{\Omega}^i]$.
\end{prop}
\begin{proof}
	The proof follows the same arguments as those in \cite{AndoKang}, where an analogous result was proven for Laplace Green's function $G^0$. We include the proof for the sake of completeness.

	Since $\S_\Omega^0[\psi_{\Omega}^i]$ is a basis of $H^{1/2}(\p \Omega)$, and since $\S_\Omega^\omega: H^{-1/2}(\p \Omega)\rightarrow H^{1/2}(\p \Omega)$ is invertible for $\omega$ small enough, we can expand $G^\omega$ for fixed $z$ as follows,
	\begin{equation}\label{eq:Gexpplasm}
	G^\omega(\cdot,z) = \sum_{i=1}^\infty c_i(z) \S_\Omega^\omega[\psi_{\Omega}^i],
	\end{equation}
	for some coefficients $c_i$ with
	$$\sum_{i=1}^\infty |c_i(z)|^2 < \infty, \quad z\notin \p {\Omega}.$$
Moreover, $\psi_i$ are orthonormal in $H^{-1/2}(\p \Omega)$ equipped with $\langle \cdot, \cdot \rangle_{-1/2}$. From $\left(\S_\Omega^\omega\right)^* = \overline{\S_\Omega^\omega}$, we have
	\begin{align}\label{eq:AndoKang1}
	-\left\langle \overline{\xi_\Omega^i}, \S_\Omega^\omega[\psi_\Omega^j] \right\rangle_{-1/2,1/2}&= -\left\langle \left(\overline{\S_\Omega^\omega}\right)^{-1}\S_\Omega^0[\psi_\Omega^i], \S_\Omega^\omega[\psi_\Omega^j] \right\rangle_{-1/2,1/2} \nonumber \\
	&= \delta_{i,j}.
	\end{align}
	Therefore,
	\begin{align} \label{eq:AndoKang2}
	\left\langle \overline{\xi_\Omega^i}, G^\omega(\cdot,z) \right\rangle_{-1/2,1/2} &= \S_\Omega^\omega[\xi_\Omega^i](z).
	\end{align}
	Combining  \eqref{eq:Gexpplasm} together with \eqref{eq:AndoKang1} and \eqref{eq:AndoKang2} shows the claim.
\end{proof}
%\begin{rmk}
%	In \eqref{eq:add}, we interchangeably think of $\S_\Omega^\omega[f](x)$ as a function $x\in \R^3$ and corresponding trace $x\in\p \Omega$. For $y\in \p \Omega$, we have $\S_\Omega^\omega[\xi_i](y) = \S_\Omega^0[\psi_i](y)$, which is not true for $y \notin \p \Omega$.
%\end{rmk}

We denote $\Omega_d = \Omega + d\mathbf{v}$. We then have the following proposition.
\begin{prop} \label{prop:V}
	Assume that $\overline\Omega \cap \overline\Omega_d = \emptyset$. Then, for $\phi \in H^{-1/2}(\p {\Omega_d})$, we have
	$$
	\S_{\Omega_d}^\omega[\phi](x-d\mathbf{v}) =  \S_{\Omega_d}^\omega[V\phi](x), \quad x \in \p \Omega_d,
	$$
	where $V: H^{-1/2}(\p \Omega_d) \rightarrow  H^{-1/2}(\p \Omega_d)$ is given by
	$$
	V[\psi_{\Omega_d}^j] = \sum_{i=1}^\infty  V_{i,j} \psi_{\Omega_d}^i, \qquad
	V_{i,j} = -\int_{\p {\Omega_d}} \S_{\Omega_d}^\omega[\xi_{\Omega_d}^i](y-d\mathbf{v})\psi_{\Omega_d}^j(y)\dx \sigma (y), \quad  i,j\geq 1.
	$$
\end{prop}
\begin{proof}
	Since
	$$
	\S_{\Omega_d}^{\omega}[\phi](x-d\mathbf{v}) = \int_{\p {\Omega_d}}G^\omega(x,y+d\mathbf{v})\phi(y) \dx \sigma(y),
	$$
	and since $y+d\mathbf{v}\notin \p {\Omega_d}$, the proposition follows from \Cref{thm:AndoKang}.
\end{proof}

We will also need the following addition theorem for the normal derivative of the single-layer potential. We let $\D_\Omega^\omega$ denote the double-layer potential (for details on this operator we refer, for example, to \cite{MaCMiPaP}).
\begin{prop} \label{prop:W}
Assume that $\overline\Omega \cap \overline\Omega_d = \emptyset$. Then, for $\phi \in H^{-1/2}(\p {\Omega_d})$, we have
$$
\frac{\p \S_{\Omega_d}^\omega}{\p \nu_{x-d\mathbf{v}}}[\phi](x-d\mathbf{v}) = W\frac{\p \S_{\Omega_d}^\omega}{\p \nu_x}\bigg|_+[\phi](x), \quad x \in \p {\Omega_d},
$$
where $W: H^{-1/2}(\p {\Omega_d}) \rightarrow  H^{-1/2}(\p {\Omega_d})$ is given by
$$
W \left(\frac{1}{2} + \K_{\Omega_d}^{\omega,*}\right)^{-1}[\psi_{\Omega_d}^j] = \sum_{i=1}^\infty  W_{i,j} \psi_{\Omega_d}^i, \qquad
W_{i,j} = \int_{\p {\Omega_d}} \D_{\Omega_d}^\omega\S_{\Omega_d}^0[\psi_{\Omega_d}^i](y+d\mathbf{v})\psi_{\Omega_d}^j(y)\dx \sigma (y), \quad  i,j\geq 1.
$$
Here, $\p/\p_{\nu_{x-d\mathbf{v}}}$ denotes the normal derivative with respect to $\Omega$.
\end{prop}
\begin{proof}
Analogously to the proof of \Cref{thm:AndoKang}, we can show that
$$
\frac{\p G^\omega}{\p \nu_{x}}(x,y) = \sum_{i=1}^\infty \D_\Omega^\omega\S^0_\Omega[\psi_{\Omega}^i](y) \psi_{\Omega}^i(x), \qquad x\in \p \Omega, \ y \notin \p \Omega.
$$
The result now follows by the same argument as the one in the proof of \Cref{prop:V}, using the jump relation
\begin{equation*}\frac{\p \S_{\Omega_d}^\omega}{\p \nu_x}\bigg|_+[\phi] = \left(\frac{1}{2} + \K_{\Omega_d}^{\omega,*}\right)[\phi].
\qedhere
\end{equation*}
\end{proof}

\subsubsection{Fictitious sources for the non-overlapping resonators}\label{sec:sources_no//}
Here we describe the method of fictitious sources when a single resonator $\Omega$ is dislocated by $d$ such that $\overline\Omega\cap \overline\Omega_d = \emptyset$, where $\Omega_d = \Omega + d\mathbf{v}$.

The arguments follow closely those of \Cref{sec:fic_small}. Again, we consider the two problems \eqref{eq:scattering_sources} and \eqref{eq:scattering} corresponding, respectively, to the original geometry with sources and to the dislocated geometry without sources. Representing the solutions as \eqref{eq:u} and \eqref{eq:utilde}, we again arrive at the equations given in \eqref{eq:A_dA}. Next, we will use \Cref{prop:V} to study these equations.

Let $U_0$ be a neighbourhood of $\Omega$ not containing $\Omega_d$. Imposing $u=\widetilde u$ in $U_0\setminus \Omega$ we find from \Cref{prop:V} that
$$
\Phi_d = \P_1\Phi, \quad \text{ where } \quad \P_1 := \begin{pmatrix} V^{-1} & 0 \\ 0 & V^{-1}\end{pmatrix}Q.
$$
As before, since $\Omega_d$ and $\Omega$ only differ by a translation, we can easily see that
\begin{equation*} \label{eq:A_Bno//}
\A_d  = Q\A Q^{-1}.
\end{equation*}
In $U_0$, we can represent $H$ as
$$H(x) = \sum_{i=1}^\infty c_i\S_{\Omega_d}^\omega(x), \quad x\in U_0,$$
for some constants $c_i, \ i=1,2,...$ This gives
$$
\begin{pmatrix} H \big|_{\p \Omega} \\ \delta \p_\nu H \big|_{\p \Omega} \end{pmatrix} = \P_2 \begin{pmatrix} H \big|_{\p \Omega_d} \\ \delta \p_\nu H \big|_{\p \Omega_d} \end{pmatrix}, \quad \text{ where } \quad \P_2 := Q\begin{pmatrix} V^{*} & 0 \\ 0 & W\end{pmatrix}.
$$
Here, $V^{*}: H^{1/2}(\p {\Omega_d}) \rightarrow  H^{1/2}(\p {\Omega_d})$ is defined by
$$
V^{*}\left[\S_{\Omega_d}^\omega[\psi_{\Omega_d}^j]\right] = \sum_{i=1}^\infty  V_{i,j} \S_{\Omega_d}^\omega[\psi_{\Omega_d}^i].
$$
Combining this together with \eqref{eq:A_dA} gives the following result.
\begin{prop} \label{prop:B_no//}
	The layer densities $\phi^i$ and $\phi^o$ and the fictitious sources $f$ and $g$ satisfy
	\begin{equation*}
	\begin{pmatrix}
	f \\ g
	\end{pmatrix} = B(\omega,\delta,d) \begin{pmatrix} \phi^i \\ \phi^o \end{pmatrix}, \qquad B(\omega,\delta,d) = \P_2 \A \P_1 - \A.
	\end{equation*}
\end{prop}

\subsubsection{Integral equation for dislocations larger than the resonator width}
We define $d_0$ as the width of one resonator in the $x_1$-direction, \ie{},
$$
d_0 = \inf\left\{d\in \R^+ \mid \overline{D_1}\cap \overline{D_1+d\mathbf{v}} = \emptyset \right\}.
$$
We define
$$\B_d = \hat{\P}_2\hat{\A}\hat{\P}_1 - \hat{\A},$$
where
$$\hat{\A} =  \begin{pmatrix}
\hat\S_D^{\omega} & - \hat\S_D^{\omega} \\ -\frac{1}{2}I + \hat\K_D^{\omega,*}  & -\delta\left(\frac{1}{2}I + \hat\K_D^{\omega,*} \right)
\end{pmatrix}, \quad \hat{\P}_1 = \begin{pmatrix} \hat{V}^{-1} & 0 \\ 0 & \hat{V}^{-1}\end{pmatrix}, \quad \hat{\P}_2 = \begin{pmatrix} \hat{V}^{*} & 0 \\ 0 & \hat{W}\end{pmatrix},$$
with
$$\hat{V} = \begin{pmatrix} V_1 & 0 \\ 0 & V_2\end{pmatrix}, \quad \hat{V}^{*} = \begin{pmatrix} V_1^{*} & 0 \\ 0 & V_2^{*}\end{pmatrix}, \quad \hat{W} = \begin{pmatrix} W_1 & 0 \\ 0 & W_2\end{pmatrix},$$
where $V_j, V_j^{*}, W_j$ are defined as in \Cref{sec:sources_no//} with $\Omega = D_j$, $j=1,2$. Then $\B_d$ describes the fictitious sources for the dimer. Following the same arguments as those in \Cref{sec:inteq_smalldis}, we obtain the following result.
\begin{prop}
	For $d > d_0$, the mid-gap frequencies of \eqref{eq:scattering_translated} are precisely the values $\omega$ such that there is a non-zero solution $\phi^{\alpha,i}, \phi^{\alpha,o} \in L^2(\p D\times Y^*)$ to the equation
	\begin{equation} \label{eq:int_eq_no//}
	\begin{pmatrix}
	\phi^{\alpha,i}
	\\[0.3em]
	\phi^{\alpha,o}
	\end{pmatrix}
	= -\big(\mathcal{A}^\alpha(\omega,\delta)\big)^{-1}\left(\sum_{m=0}^\infty e^{-\iu m\alpha} \B_d I_m\right)
	\begin{pmatrix}
	\phi^{\alpha,i}
	\\[0.3em]
	\phi^{\alpha,o}
	\end{pmatrix}.
	\end{equation}
\end{prop}

Our next goal is to show that as $d$ increases, any mid-gap frequency will remain inside the band gap. We begin by stating the following lemma, which is the analogue of \Cref{lem:C}.
\begin{lem}\label{lem:PhiBPsi}
	Assume that the resonators are in the dilute regime specified by (\ref{eq:dilute}). Then, for $d\in (d_0, \infty)$ and for small enough $\epsilon$ and $\delta$
	\begin{align*}
	|\langle \Phi_j^\diamond, \B_d \Psi_j^\diamond\rangle | > K > 0, \qquad  j = 1,2,
	\end{align*}
	for some constant $K$ independent of $d$.
\end{lem}
The proof of this result is given in \Cref{app:lem2}. We are now ready to state and prove the main result of this section. Recall that we denote the edges of the band gap by
$$\omega_1^\diamond = \omega_1^{\pi/L}, \qquad \omega_2^\diamond = \omega_2^{\pi/L}.$$
We then have the following proposition.
\begin{prop}\label{prop:main_no//}
	For $d>d_0$ and $\delta$ small enough, any mid-gap frequency $\omega(d)$ is bounded away from the edges of the band gap, \ie{}
	$$|\omega(d)-\omega_j^\diamond| > c, \qquad j=1,2,$$
	for all $d > d_0$ and for some positive constant $c$ independent of $d$.
\end{prop}

\begin{proof}
	We want to show that there are no solutions to \eqref{eq:int_eq_no//} that approaches the edges of the band gap. Assume the contrary, \ie{} that we have a solution $\omega \rightarrow \omega_j^\diamond$.
	Following the proof of \Cref{prop:smalld}, we obtain
$$\frac{\langle \Phi_j^\diamond, \B_d\Psi_j^\diamond\rangle}{4\omega_j^\diamond|D_1|(\omega-\omega_j^\diamond)} = 1 + o(1),$$
as $\omega\rightarrow \omega_j^\diamond$. But since $|\langle \Phi_j^\diamond, \B_d\Psi_j^\diamond\rangle| > K > 0$ for all $d$, this equation has no solution.
\end{proof}

\subsection{Theorem on mid-gap frequencies} \label{sec:mainthm}
We now combine the results of  the two previous sections, namely \Cref{prop:N=1}, \Cref{prop:inf} and \Cref{prop:main_no//}, into the following theorem.
\begin{thm} \label{thm:main}
	Assume that the resonators are in the dilute regime specified by \eqref{eq:dilute} and that $l_0 > 1/2$. Then, for small enough $\delta$ and $\epsilon$, there exists some $d_0=O(\epsilon)$ such that there are two mid-gap frequencies $\omega_1(d)$ and $\omega_2(d)$ %of the dislocated system \eqref{eq:scattering_translated}. The frequencies $\omega_1(d)$ and $\omega_2(d)$ that are contained inside the band gap of the periodic system
	for all $d\in [d_0,\infty)$, both of which converge to the same value $\omega_\infty$ as $d\rightarrow \infty$.
\end{thm}

\begin{corollary} \label{cor:interval}
	Assume that the resonators are in the dilute regime specified by \eqref{eq:dilute} and that $l_0 > 1/2$. Then, for small enough $\delta$ and $\epsilon$, there is an interval $\mathcal{I}=[\omega_1(d_0), \omega_2(d_0)]$ within the band gap such that if $\omega\in \mathcal{I}\setminus\{\omega_\infty\}$, then there exists some $d>d_0$ such that $\omega\in\{\omega_1(d),\omega_2(d)\}$.
%	will be a mid-gap frequency for some $d>d_0$.
\end{corollary}
%Hence, the interval $I=[\omega_1(L), \omega_2(L)]$ is contained inside the band gap, and any frequency $\omega \in I$, except possibly $\omega_\infty$, will be a resonant frequency of the dislocated system for a suitable choice of $d$.
\Cref{cor:interval} says that any frequency $\omega\in\mathcal{I}\setminus\{\omega_\infty\}$ is a mid-gap frequency of the structure for some dislocation $d$. From \Cref{prop:N=1}, we have an explicit way to compute the interval $\mathcal I$ and, as we will see from the numerical computations, this interval contains the middle region of the band gap. What we have shown is that we can choose a frequency in the middle of the band gap and create a structure having this as a resonant frequency, thus corresponding to exponentially localized edge modes that are stable under perturbations.

\Cref{prop:inf} and \Cref{prop:smalld} hint to the physical origin of the two mid-gap frequencies. For infinitely large dislocations, the system corresponds to two identical semi-infinite systems which each support edge modes with frequency $\omega_\infty$. As these two semi-infinite systems approach each other, they hybridize and $\omega_\infty$ splits into two frequencies, corresponding to monopole and dipole modes.

Seen from the other direction, $d=0$ corresponds to the periodic structure, which is known to have a band gap and no mid-gap frequencies. As $d$ increases from $0$, two mid-gap frequencies will emerge, one from each edge of the band gap.

\begin{rmk}
	The requirement that $d>d_0$ in \Cref{thm:main} was used in \Cref{sec:no//}. We assumed that the dislocation was sufficiently large that the translated resonators do not overlap with the originals. Since we are assuming that the structure is dilute and the size of each resonator is $O(\epsilon)$, $d_0=O(\epsilon)$. The non-overlapping assumption was made purely to simplify the analysis and not for any physical reason. Based on this, we conjecture that \Cref{thm:main} is true for all $d\in(0,\infty)$, which is in accordance with our numerical experiments. In this case, the interval $\mathcal I$ in \Cref{cor:interval} would include all of the band gap.
\end{rmk}

%\begin{rmk}
%	The assumption of diluteness only enters in \Cref{prop:N=1}. It is in order to guarantee the existence of two mid-gap frequencies precisely in the case $l_0 > 1/2$. Based on numerical evidence, we conjecture that this is true even in the non-dilute regime.
%\end{rmk}

\section{Finite arrays of resonators} \label{sec:finite}
In this section, we will study the finite array of resonators which is a truncation of the system studied in \Cref{sec:infinite}. We will see that this structure, which represents the physical manifestation of our above analysis, shares the important properties of the infinite system. We will also conduct a stability analysis of the structure.

Consider the structure $D$, consisting of $M$ resonators, that is a truncation of the infinite, dislocated array $\C_d$ studied in \Cref{sec:infinite}. Let $M = 4K+2$ for some $K\in \Z_+$ and assume that $D$ is given by
\begin{equation} \label{eq:struc_finite}
D =  D_2^{-K-1}\cup\Bigg(\bigcup_{m=-1}^{-K} D_1^m\cup D_2^m\Bigg)\cup\Bigg(\bigcup_{m=0}^{K-1} (D_1^m \cup D_2^m) + d \mathbf{v}\Bigg) \cup \left( D_1^{K}+d\mathbf{v}\right),
\end{equation}
where $D_1^m, D_2^m$ are as in \Cref{sec:infinite}, so that the symmetry assumptions \eqref{eq:symmetry} are satisfied and $\mathbf{v}$ is, again, the unit vector along the $x_1$-axis. Moreover, we assume $l_0 > 1/2$ (recall that $l_0=l/L$), corresponding to the case where the array supports edge modes.

%We will begin by studying a very general configuration of resonators. Explicitly, we assume that $D$ has the form
%\begin{equation} \label{finite_form}
%D = \bigcup_{i=1}^{M} D_i,
%\end{equation}
%where $D_i \subset \R^3, \ 1 \leq i \leq M/2$ is independent of $d$ and $D_i\subset \R^3 , \ M/2+1 \leq i \leq M$ can be written as $D_i = \widetilde{D}_i + dv$ for some $\widetilde{D}_i$ independent of $d$. We assume that $D_i, \ 1 \leq i \leq M$ are pairwise disjoint for all $d$ and satisfy $\p D_i \in \C^{1,s},$ for $0<s<1$ and $i = 1,\cdots,M$.

We model wave scattering by $D$ with the Helmholtz problem

\begin{equation*} \label{eq:scattering_finite}
\left\{
\begin{array} {ll}
\ds \Delta {u}+ \omega^2 {u}  = 0 \quad &\text{in } \R^3 \setminus \p D, \\
\nm
\ds  {u}|_{+} -{u}|_{-}  =0  \quad &\text{on } \partial D, \\
\nm
\ds  \delta \frac{\partial {u}}{\partial \nu} \bigg|_{+} - \frac{\partial {u}}{\partial \nu} \bigg|_{-} =0 \quad& \text{on } \partial D, \\
\nm
\ds |x| \left(\tfrac{\p}{\p|x|}-\iu \omega\right)u \to 0
&\text{as } {|x|} \rightarrow \infty.
\end{array}
\right.
\end{equation*}

\begin{figure}
	\centering
	\begin{tikzpicture}[scale=0.5]
	\draw (-6,0) circle (0.3);
	\draw (-5,0) circle (0.3);
	\draw (-3,0) circle (0.3);
	\draw (-2,0) circle (0.3);
	\draw (0,0) circle (0.3);
	\draw (1,0) circle (0.3);
	\draw (3,0) circle (0.3);
	\draw (7,0) circle (0.3);
	\draw (9,0) circle (0.3);
	\draw (10,0) circle (0.3);
	\draw (12,0) circle (0.3);
	\draw (13,0) circle (0.3);
	\draw (15,0) circle (0.3);
	\draw (16,0) circle (0.3);
	\draw[dotted] (-6.5,-1.5) rectangle (3.5,1.5);
	\draw[dotted] (6.5,-1.5) rectangle (16.5,1.5);
	%	\draw[dashed] (3.5,-0.5) -- (3.5,0.5);
	%	\draw[dashed] (6.5,-0.5) -- (6.5,0.5);
	\draw[<->] (3.5,0) -- (6.5,0) node[pos=0.5, yshift=7pt]{$d$};
	\draw[<->] (1,0.5) -- (3,0.5) node[pos=0.5, yshift=7pt]{$l$};
	\draw[<->] (0,-0.5) -- (3,-0.5) node[pos=0.5, yshift=-7pt]{$L$};
	\draw[<->] (7,0.5) -- (9,0.5) node[pos=0.5, yshift=7pt]{$l$};
	\draw[<->] (7,-0.5) -- (10,-0.5) node[pos=0.5, yshift=-7pt]{$L$};
	%	\draw[<->, opacity=0.5] (0,-0.8) -- (5,-0.8) node[pos=0.5, yshift=-7pt,]{$L$};
	\end{tikzpicture}
	\caption{An array of 14 spherical resonators formed by separating an array of 7 dimers in the centre by a dislocation distance $d>0$.} \label{fig:finite}
\end{figure}
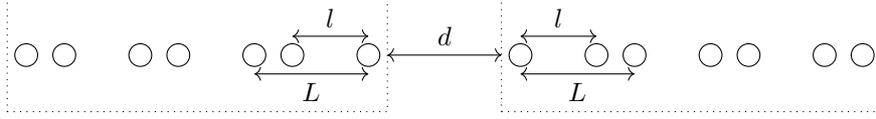

%We consider the simplest finite resonator array that reproduces the behaviour. This structure consists of six resonators, initially arranged as dimers then separated down the middle by a distance $d$ (\Cref{fig:finite}).

The resonant frequencies and eigenmodes of this finite system of resonators can be expressed in terms of the eigenpairs of the associated \emph{capacitance matrix}. Let $V_j, j=1,\cdots,M$, be the solution to
\begin{equation*} \label{eq:V}
\begin{cases}
\ds \Delta V_j =0 \quad &\mbox{in}~\R^3\setminus D,\\
\ds V_j = \delta_{ij} &\mbox{on}~\partial D_i, \ i=1,...,M,   	\\
\ds V_j(x) = O\left(\tfrac{1}{|x|}\right) & \text{as } |x|\rightarrow \infty.
\end{cases}
\end{equation*}
We then define the capacitance matrix $C=(C_{i,j})$ as
\begin{equation*} \label{eq:finite_capacitance}
C_{i,j} :=-\int_{\partial D_i}\frac{\partial V_j}{\partial \nu} \bigg|_+ \dx \sigma,\quad i,j=1,\cdots,M.
\end{equation*}

%Observe that as $\delta \rightarrow 0$, we have $\lambda \rightarrow \tfrac{1}{2}^+$. Then, using Gohberg-Sigal theory for operator-valued functions \cite{AKL, gohberg2009holomorphic} we have the following lemma.
%
%\begin{lem}
%	For any $\delta$ sufficiently small there are, up to multiplicity, $N$ characteristic values
%	$\omega_j= \omega_j(\delta), j = 1,...,N$, to the operator-valued analytic function
%	$\mathcal{A}(\omega, \delta)$
%	such that $\omega_j(0)=0$ for all $j$ and
%	$\omega_j$ depends on $\delta$ continuously.
%\end{lem}

The following theorem, first proved in \cite{ammari2017double}, shows that the eigenvalues of $C$ determine the resonant frequencies of the finite structure.
\begin{thm} \label{thm:char_approx_finite}
	The subwavelength resonant frequencies $\omega_j=\omega_j(\delta),~j=1,\cdots,M$, of $\mathcal{A}(\omega,\delta)$ can be approximated as
	$$ \omega_j= \sqrt{\frac{\delta \lambda_j }{|D_1|}}  + O(\delta),$$
	where $\lambda_j,~j=1,\cdots,M$, are the eigenvalues of the capacitance matrix $C$ and $|D_1|$ is the volume of each individual resonator.
\end{thm}

\subsection{Behaviour for large dislocations}
As the separation distance $d$ becomes large, the capacitance matrix converges to a block diagonal form. This is because, for large $d$, we have two systems of $M/2$ resonators, the interactions between which diminish with increasing $d$. This is made precise by the following lemma.

\begin{lem} \label{lem:caplarged}
	As the dislocation size $d\to\infty$, the capacitance matrix has the form
	\begin{equation*}
	C=\begin{pmatrix}\widetilde{C} & 0 \\ 0 & \widetilde{C}^{\star}\end{pmatrix} + O(d^{-1}),
	\end{equation*}
	where $\widetilde{C}$ is the capacitance matrix of the $M/2$-resonator system $D_1\cup\dots\cup D_{M/2}$ and	$\widetilde{C}^{\star}$ is the rearranged matrix given by
	\begin{equation*}
	\widetilde{C}_{i,j}^{\star}:=C_{M+1-i,M+1-j}.
	\end{equation*}
	%	If $i$ and $j$ are such that $|z_i-z_j|\to\infty$ as $d\to\infty$, then $C_{ij}\to0$ as $d\to\infty$.
\end{lem}
\begin{proof}
	We can use the jump conditions to show that the capacitance coefficients $C_{i,j}$ are given by
	$$ C_{i,j} = - \int_{\partial D_i} \psi_j \dx\sigma,\quad i,j=1,\cdots,M,$$
	where the functions $\psi_j$ are defined %\todo{Use this as the definition for C?}
	as
	$$\psi_j = (\S_D^{0})^{-1}[\chi_{\p D_j}].$$

	We make the identification $\p D=\p D_1 \times \dots\times\p D_M$ and use this to write the single layer potential $\mathcal{S}_D^0$ in a decomposed matrix form, as
	\begin{equation} \label{eq:s_decomp}
	\S_D^0=S_I + S_{II},
	\end{equation}
	where $S_I$ and $S_{II}$ are linear operators defined block-wise as
	\begin{align*}
	[S_I]_{ij}&:=\begin{cases}
	\mathcal{S}_{D_i}^0 |_{\p D_{j}}, & \text{if } i,j\leq M/2 \text{ or } i,j\geq M/2+1,\\
	0, & \text{otherwise},
	\end{cases}\\
	[S_{II}]_{ij}&:=\begin{cases}
	0, & \text{if } i,j\leq M/2 \text{ or } i,j\geq M/2+1,\\
	\mathcal{S}_{D_i}^0 |_{\p D_{j}}, & \text{otherwise}.
	\end{cases}
	\end{align*}

	The decomposition \eqref{eq:s_decomp} has been chosen so that $S_I$ contains precisely the parts of $\S_D^0$ that are unaffected by varying the parameter $d$. Conversely, based on the decay of Green's function $G^0$ we can see that, if $i\leq M/2$ and $j\geq M/2+1$ or vice versa, it holds that
	\begin{equation*}
	\| \mathcal{S}_{ D_j}^0|_{\p D_{i}}  \|_{\B(L^2(\p D_j), H^1(\p D_i))} =O(d^{-1}),
	\end{equation*}
	as $d\to\infty$, hence
	\begin{equation*}\label{SII_estim}
	\| S_{II}\|_{\B(L^2(\p D), H^1(\p D))} =O(d^{-1}).
	\end{equation*}
	Therefore, $\|S_I^{-1}S_{II}\| = O(d^{-1})$ so we may use a Neumann series to see that
	\begin{align*}
	(\mathcal{S}_D^{0})^{-1} [\chi_{\p D_j}] &= (S_I + S_{II})^{-1} [\chi_{\p D_j}] \nonumber
	\\
	&= (I + S_I^{-1} S_{II})^{-1} S_I^{-1}[\chi_{\p D_j}] \nonumber
	\\
	&= (I - S_I^{-1} S_{II} ) [\phi_j ] + O(d^{-1}), \label{SDalpha_inverse}
	\end{align*}
	where $\phi_j:= S_I^{-1}[\chi_{\p D_j}]$. Therefore,
	\begin{equation*}
	C_{i,j}=- \int_{\p D_i} (\mathcal{S}_D^0)^{-1} [\chi_{\p D_j}]\dx \sigma = - \int_{\p D_i}  (I - S_I^{-1} S_{II} ) [\phi_j] \dx \sigma + O(d^{-1}).
	\end{equation*}

	Suppose that $i\leq M/2$ and $j\geq M/2+1$, or vice versa. Then since $(S_I)^{-1}$ is also block diagonal we can see that $\phi_j|_{\p D_i}=0$ so $\int_{\p D_i} \phi_j \dx\sigma=0$. Thus, $C_{i,j}=O(d^{-1})$. Conversely, if $i,j\leq M/2$ then $(S_I^{-1} S_{II} ) [\phi_j]|_{\p D_i}=0$ so we find that
	\begin{align*}
	C_{i,j}&=- \int_{\p D_i}  \phi_j \dx \sigma + O(d^{-1})\\
	&= \widetilde{C}_{i,j} + O(d^{-1}).
	\end{align*}
	In the case that $i,j\geq M/2+1$ the result with $\widetilde{C}^\star$ follows similarly.
	%
	%.\\
	%.\\
	%\\
	% We then define the capacitance of each individual resonator $D_j$ as $\textrm{Cap}_{D_j} :=\nolinebreak -\int_{\p D_j}\phi_{j}$.
	%
	%Using the scaling property, we obtain that for $i\neq j$
	%\begin{align*}
	%\mathcal{S}_{ D_j}^0 \big|_{\p D_i} [\phi_j](x) &= \int_{\p D_j} \Big( - \frac{1}{4\pi |x - z_j|} + O(\epsilon) \Big) \phi_j(y) \dx \sigma (y)
	%= \frac{\textrm{Cap}_{D_j}}{4\pi |z_i - z_j|}  \chi_{\p D_i}(x) + O(\epsilon^2)
	%\\
	%&=\frac{\epsilon\textrm{Cap}_{B}}{4\pi |z_i - z_j|}  \chi_{\p D_i}(x) + O(\epsilon^2),
	%\end{align*}
	%where, the remainder terms are uniform with respect to $x\in \p D_i$ and $y\in \p D_j$.
	%Thus, we have
	%\begin{align*}
	%C_{ij} &=- \int_{\p D_i} (\mathcal{S}_D^0)^{-1} [\chi_{\p D_j}]\dx \sigma = - \int_{\p D_i}  (I - S_I^{-1} S_{II} ) [\phi_j] \dx \sigma + O(d^{-1}) \\
	%&= - \int_{\p D_i} \phi_j\dx \sigma + \int_{\p D_i}  S_I^{-1} \sum_{i'\neq j} \frac{\epsilon \textrm{Cap}_{B}}{4\pi |z_{i'}- z_j|}  \chi_{\p D_{i'}} \dx \sigma + O(\epsilon^3)
	%\\
	%&=- \int_{\p D_i} \phi_j \dx \sigma-  \sum_{i'\neq j} \frac{\epsilon \textrm{Cap}_{B}}{4\pi |z_{i'}- z_j|} \Big(-\int_{\p D_i}   \phi_{i'}\dx \sigma\Big) + O(\epsilon^3)
	%\\
	%&=\delta_{ij} \epsilon \textrm{Cap}_{B} - (1-\delta_{ij}) \frac{(\epsilon \textrm{Cap}_{B})^2}{4\pi |z_i - z_j|}  + O(\epsilon^3).
	%\end{align*}
	%
\end{proof}

\begin{rmk}
	At its heart, \Cref{lem:caplarged} is a consequence of the decay of the Helmholtz Green's function in free space and not a particular property of the system studied here. The dislocation of any general collection of (finitely many) resonators would yield a similar result (albeit without such elegant notation for the two blocks, which is a consequence of the structure's symmetry).
\end{rmk}

\begin{rmk}
	$\widetilde{C}^\star$ corresponds to the capacitance matrix of the $M/2$-resonator system $D_{M/2+1}\cup\dots\cup D_M$. This is the same system as that for which $\widetilde{C}$ is the capacitance matrix, but with the resonators labelled in the reverse order. That they have the same eigenvalues is easy to see from the fact that $\widetilde{C}^\star=J\widetilde{C}J$, where $J$ is the exchange matrix (1 on the off-diagonal and 0 elsewhere). Thus, in the limit as $d\to\infty$ the eigenvalues of $C$ converge pairwise to $M/2$ values.
\end{rmk}

\begin{figure}
	\begin{center}
	\begin{minipage}{0.48\linewidth}
	\includegraphics[width=\linewidth,trim={1cm 1cm 1cm 0}]{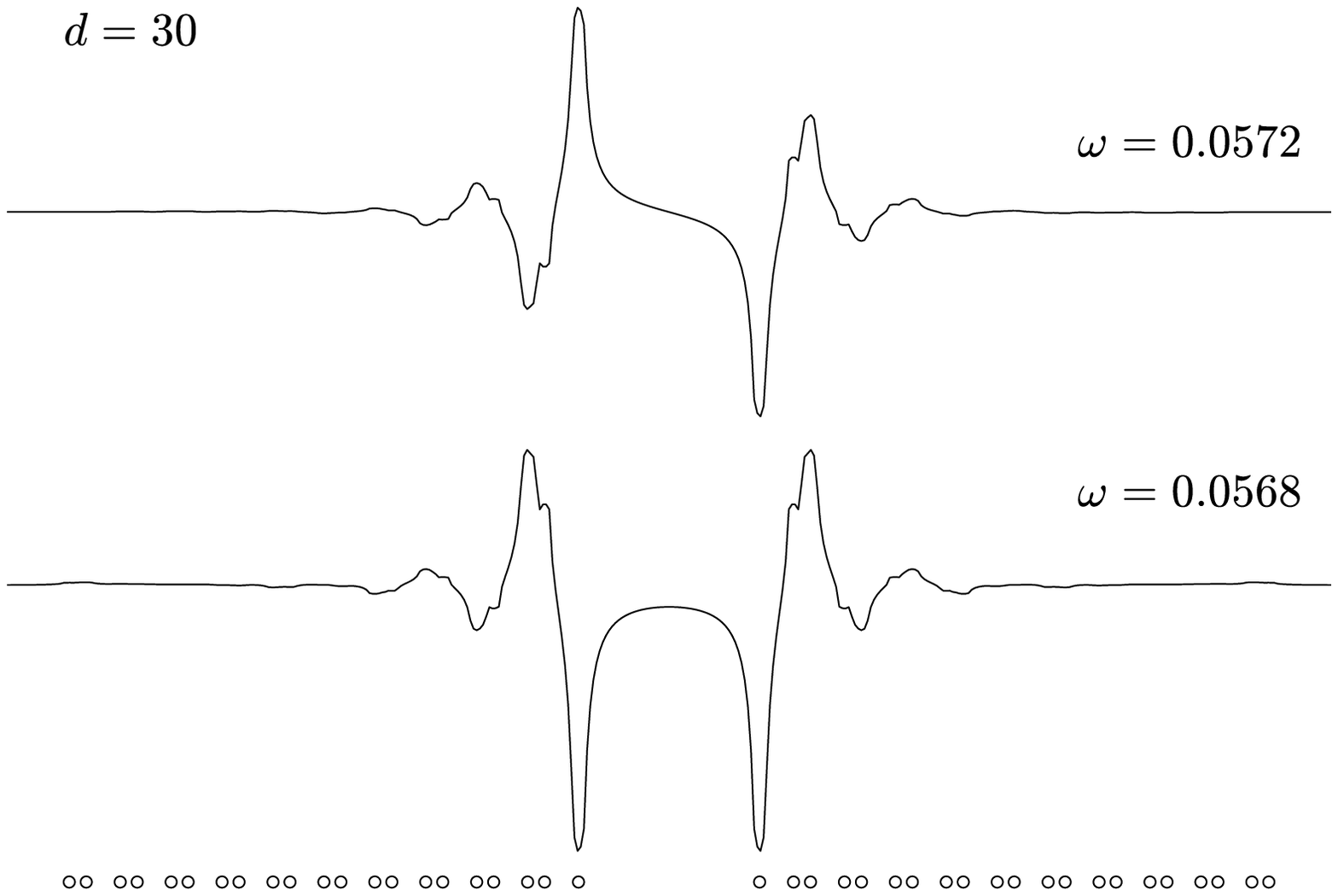}

	\includegraphics[width=\linewidth,trim={1.5cm 0 2cm 0},clip]{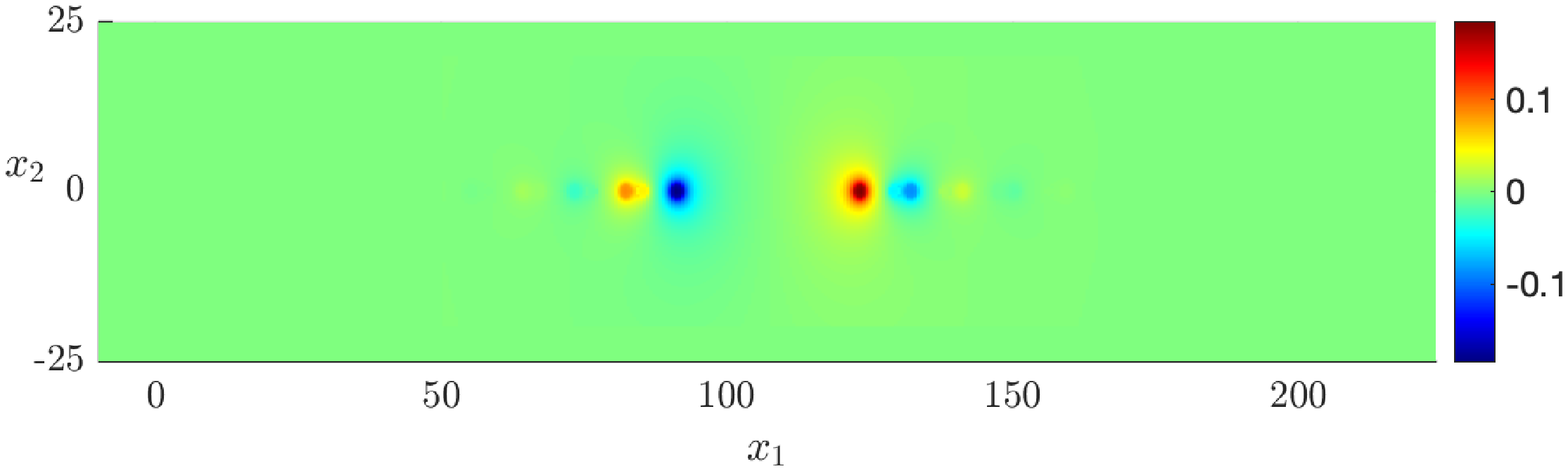}
	\end{minipage}
		\begin{minipage}{0.48\linewidth}
		\includegraphics[width=\linewidth,trim={1cm 0 -0.8cm 0}]{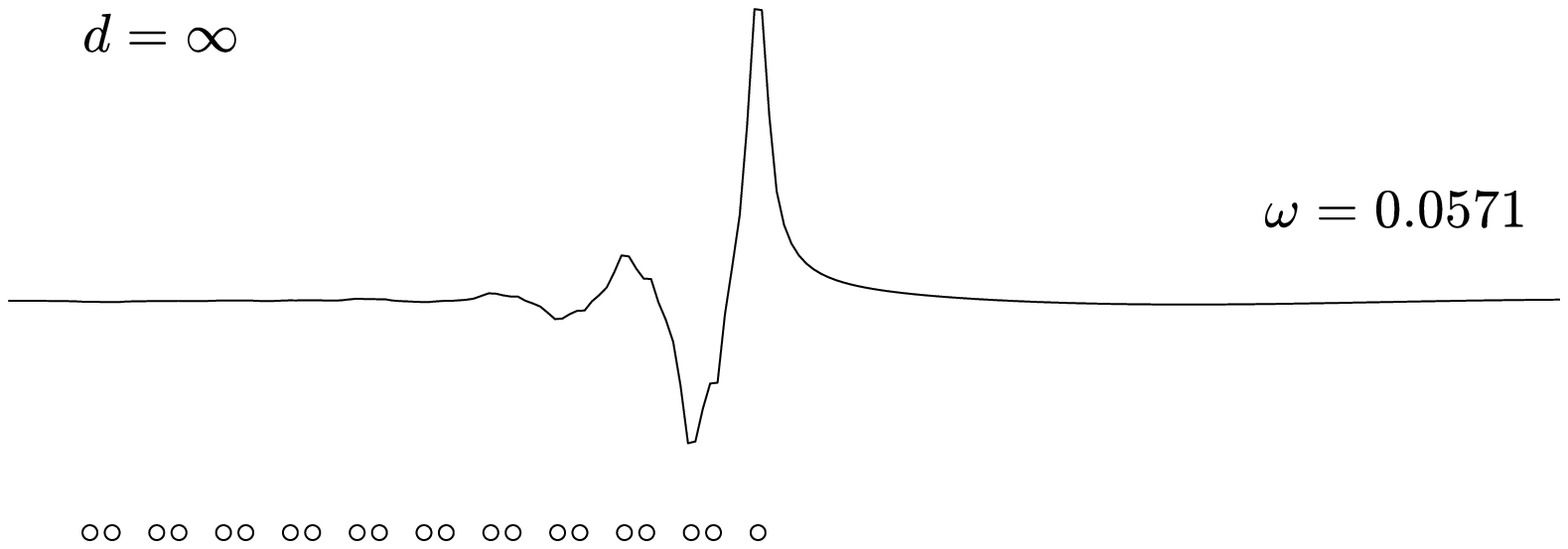}

		\includegraphics[width=\linewidth,trim={1.5cm 0 2cm 0}, clip]{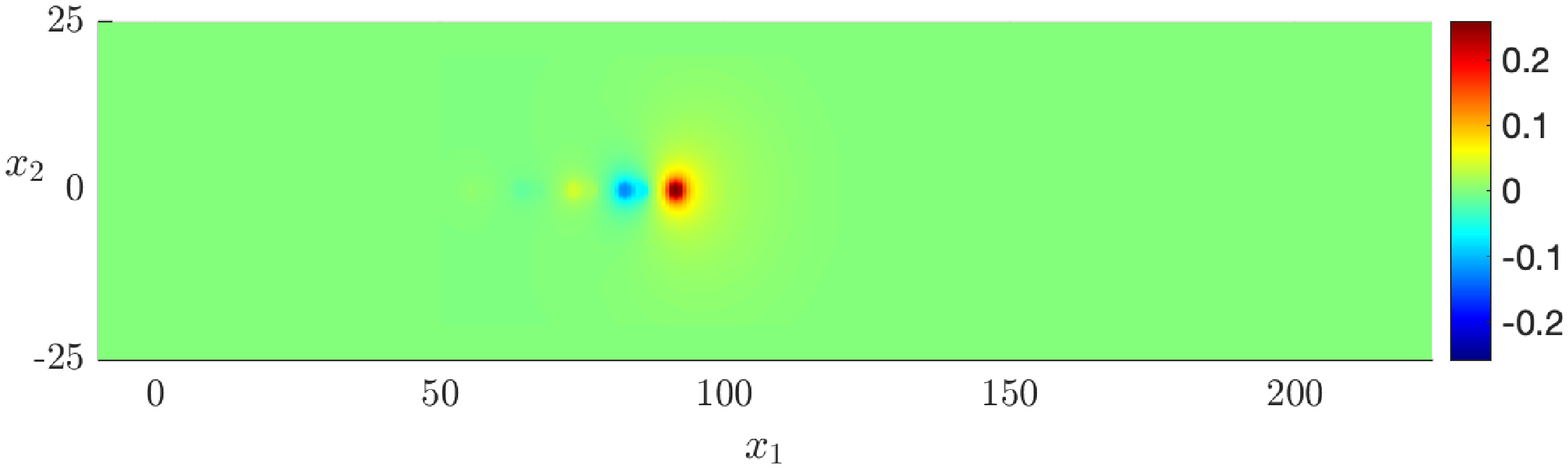}
		\end{minipage}
		% \vspace{0.2cm}

		% \includegraphics[width=0.48\linewidth,trim={1.5cm 0 2.1cm 0}, clip]{30mode3D.eps}
		% \includegraphics[width=0.48\linewidth,trim={1.5cm 0 2.1cm 0}, clip]{halfmode3D.eps}
		%		\includegraphics[width=0.48\linewidth,trim={0 0 0 8cm}]{/50mode.eps}
	\end{center}
	\caption{\textit{Left:} The two edge modes for an array of 42 spherical resonators with unit radius. Here, we simulate an array with parameters $L=9$, $l=6$, $d=30$ and $\delta=1/7000$ and plot the real parts of the edge modes along the line $x_2=0$, $x_3=0$. Below, we plot the $\omega=0.0572$ mode in the plane $x_3=0$, noting that the field has rotational symmetry about the $x_1$ axis. \textit{Right:} For comparison, the edge mode of the corresponding `half system' is shown, which can be thought of as the $d=\infty$ case.} \label{fig:eigenmodes}
\end{figure}

The behaviour for large $d$ can be understood by examining the eigenmodes, examples of which are given in \Cref{fig:eigenmodes}. The dislocation splits the structure into two ``half structures'' which interact with one another. This coupling leads to the creation of two resonant modes, with monopole- and dipole-like characteristics (\emph{cf.} \cite{ammari2017double}), which are the two edge modes.

\subsection{Stability analysis}
We consider the simplest example of a resonator array of the form \eqref{eq:struc_finite}, which has just six resonators arranged as three pairs. The geometry of this structure is parametrised by $l$ and $L$ (as in \Cref{fig:finite}). We wish to study how robust the system is with respect to variations in these parameters.

We know from \Cref{lem:caplarged} that as $d\to\infty$ this system will behave like two separate three-resonator systems. Even in the case of a three-resonator system, finding explicit representations for the entries of the capacitance matrix (with a view to \emph{e.g.} calculating its eigenvalues) is a challenging problem. Consider the case of a \emph{dilute} array of resonators: that is, a structure where the distances between the resonators ($l$ and $L$) are much larger than the size of each individual resonator. In this case, we can recall the following representation of the capacitance matrix, proved in \cite{ammari2019topological}.

\begin{lem}\label{lem:cap_estim}
	%	We assume the system $D$ consists of $N$ identical resonators of size of order $\epsilon$, i.e., $D=\cup_{j=1}^M \epsilon B + z_j$. We also assume the dilute limit, i.e., smallness of $\epsilon$.
	Consider a dilute system of $M$ identical subwavelength resonators with size of order $\epsilon$, given by
	\begin{equation*}
	D=\bigcup_{j=1}^M \left(\epsilon B + z_j\right),
	\end{equation*}
	where $0<\epsilon\ll1$, $B$ is a fixed domain of unit size and $z_j$ represents the translated position of each resonator. In the limit as $\epsilon\rightarrow 0$, the capacitance matrix is given by
	%	Then we have the following asymptotics of the capacitance matrix  as :
	\begin{equation} \label{lem:c_dilute}
	C_{i,j} =
	\begin{cases}
	\displaystyle \epsilon \mathrm{Cap}_B + O(\epsilon^3), &\quad \text{if} \ i=j,\\
	\displaystyle -\frac{\epsilon^2(\mathrm{Cap}_B)^2}{4\pi|z_i - z_j|} + O(\epsilon^3), &\quad \text{if} \ i\neq j,\\
	\end{cases}
	\end{equation}
	where $\mathrm{Cap}_B:=-\int_{\p B} (\S_B^0)^{-1}[\chi_B]\dx\sigma$.
\end{lem}

In the case of a three-resonator system with $|z_1-z_2|=l$ and $|z_1-z_3|=L$, we can use the expansion \eqref{lem:c_dilute} to show that the eigenvalues of the capacitance matrix are given, as $\epsilon\to0$, by
\begin{equation} \label{eq:eigenvalues3}
\lambda_k=\epsilon \mathrm{Cap}_B+\epsilon^2\frac{(\mathrm{Cap}_B)^2\gamma}{2\sqrt{3}\pi}%\sqrt{l^{-2}+L^{-2}+(L-l)^{-2}}
\cos\left[
\frac{1}{3}\left(\arccos\left(\frac{-3\sqrt{3}}{lL(L-l)\gamma^{3}}\right)+2k\pi\right)
\right] +O(\epsilon^3),
\end{equation}
for $k=1,2,3$, where $\gamma=\gamma(l,L):=\sqrt{l^{-2}+L^{-2}+(L-l)^{-2}}$. The convergence of the six resonant frequencies of the six-resonator system to these three values is demonstrated in \Cref{fig:six_d}.

We know, from \Cref{sec:periodic}, that the undislocated structure ($d=0$) has a subwavelength band gap if it is asymmetric, \emph{i.e.} $l/L\neq 1/2$. In the case of a sufficiently asymmetric structure, we can show that the middle eigenvalue is more stable with respect to changes in the parameter $l$, which controls the relative positions of the two repeating resonators. This is achieved by \Cref{lem:derivs}, which describes the extent to which the eigenvalues \eqref{eq:eigenvalues3} are affected by variations in the parameters $l$ and $L$. In particular, it says that if $l':=L-l$ is sufficiently small then
\begin{equation*}
\left|\frac{\p \lambda_2}{\p l}\right| \ll \left|\frac{\p \lambda_1}{\p l}\right|, \quad \left|\frac{\p \lambda_2}{\p l}\right| \ll\left|\frac{\p \lambda_3}{\p l}\right|,
\end{equation*}
and that the dependence of all three eigenvalues on $L$ is comparatively negligible.

\begin{lem} \label{lem:derivs}
	%	For sufficiently small $l>0$ it holds that
	%	$$\left|\frac{\dx \lambda_1}{\dx L}\right|,\left|\frac{\dx \lambda_2}{\dx L}\right|,\left|\frac{\dx \lambda_3}{\dx L}\right|\ll\left|\frac{\dx \lambda_2}{\dx l}\right| \ll \left|\frac{\dx \lambda_1}{\dx l}\right|,\left|\frac{\dx \lambda_3}{\dx l}\right|.$$
	%
	Let $l':=L-l$. As $l'\to0^+$, it holds that
	$$\left|\frac{\p \lambda_1}{\p l}\right|\to\infty,\quad
	\left|\frac{\p \lambda_2}{\p l}\right|=O(1), \quad
	\left|\frac{\p \lambda_3}{\p l}\right|\to\infty.$$
	%	$$\left|\frac{\dx \lambda_2}{\dx l}\right| \ll \left|\frac{\dx \lambda_1}{\dx l}\right|, \quad \left|\frac{\dx \lambda_2}{\dx l}\right| \ll\left|\frac{\dx \lambda_3}{\dx l}\right|,$$
	%	and, further, we have that
	%
	%	$$\left|\frac{\dx \lambda_2}{\dx l}\right| =O(1), \quad \left|\frac{\dx \lambda_1}{\dx l}\right|, \quad \left|\frac{\dx \lambda_2}{\dx l}\right| \ll\left|\frac{\dx \lambda_3}{\dx l}\right|,$$
	Meanwhile, for $k=1,2,3$,
	$$\left|\frac{\p \lambda_k}{\p L}\right| =O(l').$$
	%	for $k=1,2,3$.
	%	\ll \left|\frac{\dx \lambda_1}{\dx L}\right|, \quad \left|\frac{\dx \lambda_2}{\dx L}\right| \ll\left|\frac{\dx \lambda_3}{\dx L}\right|. $$
\end{lem}
\begin{proof}
	%Observe first that as $l\to0$, $l^{-1}L^{-1}(L-l)^{-1}\gamma(l,L)^{-3}\to0$ and hence
	%\begin{equation*}
	%\arccos\left(\frac{-3\sqrt{3}}{lL(L-l)\gamma^{3}}\right)\to\frac{\pi}{2}.
	%\end{equation*}
	%
	Define the functions
	\begin{equation*}
	c(l',L,k):=\cos\left[
	\frac{1}{3}\left(\arccos\left(\frac{-3\sqrt{3}}{l'L(L-l')\gamma(l',L)^{3}}\right)+2k\pi\right)\right],\quad 0<l'<L,\,k=1,2,3,
	\end{equation*}
	and
	\begin{equation*}
	s(l',L,k):=\sin\left[
	\frac{1}{3}\left(\arccos\left(\frac{-3\sqrt{3}}{l'L(L-l')\gamma(l',L)^{3}}\right)+2k\pi\right)\right],\quad 0<l'<L,\,k=1,2,3.
	\end{equation*}
	%as well as the constant $A:=\epsilon^2(\mathrm{Cap}_B)^2/2\sqrt{3}\pi$

	As $l'\to0^+$ it holds that
	\begin{equation} \label{eq:cs_limits}
	\begin{aligned}
	c(l',L,1)&\to -\frac{\sqrt{3}}{2}, &c(l',L,2)&\to 0, &c(l',L,3)&\to \frac{\sqrt{3}}{2},\\
	s(l',L,1)&\to \frac{1}{2}, &s(l',L,2)&\to -1, &s(l',L,3)&\to \frac{1}{2}.
	\end{aligned}
	\end{equation}
	In addition to this, for fixed $L$ and $k$ we see that, as $l'\to0^+$,
	\begin{equation*} \label{eq:dl}
	\frac{\p \lambda_k}{\p l'}\sim\frac{ \epsilon^2(\mathrm{Cap}_B)^2}{2\sqrt{3}\pi}\left[-\frac{1}{(l')^2}c(l',L,k)-\frac{2\sqrt{3}}{L^2}s(l',L,k)\right],
	\end{equation*}
	where the notation $\sim$ is used to mean that $f\sim g$  if and only if $\lim f/g=1$. From this and \eqref{eq:cs_limits} we can see that, as $l'\to0^+$,
	\begin{equation*}
	\frac{\dx \lambda_1}{\dx l'}\to\infty,\quad \frac{\dx \lambda_3}{\dx l'}\to-\infty.
	\end{equation*}
	Conversely, using Taylor series expansions we can see that, as $l'\to0^+$,
	\begin{equation*}
	c(l',L,2)= \frac{\sqrt{3}}{L^2}(l')^2+O\left((l')^3\right),
	\end{equation*}
	hence as $l'\to0^+$ it holds that
	\begin{equation*}
	\frac{\p \lambda_2}{\p l'}\to \frac{\epsilon^2(\mathrm{Cap}_B)^2}{2\pi L^2}.
	\end{equation*}

	Likewise, the result for $\dx \lambda_k/\dx L$ follows from the fact that, as $l'\to0^+$,
	\begin{equation*} \label{eq:dL}
	\frac{\dx \lambda_k}{\dx L}\sim\frac{ \epsilon^2(\mathrm{Cap}_B)^2}{2\sqrt{3}\pi} \left[-\frac{2l'}{L^3}c(l',L,k)-\frac{2\sqrt{3}l'}{L^3}s(l',L,k)\right],
	\end{equation*}
	for $k=1,2,3$.
	%\begin{equation*}
	%\frac{\dx \lambda_k}{\dx L}=O(l).
	%\end{equation*}
\end{proof}

\begin{figure}[t]
	\begin{center}
		\begin{subfigure}[b]{0.45\linewidth}
			\includegraphics[width=\linewidth]{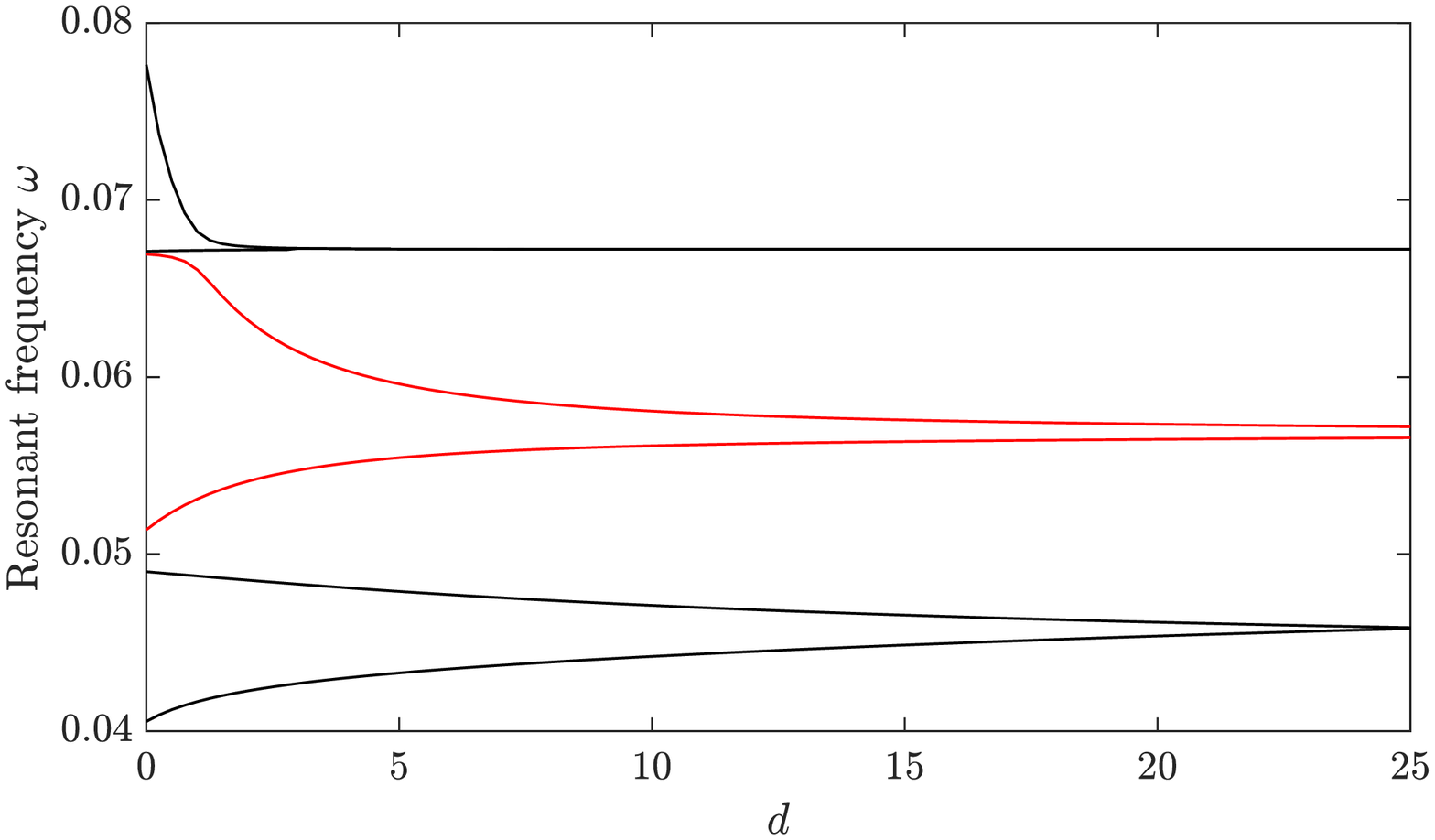}
			\caption{Resonator array with 6 resonators.} \label{fig:six_d}
		\end{subfigure}
		\hspace{10pt}
		\begin{subfigure}[b]{0.45\linewidth}
			\includegraphics[width=\linewidth]{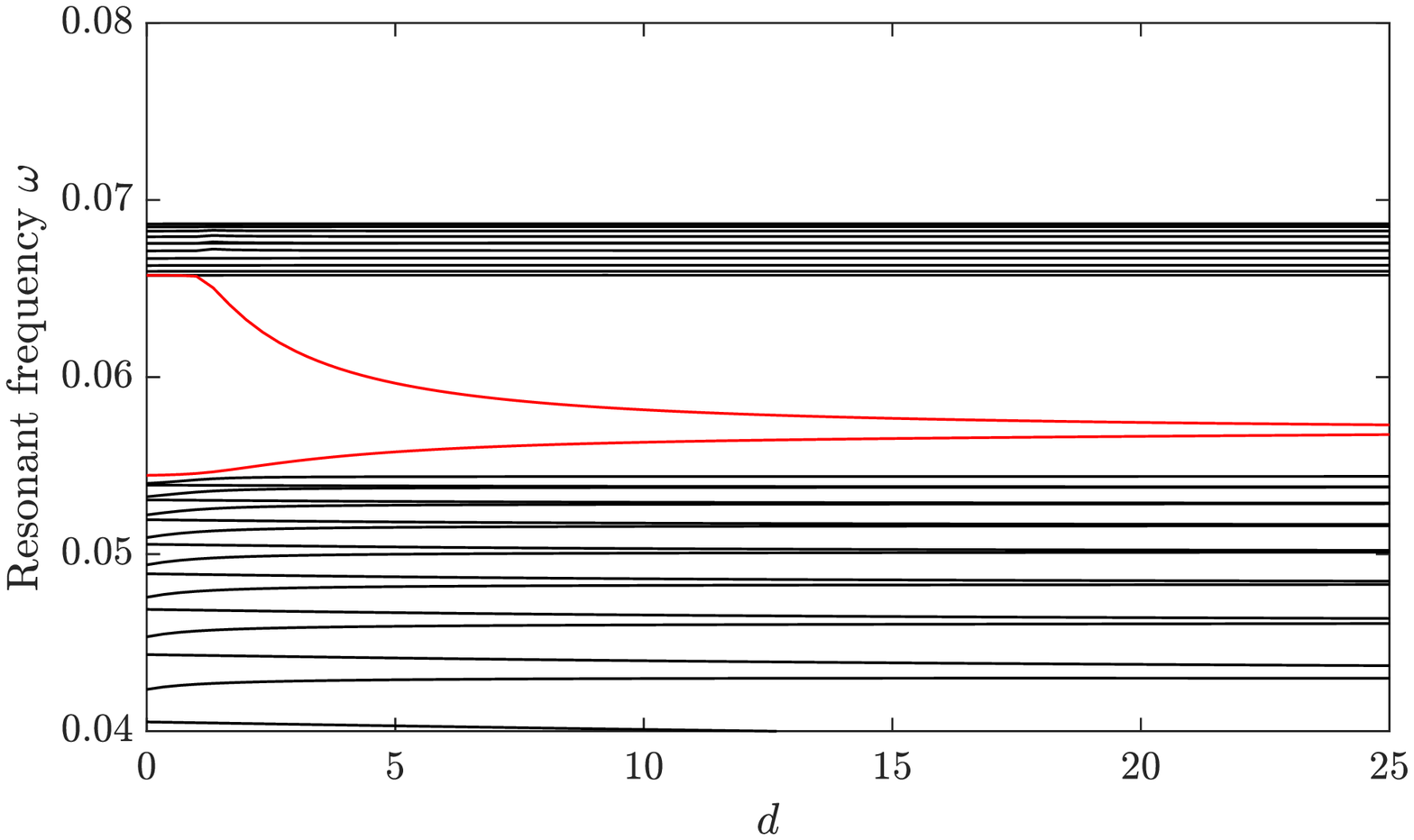}
			\caption{Resonator array with 42 resonators.} \label{fig:fortytwo_d}
		\end{subfigure}
		\caption{Simulation of the resonant frequencies of different subwavelength resonator arrays as the dislocation $d$ is increased.}\label{fig:num_find}
		\vspace{20pt}
		\begin{subfigure}[b]{0.45\linewidth}
			\includegraphics[width=\linewidth]{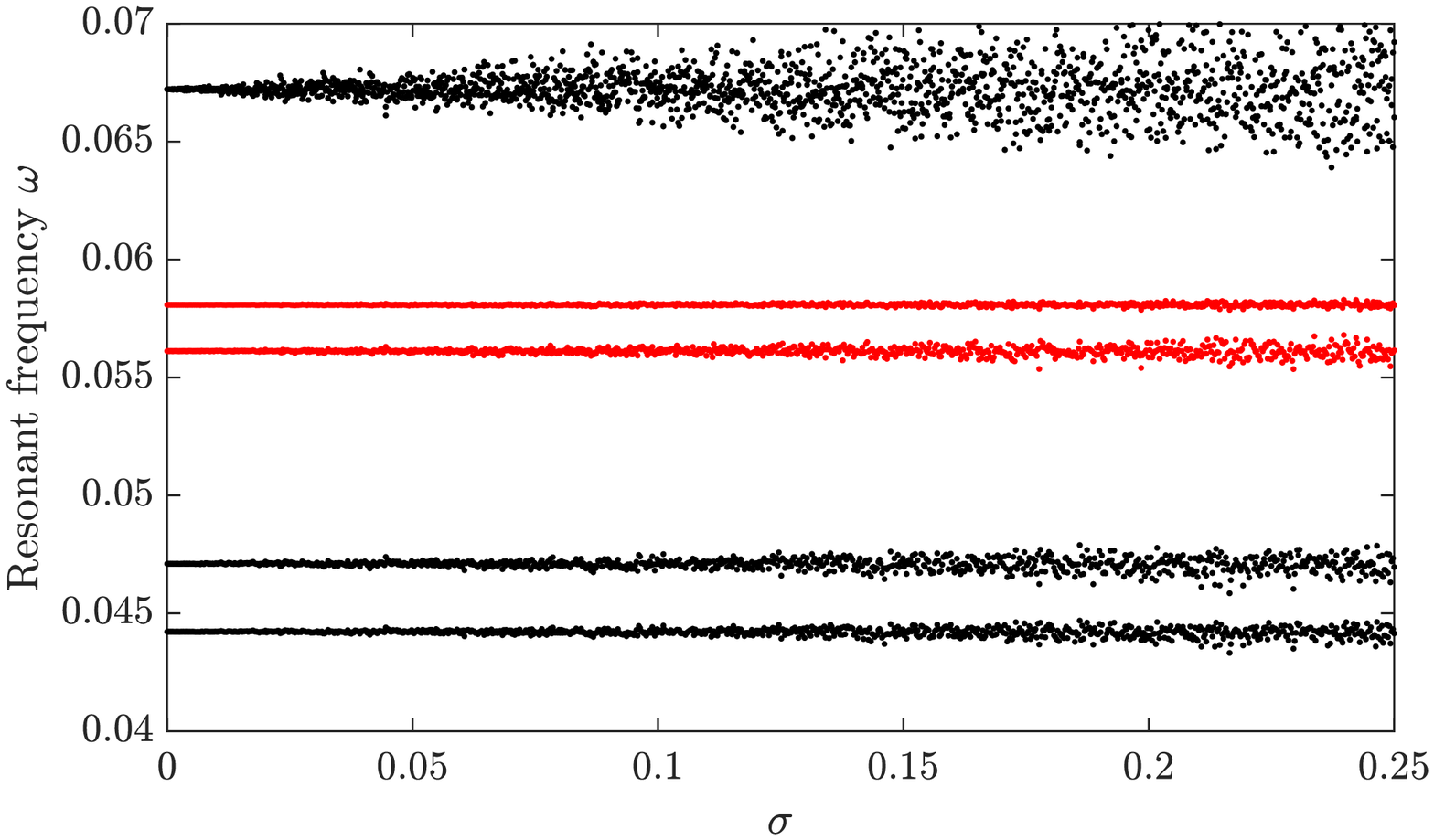}
			\caption{Increasing imperfection deviation $\sigma$, for $d=10$.} \label{fig:stabsig}
		\end{subfigure}
		\hspace{10pt}
		\begin{subfigure}[b]{0.45\linewidth}
			\begin{tikzpicture}
			\node at (0,0) {\includegraphics[width=\linewidth]{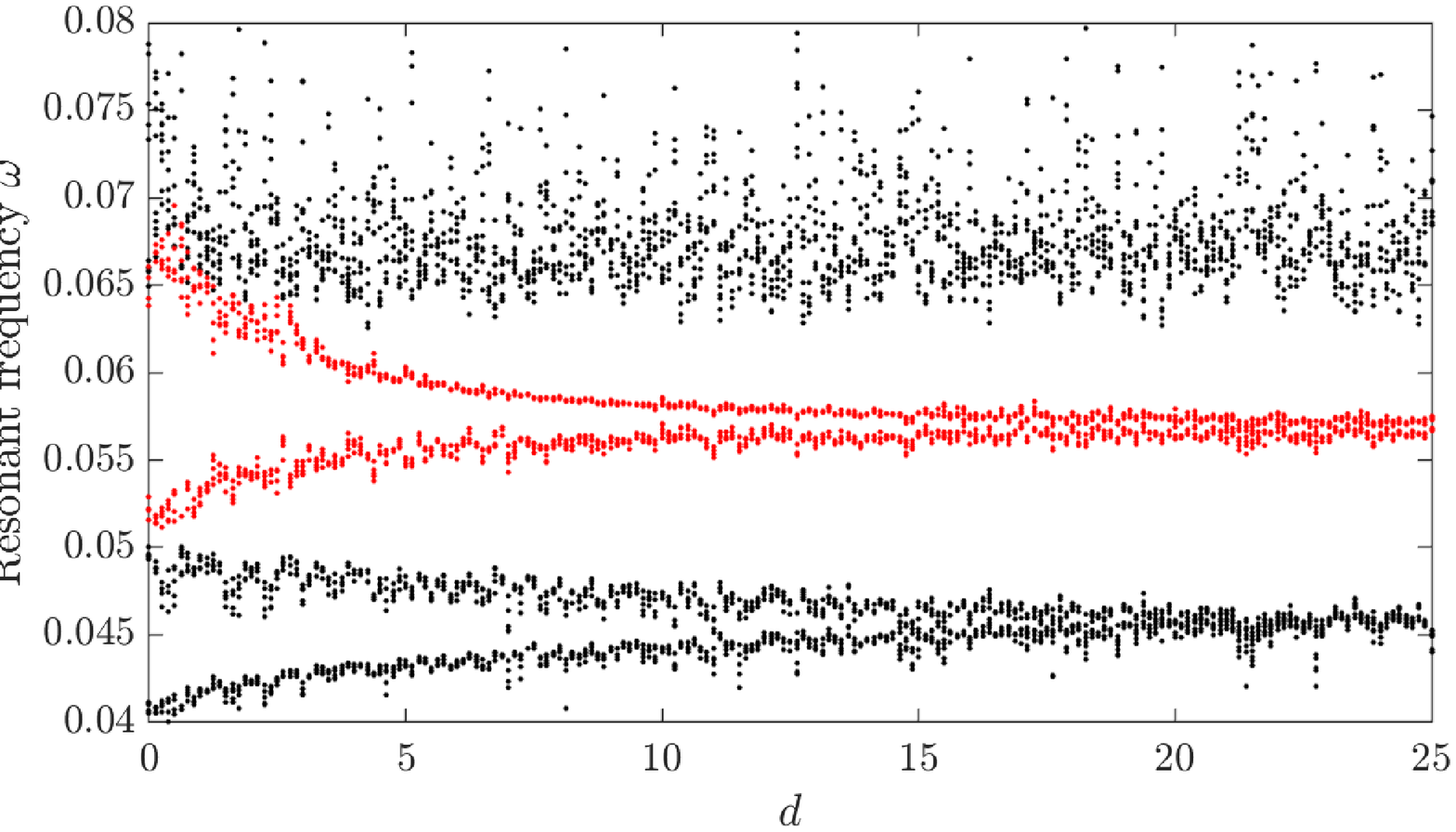}};
			\draw[->,blue,thick] (-0.3,0.3) -- (-0.5,0.15);
			\end{tikzpicture}
			% 			\missingfigure[figwidth=\linewidth]{42 resonators, stability as a function of $\sigma$}
			\caption{Increasing dislocation $d$, for $\sigma=0.2$.} \label{fig:stabd}
		\end{subfigure}
		\caption{Analysis of the stability of the resonant frequencies of a system of six resonators. An array of six resonators with dislocation size $d$ is repeatedly simulated after random imperfections, drawn from the distribution $\mathcal{N}(0,\sigma^2)$, are introduced to the resonator positions. An arrow indicates the position of minimum variance.} \label{fig:num_stab}
		% 		\vspace{20pt}
		% 		\begin{subfigure}[b]{0.45\linewidth}
		% 			\includegraphics[width=\linewidth]{/6stab_d_sig01.eps}
		% 			\caption{$\sigma=0.1$} \label{fig:six_stabd}
		% 		\end{subfigure}
		% 		\hspace{10pt}
		% 		\begin{subfigure}[b]{0.45\linewidth}
		% 			\includegraphics[width=\linewidth]{/6stab_dvar_sig025.eps}
		% 			\caption{$\sigma=0.25$} \label{fig:fortytwo_stabd}
		% 		\end{subfigure}
		%		\vspace{20pt}
		% 		\caption{Simulation of the resonant frequencies of an array of 6 subwavelength resonators as the dislocation size $d$ is increased. In each case, independent random imperfections $\mathcal{N}(0,\sigma^2)$ .} \label{fig:num_stabd}
		\caption*{In \Cref{fig:num_find,fig:num_stab}, simulations were performed on spherical resonators with radius 1 arranged with distances $L=9$ and $l=6$ (as depicted in \Cref{fig:finite}) and material contrast $\delta = 1/7000$. The multipole expansion method was used to find the subwavelength resonant frequencies associated to $\mathcal{A}$ (see the appendix of \cite{ammari2019topological} for details).}
	\end{center}
\end{figure}

The stability that is predicted by \Cref{lem:derivs} can be investigated numerically by repeatedly introducing random imperfections to the structure. In \Cref{fig:num_stab} we show the resonant frequencies for structures with random (Gaussian) perturbations added to the $x_1$ coordinate of the resonators' positions. It can, firstly, be observed that the middle eigenvalues (which both converge to $\lambda_2$, as defined in \eqref{eq:eigenvalues3}, as $d\to\infty$) are more stable, as expected. It is also interesting to observe how the stability varies as a function of the dislocation $d$. The minimal variance of any resonant frequency is observed for $\omega_4$ when $d\approx8$, as indicated by the arrow in \Cref{fig:stabd}. At this point, $\omega_4$ is in the centre of the band gap so it is as far as possible from the other (unlocalized) modes, consistent with the general principle \emph{e.g.} \cite{combes1973asymptotic, kuchment2}. This demonstrates the value of being able to control the position of mid-gap frequencies within the band gap.

\section{Concluding remarks}

In this paper, we have studied a one-dimensional array of subwavelength resonators capable of robustly manipulating waves on subwavelength scales and have proved that its properties can be fine tuned by adjusting geometric parameters. This takes advantage of the principle that eigenmodes corresponding to mid-gap frequencies that are far from the edges of that band gap will be strongly localized in space and will be robust with respect to structural imperfections. Thus, the goal was to design a structure that could be manipulated so as to place a mid-gap frequency at any given point within the band gap. This was achieved by introducing a dislocation to an array of subwavelength resonator pairs. We proved that the mid-gap frequencies emerge from the edges of the band gap and span an interval in the middle of the band gap.

Our study of the periodic structure was complemented by an analysis of the corresponding finite array of resonators. Created by truncating the infinite array, this physically-realizable structure shared the spectral behaviour of the infinite array. Further, a stability analysis confirmed the value of being able to fine-tune the structure in order to optimise robustness.

In the setting of the Schr\"odinger operator, two-dimensional structures exhibiting edge modes have been studied via the bulk-edge correspondence. It is well known that materials with non-zero bulk index can be achieved, for example, by perturbing honeycomb-like materials exhibiting \emph{Dirac cones} \cite{fefferman2012honeycomb, fefferman2016edge, drouot2019bulk}. Dirac cones have also been shown to exist in two-dimensional honeycomb structures of subwavelength resonators \cite{ammari2018honeycomb}, suggesting the potential for analogous results in this setting.

\appendix
\section{Proofs of \Cref{lem:C} and \Cref{lem:PhiBPsi}}
Here, we give proofs of \Cref{lem:C} and \Cref{lem:PhiBPsi}. Qualitatively, these results describe the strength of the fictitious source interactions in the two cases studied in \Cref{sec:smalldis} and \Cref{sec:no//}, respectively.

\subsection{Proof of \Cref{lem:C}} \label{app:lem1}
We will expand $\S_D^\omega$ and $\K_D^{\omega,*}$ in the dilute regime specified by \eqref{eq:dilute}. To keep the order of the norms in $L^2(\p D)$ and $H^1(\p D)$ constant as $\epsilon\rightarrow 0$, we let $\L$ and $\Hc$, respectively, denote the spaces $L^2(\p D)$ and $H^1(\p D)$ along with the inner products
$$\langle \cdot , \cdot\rangle_{\L} = \frac{1}{|\p D|}\langle \cdot , \cdot\rangle_{L^2(\p D)}, \qquad \langle \cdot , \cdot\rangle_{\Hc} = \frac{1}{|\p D|}\langle \cdot , \cdot\rangle_{H^1(\p D)}.$$
	Recall the matrix form of $\S_D^\omega$:
	$$\S_D^{\omega} = \begin{pmatrix} \S_{D_1}^\omega & \S_{D_2}^\omega\big|_{\p D_1} \\ \S_{D_1}^\omega\big|_{\p D_2} & \S_{D_2}^\omega \end{pmatrix} =
	\hat \S_D^\omega + \begin{pmatrix} 0 & \S_{D_2}^\omega\big|_{\p D_1} \\ \S_{D_1}^\omega\big|_{\p D_2} & 0 \end{pmatrix}.$$
	We define the centres $z_1, z_2$ of the resonators in the dilute regime specified by \eqref{eq:dilute}:
	$$z_1 = -\frac{l}{2}\mathbf{v}, \qquad z_2 = \frac{l}{2}\mathbf{v}.$$
	Then, as $\epsilon \rightarrow 0$, we have for $i\neq j$,
	\begin{align*}
	\S_{D_j}^0 \big|_{\p D_i} [\phi](x) &= \int_{\p D_j} \Big(G^0(x,z_j) + (y-z_j) \cdot \nabla_y G^0(x,y_0) \Big) \phi(y) \dx \sigma (y)
	\\
	&=-\frac{\chi_{\p D_i}(x)}{4\pi l}\int_{\p D_j} \phi(y) \dx \sigma(y)
	+ O\left( \epsilon \int_{\p D_j} |\phi(y)| \dx \sigma(y)\right).
	\end{align*}
	Here, $y_0$ means a point on the line segment joining $y$ and $z_j$. By the  Cauchy-Schwarz inequality we have $\int_{\p D_j} \phi = O(\epsilon^2\|\phi\|_\L)$. Hence we have
	\begin{align}\label{eq:Sdilute}
	\S_D^0 &= \hat{\S}_D^0 -\frac{1}{4\pi l}\begin{pmatrix} 0 & \langle \chi_{\p D_2}, \cdot \rangle \chi_{\p D_1} \\ \langle \chi_{\p D_1}, \cdot \rangle \chi_{\p D_2} & 0 \end{pmatrix} + O(\epsilon^3) \nonumber \\
	&=\hat{\S}_D^0 + \S_D^{(1)} + O(\epsilon^3),
	\end{align}
	where $\hat{\S}_D^0 = O(\epsilon)$ and $\S_D^{(1)} = O(\epsilon^2)$. Here, the error terms are with respect to the operator norm in $\B(\L,\Hc)$. In the same way, we can compute
	\begin{align}\label{eq:Kdilute}
	\K_D^{0,*} &= \hat{\K}_D^{0,*} + \frac{\mathbf{v}\cdot \nu}{4\pi l}\begin{pmatrix} 0 & -\langle \chi_{\p D_2}, \cdot \rangle \\ \langle \chi_{\p D_1}, \cdot \rangle & 0 \end{pmatrix} + O(\epsilon^3) \nonumber \\
	&=\hat{\K}_D^{0,*} + \K_D^{(1)} + O(\epsilon^3),
	\end{align}
	with respect to the operator norm in $\B(\L)$. Following the computations in the proof of Lemma~3.3 of \cite{ammari2019topological}, we have for $\alpha \neq 0$
	\begin{align*}
	\psi_1^\alpha &= \psi_1 + \epsilon\textrm{Cap}_{B}\sum_{m\neq 0}\frac{e^{\iu m\alpha L}}{4\pi |m|L}\psi_1
	+ \epsilon\textrm{Cap}_{B}\sum_{m\in\mathbb{Z}}\frac{e^{\iu m\alpha L}}{4\pi |l-mL|}\psi_2 + O(\epsilon), \\
	\psi_2^\alpha &= \psi_2 + \epsilon\textrm{Cap}_{B}\sum_{m\in\mathbb{Z}}\frac{e^{\iu m\alpha L}}{4\pi |l+mL|}\psi_1 +  \epsilon\textrm{Cap}_{B}\sum_{m\neq 0}\frac{e^{\iu m\alpha L}}{4\pi |mL|}\psi_2 + O(\epsilon),
	\end{align*}
	where the error terms are with respect the norm in $\L$. In these equations, observe that $\|\psi_i\|_{\L} = O(\epsilon^{-1})$. At $\alpha = \pi/L$, $u_j^\diamond$ and $u_j$ correspond to either monopole or dipole modes:
	$$u_j^\diamond = \frac{1}{\sqrt{2}}\left(\pm\psi_1^\diamond  + \psi_2^\diamond\right), \qquad u_j = \frac{1}{\sqrt{2}}\left(\pm\psi_1 + \psi_2\right).$$
	The sign is positive, corresponding to a monopole mode, if $l_0 < 1/2$ and  $j=1$ or $l_0 > 1/2$ and $j=2$, and negative if $l_0 < 1/2$ and $j=2$ or $l_0 > 1/2$ and $j=1$.
	Hence, from the expansions of $\psi_1^\alpha$ and $\psi_2^\alpha$ it follows that
	\begin{equation}\label{eq:udilute}
	u_j^\diamond = u_j + \epsilon u_j^{(1)}u_j + O(\epsilon),
	\end{equation}
	in $\L$, where
	$$u_j^{(1)} = \begin{cases}
	\ds \textrm{Cap}_{B}\left(\sum_{m\in\mathbb{Z}}\frac{(-1)^m}{4\pi |l+mL|} - \frac{\log(2)}{4\pi L} \right) & l_0 < 1/2, j=1 \quad \mathrm{or} \quad l_0 > 1/2, j=2 , \\
	\ds \textrm{Cap}_{B}\left(-\sum_{m\in\mathbb{Z}}\frac{(-1)^m}{4\pi |l+mL|} - \frac{\log(2)}{4\pi L} \right) & l_0 < 1/2, j=2 \quad \mathrm{or} \quad l_0 > 1/2, j=1 .
	\end{cases}$$
	From \cite{ammari2019topological} we have that
	$$\begin{cases} u_j^{(1)} < 0, \quad & l_0 < 1/2, j=1 \quad \mathrm{or} \quad l_0 < 1/2, j=2, \\
	u_j^{(1)} > 0, \quad & l_0 < 1/2, j=2 \quad \mathrm{or} \quad l_0 > 1/2, j=1.
	\end{cases}$$

	We are now ready to compute $B\Psi_j^\diamond$. Recall that $B = \P_2\A\P_1 - \A$. Since $$\P_i = I + d\P_i^{(1)} + O(d^2),$$
	with respect to the operator norm in $\B((L^2(\p D))^2)$ we have
	$$B = d\left(\P_2^{(1)}\A + \A \P_1^{(1)}\right) + O(d^2).$$
	Moreover, we compute
	\begin{align*}
	\A\P_1^{(1)}\begin{pmatrix}u_j \\ u_j^\diamond \end{pmatrix} = \begin{pmatrix} \mathbf{v}\cdot\left( \nabla \hat\S_D^{\omega}\big |_-[u_j] - \nabla\S_D^{\omega}\big |_+[u_j^\diamond]\right) \\ \left(-\frac{1}{2}+\hat\K_D^{\omega,*}\right)[\xi_1] -\delta \left(\frac{1}{2}+\K_D^{\omega,*}\right)[\xi_2] \end{pmatrix},
	\end{align*}
	where
	$$\xi_1 = \left(\hat\S_D^{\omega}\right)^{-1} \mathbf{v}\cdot \nabla \hat\S_D^{\omega}\big |_-[u_j], \qquad \xi_2 =  \left(\S_D^{\omega}\right)^{-1} \mathbf{v}\cdot \nabla \S_D^{\omega}\big |_+[u_j^\diamond].$$
	Hence %\todo{clarify}
	\begin{align*}
	\langle \Phi_j^\diamond, \A\P_1^{(1)} \Psi_j^\diamond\rangle &= -\delta\left\langle u_j^\diamond, \mathbf{v}\cdot\left( \nabla \hat\S_D^{0}\big |_-[u_j] - \nabla\S_D^{0}\big |_+[u_j^\diamond]\right) \right\rangle +  \left\langle \left(-\frac{1}{2}+\hat\K_D^{\omega}\right)[\chi_j^\diamond], \xi_1\right\rangle \\
	& -\delta\left\langle \left(\frac{1}{2}+\K_D^{0}\right)[\chi_j^\diamond],\xi_2\right\rangle \\
	&= \delta\left\langle u_j^\diamond, \mathbf{v}\cdot \nabla\S_D^{0}\big |_+[u_j^\diamond]\right\rangle + \omega^2\left\langle \hat\K_{D,2}[\chi_j^\diamond], \left(\hat\S_D^{0}\right)^{-1} \p_T\hat\S_D^{0}\big |_-[u_j]\right\rangle \\
	&\qquad \qquad -\delta\left\langle \left(\S_D^{0}\right)^{-1}[\chi_j^\diamond], \mathbf{v}\cdot \nabla \S_D^{0}\big |_+[u_j^\diamond] \right\rangle + O(\omega^3) \\
	&= \delta\left\langle u_j^\diamond, \mathbf{v}\cdot \nabla\S_D^{0}\big |_+[u_j^\diamond]\right\rangle -\delta\left\langle \left(\S_D^{0}\right)^{-1}[\chi_j^\diamond], \mathbf{v}\cdot \nabla \S_D^{0}\big |_+[u_j^\diamond] \right\rangle + O(\omega^3).
	\end{align*}
	Using the expansions in the dilute regime, we have to leading order in $\epsilon$,
	\begin{align*}
	\langle \Phi_j^\diamond, \A\P_1^{(1)} \Psi_j^\diamond\rangle &= \delta\left\langle u_j, \mathbf{v}\cdot\nabla\hat\S_D^{0}\big |_+[u_j] \right\rangle -\delta\left\langle \left(\hat\S_D^{0}\right)^{-1}[\chi_j^\diamond], \mathbf{v}\cdot \nabla \hat\S_D^{0}\big |_+[u_j] \right\rangle + O(\omega^3 + \omega^2\epsilon) \\
	&= \delta\left\langle u_j, (\mathbf{v}\cdot \nu) u_j \right\rangle - \delta\left\langle u_j, (\mathbf{v}\cdot \nu)u_j\right\rangle + O(\omega^3 + \omega^2\epsilon) \\
	&= O(\omega^3 + \omega^2\epsilon).
	\end{align*}
	Passing to higher orders in $\epsilon$ we have, after simplifications,
	\begin{align*}
	\langle \Phi_j^\diamond, \A\P_1^{(1)} \Psi_j^\diamond\rangle
	&= \delta\epsilon u_j^{(1)}\left\langle u_j, (\mathbf{v}\cdot \nu)u_j \right\rangle %+ \delta \left\langle u_j, \mathbf{v}\cdot \nabla \S_D^{(1)}\big|_+[u_j]\right\rangle
	+\delta\left\langle \left(\hat\S_D^{0}\right)^{-1}\S_D^{(1)}[u_j],  (\mathbf{v}\cdot \nu)u_j \right\rangle + O(\omega^3 + \omega^2\epsilon^2) \\
	&= \delta\epsilon \left(u_j^{(1)}\pm \frac{\textrm{Cap}_{B}}{4\pi l}\right)\left\langle u_j, (\mathbf{v}\cdot \nu)u_j \right\rangle + O(\omega^3 + \omega^2\epsilon^2),
	\end{align*}
	where $\pm$ is chosen as positive if $u_j$ is a monopole mode and negative if $u_j$ is a dipole mode. Due to the reflection symmetry of $D_1$ and $D_2$, we have $\left\langle u_j, (\mathbf{v}\cdot \nu)u_j \right\rangle = 0$, and hence
	$$\langle \Phi_j^\diamond, \A\P_1^{(1)} \Psi_j^\diamond\rangle = O(\omega^3 + \omega^2\epsilon^2).
	$$

	Next, we compute $\langle \Phi_j^\diamond, \P_2^{(1)}\A \Psi_j^\diamond\rangle$. Using  \eqref{eq:Sdilute}, \eqref{eq:Kdilute} and \eqref{eq:udilute} we can write
	$$\A = \A^{(0)} + \A^{(1)} + O(\epsilon^3), \qquad \Psi_j^\diamond = \Psi^{(0)} + \Psi^{(1)} + O(\epsilon),$$
	where the error terms are with respect to the norms in $\B(\L^2,\L\times\Hc)$ and $\L^2$, respectively. At $\omega = \omega_j^\diamond$, we have $\A^{(0)}\Psi^{(0)} = O(\omega^3)$, and hence
	$$\A\Psi_j^\diamond = \A^{(1)} \Psi^{(0)} + \A^{(0)}\Psi_1^{(1)} + O(\omega^3).$$
	We can see that
	$$\A^{(1)}\Psi^{(0)} = -\begin{pmatrix} \S_D^{(1)}[u_j] \\ \delta\K_D^{(1)}[u_j]\end{pmatrix}, \qquad \A^{(0)}\Psi^{(1)} = - \epsilon u_j^{(1)}\begin{pmatrix}\hat\S_D^\omega[u_j] \\ \delta\left(\frac{1}{2} + \hat\K_D^{\omega,*}\right)[u_j] \end{pmatrix}.$$
	Observe that $\S_D^{(1)}[u_j]$ and $\hat\S_D^\omega[u_j]$ are constant on $\p D$. Combining these results, we arrive at
	\begin{align*}
	\langle \Phi_j^\diamond, \P_2^{(1)}\A \Psi_j^\diamond\rangle &= -\delta \left\langle u_j, \K_D^{(1)}[u_j]\right\rangle  - \delta \left\langle\chi_j^\diamond, (2\tau-\p_T)\K_D^{(1)}[u_j]\right\rangle - \delta \epsilon u_j^{(1)}  \left\langle \chi_j^\diamond, (2\tau-\p_T) u_j \right\rangle +O(\omega^3+\epsilon\omega^2) \\
	&= - \delta \epsilon u_j^{(1)}  \left\langle \chi_j^\diamond, (2\tau-\p_T) u_j \right\rangle +O(\omega^3+\epsilon\omega^2) \\
	&= - \delta \epsilon u_j^{(1)}  \left\langle \chi_j^\diamond, 2\tau u_j \right\rangle +O(\omega^3+\epsilon\omega^2).
	\end{align*}
	Consequently, we obtain that
	$$\langle \Phi_j^\diamond,\B_0 \Psi_j^\diamond\rangle = - \delta \epsilon u_j^{(1)}  \left\langle \chi_j^\diamond, 2\tau u_j \right\rangle +O(\omega^3+\epsilon\omega^2).$$
	Observe that $\left\langle \chi_j^\diamond, u_j \right\rangle < 0$ and, in the case $D_1$ and $D_2$ are strictly convex, we have $\tau(x) > \tau_0 > 0$ for all $x\in D$, hence $\left\langle \chi_j^\diamond, 2\tau u_j \right\rangle < 0$. Combining this with the sign of $u_j^{(1)}$, the result follows. \qed

\subsection{Proof of \Cref{lem:PhiBPsi}} \label{app:lem2}
	We begin by computing the expansion of $\hat{V}$ in the dilute regime. Using $\psi_j$ as in the previous sections, that is, $\psi_j = (\hat \S_D^0)^{-1}[\chi_{D_j}]$, we have
	$$\psi_j = \sqrt{\epsilon \mathrm{Cap}_B}\psi_{D_j}^1, \quad j=1,2.$$
	Then
	\begin{align} \label{eq:lem1}
	\left(V_j\right)_{m,n} &= -\int_{\p D_j}\int_{\p D_j}G^\omega(x-d\mathbf{v},y) \xi_{D_j}^m(y)\psi_{D_j}^n(x)\dx \sigma(x)\dx \sigma(y) \nonumber \\
	&= -\int_{\p D_j}\int_{\p D_j}\big(G^\omega(d\mathbf{v},0) + (x-y)\cdot\nabla_x G^\omega(d\mathbf{v},0) \big)\xi_{D_j}^m(z)\psi_{D_j}^n(y)\dx \sigma(z)\dx \sigma(y) + O(\epsilon^3) \nonumber\\
	&= -\sqrt{\epsilon\mathrm{Cap}_B}G^\omega(d\mathbf{v},0)\delta_{n,1}\int_{\p D_j} \xi_{D_j}^m\dx\sigma + O(\epsilon^3)\nonumber \\
	&= \frac{\epsilon\mathrm{Cap}_B}{4\pi d}\delta_{m,1}\delta_{n,1} + O(\epsilon^3 + \omega\epsilon),
	\end{align}
	where we have used symmetry in the integration together the orthogonality relation
	$$
	\int_{\p D_j}\psi_{D_j}^m \dx \sigma = \sqrt{\epsilon \mathrm{Cap}_B}\delta_{m,1}.
	$$
	\begin{comment}
	Moreover,
	\begin{align*}
	\left(W_j\right)_{m,n} &= \int_{\p D_j} \D_{D_j}^0\S_{D_j}^0[\psi_{D_j}^m](y-d\mathbf{v})\psi_{D_j}^n(y)\dx \sigma (y) + O(\omega^2) \\
	&= \frac{\sqrt{\epsilon\mathrm{Cap}_B}}{4\pi d}\delta_{n,1}\left\langle \mathbf{v}\cdot \nu, \S_{D_j}^0[\psi_{D_j}^m]\right\rangle + O(\epsilon^2 + \omega^2).
	\end{align*}
	\end{comment}
	Observe that at $m = 1$ we have $\D_{D_j}^0[\chi_{D_j}] = 0$ outside $D_j$, and so
	\begin{equation}\label{eq:lem2}
	\left(W_j\right)_{1,n} = O(\omega^2)
	\end{equation}
	for all $n$. Recall the expansion, from the proof of \Cref{lem:C},
	$$\Psi_j^\diamond = \Psi^{(0)} + \Psi^{(1)} + O(\epsilon),\qquad  \Psi^{(0)} = \begin{pmatrix}u_j \\ u_j \end{pmatrix}, \qquad \Psi^{(1)} = \begin{pmatrix}0 \\ \epsilon u_j^{(1)} u_j \end{pmatrix},$$
	where, at $\omega = \omega_j^\diamond$, $\hat\A\Psi^{(0)} = O(\omega^3)$. Also, recall that
	$$\Phi_j^\diamond = \begin{pmatrix} -\delta u_j^\diamond \\
	\chi_j^\diamond \end{pmatrix} .$$
	Then we can compute
	$$\left\langle \Phi_j^\diamond, \hat\P_2\hat\A\hat\P_1 \Psi^{(0)}\right\rangle = O(\omega^3).$$
	Turning to higher orders of $\Psi_j^\diamond$, we have
	\begin{align*}
	\left\langle \Phi_j^\diamond, \left(\hat\P_2\hat\A\hat\P_1 - \hat{\A}\right)\Psi^{(1)}\right\rangle &= -\delta\epsilon u_j^{(1)}\bigg(\left\langle \chi_j^\diamond, W\left(\frac{1}{2} + \hat\K_D^{0,*}\right)[V^{-1}u_j]\right\rangle - \left\langle \chi_j^\diamond, \left(\frac{1}{2} + \hat\K_D^{0,*}\right)[u_j]\right\rangle \\
	&\qquad \quad  +\left\langle u_j^\diamond, V^{*}\S_D^0[V^{-1}u_j] \right\rangle - \left\langle u_j, \S_D^0[u_j]  \right\rangle  \bigg) + O(\omega^3).
	\end{align*}
	 From \eqref{eq:lem2}, it holds that
	$$\left\langle \chi_j^\diamond, W\left(\frac{1}{2} + \hat\K_D^{0,*}\right)[V^{-1}u_j]\right\rangle = O(\omega^2).$$
	Moreover, \eqref{eq:lem1} yields
	$$\left\langle u_j^\diamond, V^{*}\S_D^0[V^{-1}u_j] \right\rangle - \left\langle u_j, \S_D^0[u_j]  \right\rangle = O(\epsilon^2 + \omega^2).$$
	Finally, since $\left\langle \chi_j^\diamond, \left(\frac{1}{2} + \hat\K_D^{0,*}\right)[u_j]\right\rangle = \left\langle \chi_j^\diamond, u_j\right\rangle = \epsilon\mathrm{Cap}_B $, we have
	$$\left\langle \Phi_j^\diamond, \B_d\Psi_j^\diamond\right\rangle = \delta\epsilon^2\mathrm{Cap}_B u_j^{(1)} + O(\omega^3 + \epsilon^3\omega^2).$$
	Since the leading order is independent of $d$, the conclusion follows. \qed

\section{Proof of \Cref{prop:N=1}} \label{app:prop}

We will restrict the analysis to the equation
\begin{equation}\label{eq:N=1main}
\frac{1}{2\pi} \int_{Y^*}\Big(\eta_1\left(1- e^{\iu \theta_\alpha}\right)+\eta_2\left(1+ e^{\iu \theta_\alpha}\right)\Big)\dx \alpha = 0,
\end{equation}
since the proof of the equation in \eqref{eq:N=1} with the other sign is similar. Define
$$\lambda = \frac{\omega^2|D_1|}{\delta}, \qquad \lambda_1^\alpha = C_{11}^\alpha - |C_{12}^\alpha|, \qquad \lambda_2^\alpha = C_{11}^\alpha + |C_{12}^\alpha|.$$
Then, as $\delta \rightarrow 0$,
\begin{equation}\label{eq:int_delta}
\frac{1}{2\pi} \int_{Y^*}\Big(\eta_1\left(1- e^{\iu \theta_\alpha}\right)+\eta_2\left(1+ e^{\iu \theta_\alpha}\right)\Big)\dx \alpha =
\frac{1}{\pi} \int_{Y^*}\frac{\lambda\left(C_{11}^\alpha + \mathrm{Re}(C_{12}^\alpha)\right) - \lambda_1^\alpha\lambda_2^\alpha}{(\lambda-\lambda_1^\alpha)(\lambda-\lambda_2^\alpha)} \dx \alpha + O(\delta^{1/2}),
\end{equation}
where the imaginary part vanishes due to symmetry. Observe that for $\omega$ inside the band gap, we have $\lambda-\lambda_1^\alpha > 0$ and $\lambda-\lambda_2^\alpha < 0$. Define
\begin{equation*}
f(\alpha) = \lambda\left(C_{11}^\alpha + \mathrm{Re}(C_{12}^\alpha)\right) - \lambda_1^\alpha\lambda_2^\alpha.
\end{equation*}
We will now study the two cases $l_0 < 1/2$ and $l_0 >1/2$ separately. We will show that the right-hand side of \eqref{eq:N=1main} is always positive in the first case, while in the second case it has a sign depending on $\lambda$. We will do so by splitting the integral into two parts, one with $\alpha$ close to $0$ and one with $\alpha$ bounded away from $0$, and show that the first part is negligible.

\subsection{Case $l_0 < 1/2$} \label{app:propPart1}
In the dilute regime, as $\epsilon \rightarrow 0$, it follows from \Cref{lem:cap_estim_quasi} that the width of the band gap scales as $O(\epsilon^2)$. Moreover, if $\omega$ is inside the band gap then we are able to write that $$\lambda = \epsilon\mathrm{Cap}_B + \epsilon^2(\mathrm{Cap}_B)^2\lambda_0 + O(\epsilon^3)$$ for some $\lambda_0\in\R$. From the expansions of the capacitance coefficients in \Cref{lem:cap_estim_quasi}, and the fact that $\lambda_1^\alpha$ (resp. $\lambda_2^\alpha$) attains its maximum (resp. minimum) at $\alpha = \pi/L$, we have the following bounds on $\lambda_0$:
\begin{equation}\label{eq:lambda0}
-\frac{1}{4\pi L} \sum_{m \neq 0} \frac{e^{\iu \alpha m L}}{|m|} - \frac{1}{4\pi L} \sum_{m = -\infty}^\infty \frac{e^{\iu \alpha m L}}{|m+l_0|} < \lambda_0 < -\frac{1}{4\pi L} \sum_{m \neq 0} \frac{e^{\iu \alpha m L}}{|m|} + \frac{1}{4\pi L} \sum_{m = -\infty}^\infty \frac{e^{\iu \alpha m L}}{|m+l_0|}.
\end{equation}
We fix constants $C > 0$, $p\in \N$. Then, for $\alpha$ such that $|\alpha| > C\epsilon^p$, $f(\alpha)$ can be expanded in the dilute regime as
\begin{align}
f(\alpha) =& \epsilon^3(\mathrm{Cap}_B)^3\left(\lambda_0 - \frac{1}{4\pi l} + \frac{1}{4\pi L}\sum_{m \neq 0} \frac{\cos(m\alpha L)}{|m|} - \frac{1}{4\pi L}\sum_{m \neq 0} \frac{\cos(m\alpha L)}{|m+l_0|}\right) + o(\epsilon^3)\nonumber \\
=&\epsilon^3(\mathrm{Cap}_B)^3\left(\lambda_0 - \frac{1}{4\pi l} + \frac{1}{4\pi L}\sum_{m = 1}^\infty \cos(m\alpha L)\left( \frac{2}{m} - \frac{1}{m+l_0} - \frac{1}{m-l_0}\right)  \right)+ o(\epsilon^3).  \label{eq:fasym}
\end{align}
Define $g(\alpha)$ as
$$g(\alpha) = \sum_{m = 1}^\infty e^{\iu m\alpha L}\left( \frac{2}{m} - \frac{1}{m+l_0} - \frac{1}{m-l_0}\right).$$
We can rewrite $g$ as
\begin{align*}
g(\alpha) &= e^{\iu \alpha L}\sum_{m = 0}^\infty e^{\iu m\alpha L}\left( \frac{2}{m+1} - \frac{1}{m+1+l_0} - \frac{1}{m+1-l_0}\right) \\
& = e^{\iu \alpha L}\big( 2\Phi(e^{\iu \alpha L}, 1, 1) - \Phi(e^{\iu \alpha L}, 1, 1+l_0) - \Phi(e^{\iu \alpha L}, 1, 1-l_0)\big).
\end{align*}
Here, $\Phi(z,s,a)$ denotes Lerch's transcendent function, defined by the power series
\begin{equation*}
\Phi(z,s,a) = \sum_{m=0}^\infty \frac{z^m}{(a+m)^s},
\end{equation*}
for $z\in \mathbb{C}$ where this series converges, extended by analytic continuation elsewhere (for details on this function we refer, for example, to  \cite{bateman1953higher}). For arguments in the regime $\text{Re}(s) > 0, \text{Re}(a) > 0$ and $z\in \mathbb{C}\setminus [1,\infty)$, this function admits an integral representation as
\begin{equation*}
\Phi(z,s,a) = \frac{1}{\Gamma(s)}\int_0^\infty \frac{t^{s-1}e^{-at}}{1-ze^{-t}}\dx t,
\end{equation*}
where $\Gamma$ is the Gamma function. From this, we have a representation of $g(\alpha), \alpha\neq 0,$ as
\begin{align*}
g(\alpha) &= \int_0^\infty \frac{2e^{-t} - e^{-(1+l_0)t} - e^{-(1-l_0)t}}{1-e^{\iu \alpha L}e^{-t}}\dx t \\
&= \int_0^\infty \frac{\left(\cosh(l_0t) - 1\right)\left(e^{-t} - \cos(\alpha L)\right)}{\cosh(t)-\cos(\alpha L)}\dx t.
\end{align*}
From \eqref{eq:fasym}, using the bounds on $\lambda_0$ from \eqref{eq:lambda0} and for $\alpha$ such that $|\alpha| > C\epsilon^p$, we have
\begin{align*}
f(\alpha) &< \frac{\epsilon^3(\mathrm{Cap}_B)^3}{4\pi L}\left(\sum_{m = 1}^\infty \left(\cos(m\alpha L)-(-1)^m\right)\left( \frac{2}{m} - \frac{1}{m+l_0} - \frac{1}{m-l_0}\right)  \right)+ o(\epsilon^3)\nonumber \\
&= \frac{\epsilon^3(\mathrm{Cap}_B)^3}{4\pi L}\left(\mathrm{Re}\big(g(\alpha)\big) - g(\pi/L) \right) + o(\epsilon^3)\nonumber \\
&= \frac{\epsilon^3(\mathrm{Cap}_B)^3}{4\pi L}\int_0^\infty\left(\cosh(l_0t) - 1\right)\sinh(t)\left( \frac{1}{\cosh(t)+1} - \frac{1}{\cosh(t)-\cos(\alpha L)}\right) + o(\epsilon^3) \nonumber\\
&= A_1(\alpha)\epsilon^3 + o(\epsilon^3)
\end{align*}
for some $A_1(\alpha) \leq 0$ independent of $\epsilon$, with $A_1(\alpha) = 0$ precisely when $\alpha = \pi/L$. It follows that
\begin{equation} \label{eq:I2}
\frac{1}{\pi} \int_{Y^*\setminus[-C\epsilon^p,C\epsilon^p]}\frac{\lambda\left(C_{11}^\alpha + \mathrm{Re}(C_{12}^\alpha)\right) - \lambda_1^\alpha\lambda_2^\alpha}{(\lambda-\lambda_1^\alpha)(\lambda-\lambda_2^\alpha)} \dx \alpha  =  \frac{A_2}{\epsilon} + o(\epsilon^{-1})
\end{equation}
for some constant $A_2 > 0$. From the scaling property \eqref{eq:capcoeff_scale}, we know that $|f(\alpha)| < \epsilon^2 K_1$ for some $K_1>0$ independent on $\alpha$.  The minimum of $\big|(\lambda-\lambda_1^\alpha)(\lambda-\lambda_2^\alpha)\big|$ is attained at $\pi/L$ and, from \Cref{lem:cap_estim_quasi}, we have $\big|(\lambda-\lambda_1^\alpha)(\lambda-\lambda_2^\alpha)\big| > K_2 \epsilon^4$. Therefore, we have
\begin{align*}
\left| \frac{1}{\pi} \int_{[-C\epsilon^p,C\epsilon^p]}\frac{\lambda\left(C_{11}^\alpha + \mathrm{Re}(C_{12}^\alpha)\right) - \lambda_1^\alpha\lambda_2^\alpha}{(\lambda-\lambda_1^\alpha)(\lambda-\lambda_2^\alpha)} \dx \alpha \right|
&< A_3\epsilon^{p-2},
\end{align*}
for some constant $A_3$. Choosing $p > 2$, and combining this with \eqref{eq:I2}, we find that
$$\frac{1}{\pi} \int_{Y^*}\frac{\lambda\left(C_{11}^\alpha + \mathrm{Re}(C_{12}^\alpha)\right) - \lambda_1^\alpha\lambda_2^\alpha}{(\lambda-\lambda_1^\alpha)(\lambda-\lambda_2^\alpha)} \dx \alpha  > 0
$$
for $\epsilon$ small enough. Therefore, when $l_0 < 1/2$, by \eqref{eq:int_delta} we find that, for $\lambda$ sufficiently close to $\lambda_1^{\pi/L}$, we have
\begin{equation*}
\frac{1}{2\pi} \int_{Y^*}\Big(\eta_1\left(1- e^{\iu \theta_\alpha}\right)+\eta_2\left(1+ e^{\iu \theta_\alpha}\right)\Big)\dx \alpha > 0,
\end{equation*}
when $\epsilon$ and $\delta$ are small enough.

\subsection{Case $l_0 > 1/2$}
We will show that \eqref{eq:N=1main} has a solution. We denote the left-hand side by
$$
I(\lambda) := \frac{1}{2\pi} \int_{Y^*}\Big(\eta_1\left(1- e^{\iu \theta_\alpha}\right)+\eta_2\left(1+ e^{\iu \theta_\alpha}\right)\Big)\dx \alpha.
$$
From \Cref{lem:cap_estim_quasi}, we find that for $\epsilon$ small enough, $C_{12}^{\pi/L}>0$ in the case $l_0 > 1/2$. Hence $e^{\iu \theta_{\pi/L}} = 1$, so $I(\lambda) \rightarrow -\infty$ as $\lambda \rightarrow \lambda_2^{\pi/L}$. Next, we will show that $I(\lambda)$ is positive for $\lambda$ sufficiently close to $\lambda_1^{\pi/L}$.

Since $C_{12}^{\pi/L}$ is positive, we now have the following bounds for $\lambda_0$:
\begin{equation*}
-\frac{1}{4\pi L} \sum_{m \neq 0} \frac{e^{\iu \alpha m L}}{|m|} + \frac{1}{4\pi L} \sum_{m = -\infty}^\infty \frac{e^{\iu \alpha m L}}{|m+l_0|} < \lambda_0 < -\frac{1}{4\pi L} \sum_{m \neq 0} \frac{e^{\iu \alpha m L}}{|m|} - \frac{1}{4\pi L} \sum_{m = -\infty}^\infty \frac{e^{\iu \alpha m L}}{|m+l_0|}.
\end{equation*}
Fix some small $\kappa >0$ and choose $\lambda_0$ as
$$\lambda_0 = \kappa - \frac{1}{4\pi L} \sum_{m \neq 0} \frac{e^{\iu \alpha m L}}{|m|} + \frac{1}{4\pi L} \sum_{m = -\infty}^\infty \frac{e^{\iu \alpha m L}}{|m+l_0|}.$$
Observe that $\kappa \rightarrow 0$ corresponds to $\lambda \rightarrow \lambda_1^{\pi/L}$. Using \eqref{eq:fasym} and following the same subsequent steps, we find that
\begin{align*}
f(\alpha) &= \epsilon^3\left( \mathrm{Cap}_B)^3\kappa + A_1(\alpha) \right) + o(\epsilon^3).
\end{align*}
Then, analogously to \eqref{eq:I2}, we have
$$
\frac{1}{\pi} \int_{Y^*\setminus[-C\epsilon^p,C\epsilon^p]}\frac{\lambda\left(C_{11}^\alpha + \mathrm{Re}(C_{12}^\alpha)\right) - \lambda_1^\alpha\lambda_2^\alpha}{(\lambda-\lambda_1^\alpha)(\lambda-\lambda_2^\alpha)} \dx \alpha  =  \frac{A_2 + A_4\kappa}{\epsilon} + o(\epsilon^{-1}),
$$
where, again, $A_2$ is a constant $A_2 >0$ and $A_4$ is a constant $A_4 < 0$. Thus, for $\kappa$ small enough, we have that $A_2 + A_4\kappa > 0$, so we can proceed as in Section~\ref{app:propPart1} to show that
$$
I(\lambda) > 0,
$$
for $\lambda$ sufficiently close to $\lambda_1^{\pi/L}$ and for small enough $\epsilon$ and $\delta$. This, combined with the fact that $I(\lambda) < 0$ for $\lambda$ sufficiently close to $\lambda_2^{\pi/L}$, allows us to conclude that $I(\hat\lambda) = 0$ for some $\lambda_1^{\pi/L} < \hat\lambda < \lambda_2^{\pi/L}$.

In order to show that this solution $\hat\lambda$ is unique, we show that $I(\lambda)$ is strictly monotonic for $\lambda_1^{\pi/L} < \nolinebreak \lambda < \lambda_2^{\pi/L}$. Differentiating \eqref{eq:int_delta} gives
\begin{equation*}\label{eq:int_deriv}
I'(\lambda) =
\frac{1}{\pi} \int_{Y^*}\tfrac{\left(C_{11}^\alpha + \mathrm{Re}(C_{12}^\alpha)\right)(\lambda-\lambda_1^\alpha)(\lambda-\lambda_2^\alpha) - \left(\lambda\left(C_{11}^\alpha + \mathrm{Re}(C_{12}^\alpha)\right)- \lambda_1^\alpha\lambda_2^\alpha\right)\left(2\lambda-\lambda_1^\alpha-\lambda_2^\alpha\right) }{(\lambda-\lambda_1^\alpha)^2(\lambda-\lambda_2^\alpha)^2} \dx \alpha + O(\delta^{1/2}).
\end{equation*}
Then we have that
\begin{align}
&\left(C_{11}^\alpha + \mathrm{Re}(C_{12}^\alpha)\right) (\lambda-\lambda_1^\alpha)(\lambda-\lambda_2^\alpha) - \left(\lambda\left(C_{11}^\alpha + \mathrm{Re}(C_{12}^\alpha)\right)- \lambda_1^\alpha\lambda_2^\alpha\right)\left(2\lambda-\lambda_1^\alpha-\lambda_2^\alpha\right) \nonumber \\
&=\left(C_{11}^\alpha + \mathrm{Re}(C_{12}^\alpha)\right)(-\lambda^2+\lambda_1^\alpha\lambda_2^\alpha) + \lambda_1^\alpha\lambda_2^\alpha\left(2\lambda-\lambda_1^\alpha-\lambda_2^\alpha\right)\nonumber\\
&\leq \begin{cases}
\lambda_2^\alpha\left(-\lambda^2+\lambda_1^\alpha\lambda_2^\alpha + \lambda_1^\alpha\left(2\lambda-\lambda_1^\alpha-\lambda_2^\alpha\right)\right), & \quad \text{if} \ \lambda^2\leq \lambda_1^\alpha-\lambda_2^\alpha, \\[0.3em]
\lambda_1^\alpha\left(-\lambda^2+\lambda_1^\alpha\lambda_2^\alpha + \lambda_2^\alpha\left(2\lambda-\lambda_1^\alpha-\lambda_2^\alpha\right)\right), & \quad \text{if} \ \lambda^2> \lambda_1^\alpha-\lambda_2^\alpha,
\end{cases}\nonumber\\
&=\begin{cases}
-\lambda_2^\alpha\left(\lambda-\lambda_1^\alpha\right)^2, & \quad \text{if} \ \lambda^2\leq \lambda_1^\alpha-\lambda_2^\alpha,\\[0.3em]
-\lambda_1^\alpha\left(\lambda-\lambda_2^\alpha\right)^2, & \quad \text{if} \ \lambda^2> \lambda_1^\alpha-\lambda_2^\alpha.
\end{cases} \label{eq:h_bounds}
\end{align}
Using the bounds \eqref{eq:h_bounds} we have that if $\lambda_1^{\pi/L} < \lambda < \lambda_2^{\pi/L}$ then $I'(\lambda)<0$, provided $\delta$ is sufficiently small. Therefore, if $l_0>1/2$ then \eqref{eq:N=1main} has a unique solution, when $\epsilon$ and $\delta$ are small enough.

\bibliographystyle{abbrv}
\bibliography{edgemode}{}

\end{document}